\documentclass{amsart}

\usepackage{amssymb}
\usepackage{xspace}
\usepackage{xcolor}
\usepackage[pdftex,backref]{hyperref}
\usepackage[all]{xy}
\usepackage[utf8x]{inputenc}
\usepackage{paralist}
\usepackage{tikz}

\newcommand{\ie}{i.e.,\xspace}

 \newcommand{\jq}{\ensuremath{\mathsf{j}}}
\newcommand{\C}{\mathbb{C}}
\newcommand{\R}{\mathbb{R}}
\newcommand{\Z}{\mathbb{Z}}
\newcommand{\Sym}{\mathrm{Sym}}
\newcommand{\Xsp}{\bar{X}_{(P)}}
\newcommand{\Xspb}{\bar{X}_{(\bar{P})}}

\newcommand{\V}{\mathbf{v}}
\newcommand{\fami}{\mathfrak{F}}
\newcommand{\EGF}[2][\fami]{\mathit{E}_{#1}(#2)}
\newcommand{\BGF}[2][\fami]{\mathit{B}_{#1}(#2)}
\newcommand{\Xfami}[1][\fami]{X_{#1}}
\newcommand{\Xsubgr}[1]{X_{(#1)}}
\newcommand{\CC}[2][*]{\mathbf{#2}_{#1}}

\newcommand{\SL}[2][2]{\ensuremath{{SL}_{#1}(#2)}}
\newcommand{\PSL}[2][2]{\ensuremath{{{PSL}_{#1}(#2)}}}

\newcommand{\hyp}[1][3]{\ensuremath{\mathbb{H}^{#1}}}
\newcommand{\hypc}[1][3]{\overline{\mathbb{H}}^{#1}}
\newcommand{\Mo}{\ensuremath{M_{0}}}
\newcommand{\Mp}{\ensuremath{M_{+}}}
\newcommand{\dsubgr}{\ensuremath{\Gamma}}
\newcommand{\hg}[2][1]{\ensuremath{\pi_{#1}(#2)}}
\newcommand{\homo}[1][\bullet]{\ensuremath{H_{#1}}}
 \newcommand{\EMc}[2]{\ensuremath{\mathbf{K}(#1,#2)}}
\newcommand{\pB}[1]{\ensuremath{\mathcal{P}(#1)}}
\newcommand{\Bg}[1]{\mathcal{B}(#1)}
\newcommand{\EpB}[1]{\widehat{\mathcal{P}}(#1)}
\newcommand{\EBg}[1]{\widehat{\mathcal{B}}(#1)}
 \newcommand{\antp}[1]{\ensuremath{\wedge_\Z^2(#1)}}
 \newcommand{\Asym}[2]{\ensuremath{{#1}\wedge{#2}}}
\newcommand{\fcPSL}[1]{[#1]_{\mathrm{PSL}}} 
 \newcommand{\suce}[2][n]{\ensuremath{{#2_1,\ldots,#2_{#1}}}}
\newcommand{\sets}[2]{\ensuremath{\{\,{#1}\mid {#2}\,\}}}
\newcommand{\dH}[2]{\langle\,#1,#2\,\rangle}

\newcommand{\co}{\colon\thinspace}

\DeclareMathOperator{\Log}{Log}
\DeclareMathOperator{\Arg}{Arg}
\DeclareMathOperator{\Vol}{Vol}
\DeclareMathOperator{\CS}{CS}

\newtheorem{theorem}{Theorem}[section]
\newtheorem{lemma}[theorem]{Lemma}
\newtheorem{proposition}[theorem]{Proposition}
\newtheorem{corollary}[theorem]{Corollary}

\theoremstyle{definition}

\theoremstyle{definition}
\newtheorem{remark}[theorem]{Remark}

\numberwithin{equation}{section}

\begin{document}

\setcounter{tocdepth}{1}
\title{Invariants of hyperbolic $3$-manifolds in relative group homology}

%    Remove any unused author tags.

%    author one information
\author{Jos\'e Antonio Arciniega-Nev\'arez}
\address{División de Ingenierías, Campus Guanajuato, Universidad de Guanajuato\\
Av. Juárez No. 77, Zona Centro, Guanajuato, Gto., México, C.P. 36000}
\email{ja.arciniega@ugto.mx}
%\thanks{}

%    author two information
\author{Jos\'e Luis Cisneros-Molina}
\address{Instituto de Matem\'aticas, Unidad Cuernavaca\\ Universidad Nacional
Autonoma
de M\'exico\\ Avenida  Universidad sin n\'umero, Colonia Lomas de Chamilpa\\
Cuernavaca,  Morelos, M\'exico}
% \curraddr{Department of Mathematics\\ Columbia University\\ 2990 Broadway\\ New
% York NY, 10027 }
\email{jlcisneros@im.unam.mx}
%\thanks{}

%    For articles to be published after 1 January 2010, you may use
%    the following version:
%\subjclass[2010]{Primary }


%\date{}

%\dedicatory{}

\begin{abstract}
Let $M$ be a complete oriented hyperbolic $3$--manifold of finite volume. Using
classifying spaces for families of subgroups we construct a class $\beta_P(M)$
in the Adamson relative homology group $H_3([\PSL{\C}:\bar{P}];\Z)$, where
$\bar{P}$ is the subgroup of parabolic transformations which fix $\infty$ in the
Riemann sphere. We also prove that the classes $F(M)$ in the Takasu relative homology groups
$H_3(\PSL{\C},\bar{P};\Z)$ constructed by Zickert in \cite{Zickert:VCSIR}, which are not 
well-defined and depend of a choice of decorations by horospheres are all mapped to $\beta_P(M)$
via a canonical comparison homomorphism $H_3(\PSL{\C},\bar{P};\Z)\to H_3([\PSL{\C}:\bar{P}];\Z)$.
To do this, we simplify the construction of the classes $F(M)$ using a simpler complex which
computes $H_3(\PSL{\C},\bar{P};\Z)$, getting a simple simplicial formula for $F(M)$, which
in turn gives a simpler and more efficient formula to compute the volume and Chern--Simons invariant
than the one given in \cite{Zickert:VCSIR}. The constructions can be extended for any boundary-parabolic 
$\PSL{\C}$-representation.
\end{abstract}

\maketitle

\tableofcontents

\section{Introduction}

Let $G$ be a discrete group and $H$ a subgroup of $G$. In the literature there are two 
versions of \emph{relative homology groups} for the pair $(G,H)$: the first one, denoted by 
$H_n([G:H];\Z)$, was defined by Adamson in \cite{Adamson:CTNNSNNF}, and the second one,
denoted by $H_n(G,H;\Z)$ was defined by Takasu in \cite{Takasu:RHRCTG}. 
In general they do not coincide \cite{Arciniega-Cisneros:CRGCH}.

Given a compact oriented hyperbolic $3$--manifold $M$ there is a canonical
representation $\bar{\rho}\co\hg{M}=\dsubgr\to\PSL{\C}$. To this representation corresponds a
map $B\bar{\rho}\co M\to B\PSL{\C}$ where $B\PSL{\C}$ is the classifying space
of $\PSL{\C}$ considered as a discrete group. There is a
well-known invariant $\fcPSL{M}$ of $M$ in the group $H_3(\PSL{\C};\Z)$ given by
the image of the
fundamental class of $M$ under the homomorphism induced in homology by
$B\bar{\rho}$.

In the case when $M$ has cusps, a construction of $\fcPSL{M}$ was first
described by Neumann in \cite[\S14]{Neumann:EBGCCSC}.
Consider the Takasu relative group
$H_3(\PSL{\C},\bar{P};\Z)$ where $\bar{P}$ is the
image in $\PSL{\C}$ of the subgroup $P$ of $\SL{\C}$ of matrices of the form 
$\bigl(\begin{smallmatrix}
\pm1 & b\\                                                                      
0  & \pm1
\end{smallmatrix}\bigr)$. 
The long exact sequence of Takasu relative group homology \cite[Proposition~2.1]{Takasu:RHRCTG}
simplifies to short exact sequences (see \cite[Proposition~14.3]{Neumann:EBGCCSC}), in particular
\begin{equation*}
 0\to H_3(\PSL{\C};\Z)\to H_3(\PSL{\C},\bar{P};\Z)\to H_2(\bar{P};\Z)\to0,
\end{equation*}
and there is a natural splitting
\begin{equation}\label{eq:Psi}
\Psi\colon H_3(\PSL{\C},\bar{P};\Z)\to H_3(\PSL{\C};\Z).
\end{equation}
The hyperbolic $3$-manifold $M$ can be compactified to a manifold $\bar{M}$ with $\partial\bar{M}$ consisting of tori. The fundamental class $[\bar{M},\partial\bar{M}]\in H_3(\bar{M},\partial\bar{M};\Z)$
determines a class in $H_3(\PSL{\C},\bar{P};\Z)$, and its image under the homomorphism $\Psi$ gives $\fcPSL{M}$.
This was proved in detail by Zickert in \cite{Zickert:VCSIR}, where a complex of truncated simplices
with labelings on its vertices by elements of $\PSL{\C}$ is used to compute Takasu relative homology groups $H_n(\PSL{\C},\bar{P};\Z)$, and using
the fact that  $H_3(\PSL{\C};\Z)$ is isomorphic to the extended Bloch group $\EBg{\C}$ defined by Neumann~\cite[Theorem~2.6]{Neumann:EBGCCSC}
the homomorphism \eqref{eq:Psi} is constructed. 
Suppose that the complete oriented non-compact hyperbolic $3$--manifold of finite volume $M$ (\ie with
cusps) is triangulated. From the triangulation and using the complex of truncated simplices, Zickert constructs a class $F(M)$ in $H_3(\PSL{\C},\bar{P};\Z)$,
which is not well-defined, it depends on choices of horoballs at the cusps, but the image $\fcPSL{M}=\Psi(F(M)\in H_3(\PSL{\C};\Z)$ is independent of these choices \cite[Theorem~6.10]{Zickert:VCSIR}.

In the present article we generalize the construction given by Cisneros-Molina and Jones 
in \cite{Cisneros-Jones:Bloch} using classifying spaces for families of isotropy subgroups 
to construct a well-defined class $\beta_P(M)$ in the Adamson relative homology group $H_3([\PSL{\C}:\bar{P}];\Z)$.
A family $\fami$ of subgroups of a discrete group $G$ is a set of subgroups of
$G$ which is closed under conjugation and subgroups. A classifying space
$\EGF{G}$ for the family $\fami$ is a terminal object in the category of
$G$--sets with isotropy subgroups in the family $\fami$. We consider the case
$G=\PSL{\C}$ and the family $\fami(\bar{P})$ of subgroups generated by the
subgroup $\bar{P}$. One can compute Adamson relative group homology as
$H_3([\PSL{\C}:\bar{P}];\Z)=H_3(\BGF[\fami(\bar{P})]{G};\Z)$ where 
$\BGF[\fami(\bar{P})]{G}$ is the orbit space of $\EGF[\fami(\bar{P})]{G}$.
Using general properties of classifying spaces for families of isotropy
subgroups and the canonical representation $\bar{\rho}\co\hg{M}\to\PSL{\C}$ we
construct a canonical map $\widehat{\psi}_P\co\widehat{M}\to
\BGF[\fami(\bar{P})]{G}$ where $\widehat{M}$ is the end compactification of $M$.
We have that $H_3(\widehat{M};\Z)\cong\Z$ and the image of the generator under
the homomorphism $(\widehat{\psi}_P)_*\co H_3(\widehat{M};\Z)\to
H_3(\BGF[\fami(\bar{P})]{G};\Z)=H_3([\PSL{\C}:\bar{P}];\Z)$ gives a well defined element $\beta_P(M)\in
H_3([\PSL{\C}:\bar{P}];\Z)$. 

There is a canonical comparison homomorphism between Takasu and Adamson relative group homologies 
\begin{equation}\label{eq:comp.hom}
 \lambda_3\colon H_3(\PSL{\C},\bar{P};\Z)\to H_3([\PSL{\C}:\bar{P}];\Z).
\end{equation}
We prove that under $\lambda_3$ the class $F(M)$ is mapped to $\beta_P(M)$, independently of
the choices made to define $F(M)$, in other words, the preimage $\lambda_3^{-1}(\beta_P(M))$ consists of the classes $F(M)$ 
varying the choices of horoballs. To prove this, first we simplify the construction of $F(M)$ in \cite{Zickert:VCSIR},
instead of using the complex of truncated simplices with labelings on its vertices by elements of $\PSL{\C}$, we use a simpler complex, called the $h^K_H$-subcomplex, which computes the groups
$H_n(\PSL{\C},\bar{P};\Z)$ \cite[Theorem~6.4]{Arciniega-Cisneros:CRGCH} and we prove that the two complexes are isomorphic.
Also following \cite{Dupont-Zickert:DFCCSC} we write explicitly homomorphism \eqref{eq:Psi} in this context. 

The construction of $\beta_P(M)$ has some advantages:
\begin{enumerate}
 \item It does not require a triangulation of $M$.
 \item It shows that the class $\beta_P(M)$ obtained in the Adamson
relative homology group $H_3([\PSL{\C}:\bar{P}];\Z)$ is a well-defined invariant of $M$
which lifts the classical Bloch invariant $\beta(M)$.
 \item Choosing an ideal triangulation of $M$ one can give a explicit formula
for the class $\beta_P(M)$.
\item Using $\beta_P(M)$ is possible to define the fundamental class $\fcPSL{M}$ without a triangulation of $M$
by choosing a preimage $F(M)\in\lambda_3^{-1}(\beta_P(M))$ and defining $\fcPSL{M}=\Psi(F(M))$.
\end{enumerate}

Our construction also works in the general context of $(G,H)$--representations
of tame $n$--manifolds considered in \cite{Zickert:VCSIR}. In
Section~\ref{sec:G.H.rep} we prove that a $G$--representation mapping peripheral subgroups 
to conjugates of a fixed subgroup $H$ gives a well-defined class in the
Adamson relative homology group $H_n([G:H];\Z)$. Since in general Adamson and 
Takasu relative homology groups do not necessarily coincide, the
class in $H_n(G,H;\Z)$ constructed by Zickert in \cite[\S5]{Zickert:VCSIR}
depends of a choice of decoration and different classes given by different
choices are all mapped to the class in $H_n([G:H];\Z)$ by a canonical
homomorphism $H_n(G,H;\Z)\to H_n([G:H];\Z)$. So in general it is
more appropriate to use Adamson relative group homology than Takasu relative
group homology because we obtain classes independent of choice.

The article is organized as follows. In Section~\ref{sec:rgh} we recall the definitions and basic properties of Adamson and Takasu relative group homology theories, in particular, 
the topological definition of Adamson relative group homology using classifying spaces for families of subgroups, and the so-called $h^K_H$-subcomplex  
and conditions for which it computes Takasu relative group homology \cite[Theorem~6.4]{Arciniega-Cisneros:CRGCH}. In Section~\ref{sec:GUPB} we give explicit models
for some homogeneous spaces of $\SL{\C}$ and $\PSL{\C}$, explicit equivariant maps between them and we prove that the $h^K_H$-subcomplex given in Section~\ref{sec:rgh} computes $H_n([\PSL{\C}:\bar{P}];\Z)$. 
In Section~\ref{sec:inv} we define the invariants $\beta_B(M)$ and $\beta_P(M)$ of an orientable complete finite volume hyperbolic $3$-manifold in Adamson relative group homology. In Section~\ref{sec:bloch} we recall the definition of the Bloch group and extended Bloch group. In Section~\ref{sec:compute} 
using an ideal triangulation of $M$ we give formulas for the invariants $\beta_B(M)$ and $\beta_P(M)$.
In Section~\ref{sec:maps}, using the models of homogeneous spaces given in Section~\ref{sec:GUPB} we give two versions of the $h^K_H$-subcomplex and following Dupont--Zickert \cite{Dupont-Zickert:DFCCSC} we 
construct a homomorphism $\widehat{\sigma}\colon H_3(\PSL{\C},\bar{P};\Z)\to\EBg{\C}$.
In Section~\ref{sec:zrc} we prove that the complex of truncated simplices and the $h^K_H$-subcomplex are isomorphic and that under this isomorphism, the homomorphism $\widehat{\sigma}$ 
corresponds to the homomorphism $\Psi$ in \eqref{eq:Psi} given by Zickert using the complex of truncated simplices. Then, we express Zickert classes $F(M)$ in terms of the $h^K_H$-subcomplex to prove 
that all the clases $F(M)$ are mapped to the invariant $\beta_P(M)$ under the comparison homomorphism \eqref{eq:comp.hom}. In Section~\ref{sec:G.H.rep} we extend our construction for $(G,H)$--representations of tame manifolds, 
in particular for boundary-parabolic representations and we define its $PSL$-fundamental class $\fcPSL{\rho}\in H_3(\PSL{\C};\Z)$. 
In Section~\ref{sec:volume} we recall Zickert's definition of the complex volume of a boundary-parabolic representation $\rho$ and using a fomula for the class $\fcPSL{\rho}$ obtained using the $h^K_H$-subcomplex 
to compute $H_3(\PSL{\C},\bar{P};\Z)$ and the corresponding version $\widehat{\sigma}$ of homomorphism \eqref{eq:Psi}, we give an effective formula to compute the complex volume of $\rho$, in particular when $\rho$ is 
the geometric representation we get efficient formulas for the volume and Chern-Simons invariant of a complete oriented hyperbolic $3$-manifold of finite volume.

\subsection*{Remarks about this second version of the article}

There are some important diferences in this version of the article in comparison with the first one. The main difference is that in the first version, the proof of Proposition~3.20 is
not correct (and probably also the statement). The main implication of this proposition was that for the pair $(\PSL{\C},\bar{P})$ Adamson and Takasu relative homology groups coincide and that
the relative class $F(M)$ constructed by Zickert coincide with the invariant $\beta_P(M)$. This contradicted the remark made by Zickert \cite[Remark~5.19]{Zickert:VCSIR} that the class $F(M)$ is not
well-defined and depends on the choice of decoration on the truncated simplices since the invariant $\beta_P(M)$ is well-defined. In this second version we prove that the different classes $F(M)$
obtained by diferent choices of decorations map to $\beta_P(M)$ via the comparison homomorphism as was mentioned in the introduction.

Following, we briefly list the changes made in this second version. The results on Sections~2 and 3 of the first version were published in the article \cite{Arciniega-Cisneros:CRGCH} and Section~\ref{sec:rgh} of this version
was rewriten and only contains a summary of the results that we need, making reference for the proofs to \cite{Arciniega-Cisneros:CRGCH}. In this version we call Adamson relative group homology to what we called 
Hochschild relative group homology in the first version, since was Adamson who defined the chain complex which gives the theory and Hochschild only interpreted it in terms of relative homological algebra, also to avoid confusion with Hochschild homology
for associative algebras. Section~4 (now Section~\ref{sec:GUPB}) basically was unchanged, 
only Corollary~4.16 was removed becase was implied by Proposition~3.20. Section~5 and 6 (now Sections~\ref{sec:inv} and \ref{sec:bloch}) did not change, except for minor corrections and we added Theorem~\ref{thm:EBG.HPSL}.
Sections~7 and 8 were permuted (now Sections~ \ref{sec:maps} and \ref{sec:compute} respectively) and Subsection~7.3 was removed. In Section~8 (now Section~\ref{sec:zrc}) Subsection~\ref{ssec:fun.class} 
was rewritten. Section~10 (now Section~\ref{sec:G.H.rep}) was unchanged and finally Section~11 (now Section~\ref{sec:volume}) was rewritten.

Recently there has been some interest in this work \cite{Nosaka:QTP,Nosaka:OFR3CKGR}. 
We apologize for the long delay to upload this corrected version of the article to the ArXiv and for any inconvinience that this may have caused. We also are very grateful to all the people interested in this work.

\section{Relative group homologies}\label{sec:rgh}

Let $G$ be a discrete group and consider $\Z$ as a trivial $G$--module. It is well known that the homology groups $\homo(G;\Z)$ of $G$ can be defined topologically, as the homology groups of the classifying space $BG$ of $G$ with coefficients in $\Z$ (see \cite[Proposition~II (4.1)]{Brown:CohoGr}, or algebraically, as the homology groups of the \emph{standard free $G$--resolution} of $\Z$, tensored with $\Z$; or more generally as the groups $\mathrm{Tor}^G_n(\Z,\Z)$ (see \cite[III \S2]{Brown:CohoGr}).

Now consider a subgroup $H$ of $G$. 
In the literature there are two versions of \emph{relative homology groups} for the pair $(G,H)$. The first one, defined by Adamson in \cite{Adamson:CTNNSNNF}, generalises in a natural way the algebraic definition, while the second one, defined by Takasu in \cite{Takasu:RHRCTG}, generalises in a natural way the topological definition.

In this section we recall the definitions of Adamson and Takasu relative group homologies and some of their properties that we will use in the sequel. Here we mainly follow \cite{Arciniega-Cisneros:CRGCH}.

\subsection{Homology of permutation representations}

A \emph{permutation representation} is a pair $(G,X)$ where $G$ is a group and $X$ a $G$--set.
For any $G$--set $X$ we construct a complex $(\CC{C}(X),\partial_*)$ of abelian
groups, where $\CC[n]{C}(X)$ is the free abelian group generated by the
ordered $(n+1)$--tuples of elements of $X$. 
Define the $i$--th face homomorphism $d_i\co \CC[n]{C}(X)\to\CC[n-1]{C}(X)$ by
$d_i(x_0,\dots,x_n)=(x_0,\dots,\widehat{x_i},\dots,x_n)$, 
where $\widehat{x_i}$ denotes deletion. The boundary homomorphism
$\partial_n\co \CC[n]{C}(X)\to\CC[n-1]{C}(X)$ is given by $\partial_n=\sum_{i=0}^n (-1)^id_i$.
The standard computation shows that $\partial_n\circ\partial_{n+1}=0$
proving that $\CC[n]{C}(X)$ is indeed a complex. Define $\CC[-1]{C}(X)=\Z$ as
the infinite cyclic group generated by $(\ )$ and define $\partial_0(x)=(\ )$
for any $x\in X$. Notice that this extended complex is precisely the augmented
complex
\begin{equation*}
 \CC[n]{C}(X)\to\dots\to \CC[2]{C}(X)
\xrightarrow{\partial_2}\CC[1]{C}(X)\xrightarrow{\partial_1}\CC[0]{C}
(X)\xrightarrow{\varepsilon}\Z\to0,
\end{equation*}
with $\varepsilon=\partial_0$ the augmentation homomorphism. 

\begin{remark}\label{rem:resol}
This complex is acyclic \cite[Proposition~3.2]{Arciniega-Cisneros:CRGCH}, hence it is a
$G$--resolution of the trivial $G$--module $\Z$.
\end{remark}

The action of $G$ on $X$ induces an action of $G$ on $\CC[n]{C}(X)$ with
$n\geq0$ given by
\begin{equation*}
 g\cdot(x_0,\dots,x_n)=(g\cdot x_0,\dots,g\cdot x_n)
\end{equation*}
which endows $\CC[n]{C}(X)$ with the structure of a $G$--module. We also let $G$
act on $\CC[-1]{C}(X)=\Z$ trivially.
Denote by $(\CC{B}(X),\partial_*\otimes \mathrm{id}_\Z)$ the complex given by
\begin{equation}\label{eq:B(X)}
 \CC{B}(X)=\CC{C}(X)\otimes_{\Z[G]}\Z.
\end{equation}
We see $\CC{C}(X)$ as a right $G$--module by defining
\begin{equation*}
 (x_0,\dots,x_n)\cdot g:= g^{-1}\cdot(x_0,\dots,x_n).
\end{equation*}
The groups
\begin{equation*}
 H_n(G,X;\Z)=H_n(\CC{B}(X)),
\end{equation*}
are the \emph{homology groups of the permutation representation $(G,X)$}. The
corresponding cohomology groups were defined by Snapper in
\cite{Snapper:CPRISS}.

\begin{remark}
 If $X$ is a free $G$--set, in particular $X=G$ acting
on itself by left multiplication, the complex
$\CC{B}(G)$ computes the homology of $G$ \cite[Remark~3.4]{Arciniega-Cisneros:CRGCH}.
\end{remark}

\subsection{Adamson relative group homology}\label{ssec:Adamson}

If $H$ is a subgroup of $G$ and $X=G/H$ then we have that the homology of the
complex $\CC{B}(G/H)$ gives the \emph{Adamson relative homology groups}
\begin{equation}\label{eq:Hoch.def}
 H_i([G:H];\Z)=H_i(\CC{B}(G/H)),\quad n=0,1,2,\dots,
\end{equation}
defined by Adamson in \cite[\S3]{Adamson:CTNNSNNF} and Hochschild in \cite[\S4]{Hochschild:RHA}. 
More generally, let $\Xsubgr{H}$ be a $G$--set,
such that the action of $G$ on $X$ is \emph{transitive} and the set of isotropy
subgroups of points in $\Xsubgr{H}$ is the conjugacy class of $H$ in $G$. Then
$\Xsubgr{H}$ and $G/H$ are isomorphic as $G$--sets and we have that the homology
of the \emph{transitive permutation representation} $(G,\Xsubgr{H})$ is the
Adamson relative group homology, that is,
\begin{equation*}
 H_n([G:H];\Z)=H_n(G,\Xsubgr{H};\Z)=H_n(\CC{B}(\Xsubgr{H})),\quad n=0,1,2,\dots.
\end{equation*}
Notice that when $H=\{1\}$ we recover the homology groups $H_n(G;\Z)$ of $G$.

There is also a topological definition of Adamson relative group homology, using classifying spaces for
families of isotropy subgroups, which generalise de classifying space of a group. 
We recommend the survey article by L\"uck
\cite{Luck:SCSFS} and the bibliography there for more details.

Let $G$ be a discrete group.
A \emph{family} $\fami$ of subgroups of $G$ is a set of
subgroups of $G$ which is closed under conjugation and taking subgroups.
Let $\{H_{i}\}_{i\in I}$ be a set of subgroups of $G$, we denote by
$\fami(H_{i})$ the family
consisting of all the subgroups of the $\{H_{i}\}_{i\in I}$ and all their
conjugates 
by elements of $G$.

Let $\fami$ be a family of subgroups of $G$. A model for the \emph{classifying
space
for the family $\fami$ of subgroups} is a $G$--$CW$--complex $\EGF{G}$ which has
the following
properties:
\begin{enumerate}
 \item All isotropy groups of $\EGF{G}$ belong to $\fami$.
 \item For any $G$--$CW$--complex $Y$, whose isotropy groups belong to $\fami$,
there is up to $G$--homotopy a unique $G$--map $Y\to \EGF{G}$.\label{it:G-htpy}
\end{enumerate}
In other words, $\EGF{G}$ is a terminal object in the $G$--homotopy category of
$G$--$CW$--complexes, whose isotropy groups belong to $\fami$. In particular two
models
for $\EGF{G}$ are $G$--homotopy equivalent and for two families
$\fami_1\subseteq\fami_2$ there is up to $G$--homotopy precisely one $G$--map
$E_{\fami_1}(G)\to E_{\fami_2}(G)$.

There is a homotopy characterization of $\EGF{G}$ which allows us to determine
whether or not a given $G$--$CW$--complex is a model for $\EGF{G}$.

\begin{theorem}[{\cite[Theorem.~1.9]{Luck:SCSFS}}]\label{thm:hom.ch}
 A $G$--$CW$--complex $X$ is a model for $\EGF{G}$ if and only if all its
isotropy groups belong to $\fami$ and  the $H$--fixed point set $X^H$ is weakly
contractible for each $H\in\fami$ and is empty otherwise. In particular,
$\EGF{G}$ is contractible.
\end{theorem}

Let $H$ be a subgroup of $G$. For a $G$--space $X$, let
$\mathrm{res}_{H}^{G} X$ be the $H$--space obtained by restricting the
group action. If $\fami$ is a family of subgroups of $G$, let
$\fami/H=\sets{L\cap H}{L\in\fami}$ be the induced family of subgroups of
$H$.

\begin{proposition}[{\cite[Proposition~7.2.4]{tomDieck:SLN},
\cite[Proposition~A.5]{Farrell-Jones:ICAKT}}]\label{prop:res}
\begin{equation*}
\mathrm{res}_{H}^{G}\EGF{G}=E_{\fami/H}(H).
\end{equation*}
\end{proposition}

The following proposition gives a simplicial construction of a model for
$\EGF{G}$. Let $\{H_{i}\}_{i\in I}$ be a set of subgroups of $G$
such that every group in $\fami$ is conjugate to a subgroup of an
$H_{i}$, that is, $\fami=\fami(H_{i})$. Consider the disjoint union
$\Delta_\fami=\coprod_{i\in I}G/H_i$, or more generally, let $\Xfami$ be any
$G$--set such that $\fami$ is precisely the set of subgroups of $G$ which fix at
least one point of $\Xfami$. Notice that $\Delta_\fami$ is an example of such a
$G$--set.

\begin{proposition}[{\cite[Proposition~4.16]{Arciniega-Cisneros:CRGCH}}]\label{prop:const3}
Let $\Xfami$ be as above. A model for $\EGF{G}$ is the geometric realization $Y$
of the simplicial set whose $n$--simplices are the ordered
$(n+1)$--tuples $(x_0,\dots,x_n)$ of elements of $\Xfami$. The
face operators are given by
\begin{equation*}
d_i(x_0,\dots,x_n)=(x_0,\dots,\widehat{x_i},\dots,x_n),
\end{equation*}
where $\widehat{x_i}$ means omitting the element $x_i$. The degeneracy operators
are defined by
\begin{equation*}
s_i(x_0,\dots,x_n)=(x_0,\dots,x_i,x_i,\dots,x_n).
\end{equation*}
The action of $g\in G$ on an $n$--simplex $(x_0,\dots,x_n)$ of $Y$ gives $(gx_{0},\dots,gx_{n})$.
\end{proposition}

\begin{remark}\label{rem:fam.H}
Let $H$ be a subgroup of $G$ and consider the family $\fami(H)$ which consists
of all the subgroups of $H$ and their conjugates by elements of $G$. In this
case we can take $\Xfami[\fami(H)]=G/H$.
\end{remark}

\begin{remark}\label{rem:EG}
When $\fami=\{1\}$, the above construction corresponds to the \textit{universal
bundle} $EG$ of $G$. The $G$--orbit space of $EG$ is the classical classifying space $BG$ of $G$. 
In analogy with  $BG$, we denote by $\BGF{G}$ the $G$--orbit space of $\EGF{G}$.
Thus when $\fami=\{1\}$, we have that $B_{\{1\}}(G)=BG$.
\end{remark}

\begin{proposition}[{\cite[Corollary~4.27]{Arciniega-Cisneros:CRGCH}}]\label{prop:CS.HRGH}
Let $G$ be a discrete group. Let $H$ be a subgroup of $G$. Consider the family
of subgroups $\fami(H)$ generated by $H$. Then
\begin{equation*}
 H_n(\BGF[\fami(H)]{G};\Z)\cong H_n([G:H];\Z),\quad n=0,1,2,\dots,
\end{equation*}
where $\BGF[\fami(H)]{G}$ is the orbit space of the classifying space
$\EGF[\fami(H)]{G}$.
\end{proposition}

When $H=\{1\}$ we recover the well-known fact $H_\bullet(BG;\Z)\cong H_\bullet(G;\Z)$.

So we can say that the space $\BGF[\fami(H)]{G}$ is a classifying space for the
permutation representation $(G,\Xsubgr{H})$, compare with Blowers
\cite{Blowers:CSPR} where a classifying space for an arbitrary permutation
representation is constructed for Snapper's cohomology \cite{Snapper:CPRISS}.

\begin{proposition}[{\cite[Proposition~4.28]{Arciniega-Cisneros:CRGCH}}]\label{prop:3.IT}
Let $K$ be a normal subgroup of $G$ contained in $H$. Then we have an
isomorphism of relative homology groups
\begin{equation*}
 H_n([G:H];\Z)\cong H_n([G/K:H/K];\Z),\quad n=0,1,2,\dots.
\end{equation*}
\end{proposition}
If $H$ is a normal subgroup of $G$, then we have an isomorphism $H_\bullet(G/H;\Z)\cong H_\bullet([G:H];\Z)$.

\subsubsection{Natural $G$--maps and induced homomorphism}\label{sssec:maps}

Let $H$ and $K$ be subgroups of $G$. There exists a $G$--map $G/H\to G/K$ if and
only if 
there exists $a\in G$ such that $a^{-1}Ha\subset K$ and is given by
\begin{align*}
R_a\co G/H&\to G/K,\\
gH&\mapsto gaK. 
\end{align*}
Any $G$--map $G/H\to G/K$ is of the form $R_a$ for some $a\in G$ such that
$a^{-1}Ha\subset K$ and $R_a=R_b$ only if $a^{-1}b\in K$, see tom~Dieck
\cite[Proposition~I(1.14)]{tomDieck:TransGrp}.

Let $H$ and $K$ be subgroups of $G$ such that $H$ is conjugate to a subgroup of
$K$, then there is a $G$--map
\begin{equation*}
h_H^K\co \Xsubgr{H}\to\Xsubgr{K}.
\end{equation*}
This induces a $G$--homomorphism
\begin{equation*}
 (h_H^K)_*\co\CC{C}(\Xsubgr{H})\to\CC{C}(\Xsubgr{K}),
\end{equation*}
which in turn induces a homomorphism of homology groups
\begin{equation}\label{eq:hHK}
 (h_H^K)_*\co H_n([G:H];\Z)\to H_n([G:K];\Z).
\end{equation}

\begin{remark}\label{rem:hik*}
Let $H$ and $K$ be subgroups of $G$. Consider the families $\fami(H)$ and
$\fami(K)$ generated by $H$ and $K$ respectively and suppose that
$\fami(H)\subset\fami(K)$. 
Then there exists a $G$--map unique up to $G$--homotopy
$\EGF[\fami(H)]{G}\to\EGF[\fami(K)]{G}$.
Notice that $\fami(H)\subset\fami(K)$ implies that $H$ is conjugate to a
subgroup of $K$ and therefore there exists a $G$--map
$h_H^K\co\Xsubgr{H}=\Xfami[\fami(H)]\to\Xsubgr{K}=\Xfami[\fami(K)]$. Using the
Simplicial Construction of Proposition~\ref{prop:const3} we can see the $G$--map
$\EGF[\fami(H)]{G}\to\EGF[\fami(K)]{G}$ as the simplicial $G$--map given in
$n$--simplices by
\begin{equation*}
 (x_0,\dots,x_n)\mapsto \bigl(h_H^K(x_0),\dots,h_H^K(x_n)\bigr),\qquad
x_i\in\Xfami[\fami(H)],\ i=0,\dots,n.
\end{equation*}
This map induces a canonical map $\BGF[\fami(H)]{G}\to\BGF[\fami(K)]{G}$
between the corresponding $G$--orbit spaces. This map in turn induces a
canonical homomorphism in homology
\begin{equation*}
 H_n(\BGF[\fami(H)]{G};\Z)\to H_n(\BGF[\fami(K)]{G};\Z),
\end{equation*}
which by Proposition~\ref{prop:CS.HRGH} corresponds to the homomorphism
$(h_H^K)_*$ in \eqref{eq:hHK}.
\end{remark}

\subsection{Takasu relative group homology}

Let $H$ be a subgroup of $G$, the classifying space $BH$ can be regarded as a
subspace of the classifying space $BG$: let $\iota\colon BH\to BG$ be the map induced by the inclusion of $H$ in $G$; the \emph{mapping cylinder} $\mathrm{Cyl}(\iota)$ of $\iota$ is a model for $BG$ since it is homotopically equivalent to $BG$ and it clearly contains $BH$ as subspace. We define the \emph{Takasu relative
homology groups} denoted by $H_n(G,H)$, as the homology of the pair
$H_n(BG,BH;\Z)$.

There is a description of these relative
homology groups in terms of $G$--projective (in particular $G$--free)
resolutions: Let $H$ be a subgroup of $G$. Given a $G$--module $N$ we consider
the epimorphism
\begin{align*}
 \theta\co \Z[G]\otimes_{\Z[H]}N&\to N\\
   g\otimes n &\mapsto gn
\end{align*}
and define the $G$--module
\begin{equation*}
 I_{(G,H)}(N)=\ker\theta.
\end{equation*}
We have that $I_{(G,H)}(N)$ is a covariant exact functor from the category of
left $G$--modules to itself with respect to the variable $N$, see Takasu
\cite[Proposition~1.1 (i)]{Takasu:RHRCTG}.

Consider $\Z$ as a trivial $G$--module. For any free $G$--resolution $\CC{F}$ of
$I_{(G,H)}(\Z)$, we have a canonical isomorphism
\begin{equation}\label{eq:Takasu.Tor}
H_n(G, H )\cong  H_{n-1}(\CC{F}\otimes_{\Z[G]} \Z),\quad n=1,2,\dots.
\end{equation}
That is, $H_n(G,H) = \mathrm{Tor}^{\Z[G]}_{n-1}(I_{(G,H)}(\Z),\Z)$, see Takasu
\cite[Def.~\S2 (i) \& Proposition~3.2]{Takasu:RHRCTG}.

\begin{remark}\label{rem:IGH=ker}
Notice that  $\Z[G]\otimes_{\Z[H]}\Z \cong \Z[G/H]=\CC[0]{C}(G/H)$ and with this
isomorphism the epimorphism $\theta$ corresponds to the augmentation 
$\varepsilon\co\CC[0]{C}(G/H)\to\Z$. Therefore
\begin{equation*}
 I_{(G,H)}(\Z)=\ker\varepsilon.
\end{equation*}
\end{remark}

The relative homology groups fit in the long exact sequence
\cite[Proposition~2.1]{Takasu:RHRCTG}
\begin{equation}\label{eq:long.es}
\dots\to H_i(H;\Z)\to H_i(G;\Z)\to H_i(G,H;\Z)\to H_{i-1}(H;\Z)\to\dots
\end{equation}
which corresponds to the long exact sequence for the pair $(BG,BH)$.

Let $H$ and $K$ be a subgroup of $G$ such that $H\leq K\leq G$. 
There exists an induced homomorphism \cite[Proposition~1.2]{Takasu:RHRCTG}
\begin{equation*}
I_{(G,H)}(\Z)\to I_{(G,K)}(\Z),
\end{equation*}
which in turn, by the functoriality of $Tor_{n-1}^{\Z[G]}$ induces a
homomorphism
\begin{equation*}
H_n(G,H;\Z)\to H_n(G,K;\Z).
\end{equation*}

\subsubsection{Free $G$--resolution for $I_{(G,H)}(\Z)$}\label{sssec:free}

Let $G$ be a group of infinite order. Let $H$ and $K$ be subgroups of $G$ such that $H$ is conjugate to a subgroup of
$K$ and suppose $K$ has \emph{infinite} index in $G$. Consider a $G$--map $h_H^K\colon G/H\to G/K$ described in Subsection~\ref{sssec:maps}.

Let $\CC{C}^{h_H^K\neq}(\Xsubgr{H})$ be the subcomplex of $\CC{C}(\Xsubgr{H})$
generated by tuples mapping to different elements by the homomorphism $h_H^K$.
This subcomplex is acyclic \cite[Lemma~6.1]{Arciniega-Cisneros:CRGCH} and we call it the \emph{$h_H^K$--subcomplex} of $\CC{C}(\Xsubgr{H})$.
As before, set 
\begin{equation*}
 \CC{B}^{h_H^K\neq}(\Xsubgr{H})=\CC{C}^{h_H^K\neq}(\Xsubgr{H})\otimes_{\Z[G]}\Z.
\end{equation*}

\begin{remark}\label{rem:H.neq}
In the case $H=K$, we can take $h_H^H$ to be the identity. Hence
$\CC{C}^{h_H^H\neq}(\Xsubgr{H})$ is generated by tuples of distinct elements, so
we just denote it by $\CC{C}^{\neq}(\Xsubgr{H})$. 
\end{remark}

In some cases the $h_H^K$--subcomplex $\CC{C}^{h_H^K\neq}(\Xsubgr{H})$ 
computes $H_n(G,H;\Z)$.

Let $H$ and $K$ be subgroups of $G$ such that $H$ is a subgroup of
$K$. The group $H$ is \emph{$K$--malnormal} if for any
$g\notin K$ we have $H\cap gHg^{-1}=\{e\}$ where $e\in G$ is the identity
element. If $H$ is $H$--malnormal we just say it is \emph{malnormal}.

Let $H$ and $K$ be subgroups of $G$ such that $H$ is conjugate to a subgroup of
$K$. We say that the subgroup $H$ is \emph{$K$--malnormal conjugate} if $H$ is conjugate to a $K$--malnormal subgroup of $K$.

\begin{theorem}[{\cite[Theorem~6.4]{Arciniega-Cisneros:CRGCH}}]\label{thm:hHK.Takasu}
Let $H$ and $K$ be subgroups of $G$ such that $H$ is conjugate to a subgroup of
$K$ and suppose $K$ has \emph{infinite} index in $G$. Also assume that 
$H$ is $K$--malnormal conjugate.
Consider the $h_H^K$--subcomplex $\CC{C}^{h_H^K\neq}(G/H)$. Then $\CC{C}^{h_H^K\neq}(G/H)$ is a free
$G$-resolution of $I_{(G,H)}(\Z)$ and therefore for any trivial $G$--module $A$ we have isomorphisms
\begin{align*}
\homo[n](G,H;A)&\cong \homo[n](\CC{B}^{h_H^K\neq}(G/H;A)),\\
\end{align*}
for $ n=1,2,\dots$.
\end{theorem}

In Theorem~\ref{thm:hHK.Takasu} one has to be careful with the shift of indices in the definition of Takasu relative group homology (see \cite[Remark~6.5]{Arciniega-Cisneros:CRGCH}).

\begin{theorem}[{\cite[Theorem~7.13]{Arciniega-Cisneros:CRGCH}}]\label{thm:malnormal.A=T}
Let $H$ be a malnormal subgroup of $G$ and let $A$ be a trivial $G$--module. Then
\begin{equation*}
\homo[n](G,H;A)=\homo[n]([G:H];A)
\end{equation*}
for $n\geq 1$.
\end{theorem}

\subsection{Comparison of Adamson and Takasu relative group
homologies}\label{ssec:Taka.Hoch}

Adamson and Takasu relative group homologies do not coincide in general (see \cite[\S7]{Arciniega-Cisneros:CRGCH}).

From the definitions of $H_n(G,H;\Z)$ and $H_n([G:H];\Z)$ in terms of
resolutions we get a canonical homomorphism between them (see \cite[\S7.1]{Arciniega-Cisneros:CRGCH}).

Let $\CC{F}$ be  any free $G$--resolution of $I_{(G,H)}(\Z)$. By
Remarks~\ref{rem:resol} and \ref{rem:IGH=ker} we have that the complex
\begin{equation*}
 \dots\to
\CC[3]{C}(G/H)\xrightarrow{\partial_3}\CC[2]{C}(G/H)\xrightarrow{\partial_2}\CC[
1]{C}(G/H)\xrightarrow{\partial_1}I_{(G,H)}(\Z)\to 0,
\end{equation*}
is a $G$--resolution of $I_{(G,H)}(\Z)$. By the Comparison Theorem for
Resolutions there is a chain map
$\lambda_{n+1}\colon\CC[n]{F}\to\CC[n+1]{C}(G/H)$, $n\geq0$, unique up to chain homotopy
which by \eqref{eq:Takasu.Tor} and \eqref{eq:Hoch.def} induces a homomorphism
\begin{equation}\label{eq:Taka.Hoch}
 \lambda_n\colon H_n(G,H;\Z)\to H_n([G:H];\Z),\quad n=2,3,\dots.
\end{equation}

As a consequence of Theorem~\ref{thm:hHK.Takasu} we have the following  

\begin{corollary}[{\cite[Corollary~7.5]{Arciniega-Cisneros:CRGCH}}]\label{cor:sc.c}
Let $H$ and $K$ be subgroups of $G$ such that $H$ is conjugate to a subgroup of
$K$ and suppose $K$ has \emph{infinite} index in $G$. Also assume that 
$H$ is $K$--malnormal conjugate.
Then the chain morphism $\lambda_*$ is given by the inclusion of complexes
$\CC{C}^{h_H^K\neq}(G/H)\to\CC{C}(G/H)$.
\end{corollary}

In \cite[Theorem~7.13]{Arciniega-Cisneros:CRGCH} it is proved that a sufficient 
condition for the equality of Adamson and Takasu relative homologies of a pair $(G,H)$ is that
the subgroup $H$ be a malnormal subgroup.

\section{The group $\SL{\C}$ and some of its subgroups}\label{sec:GUPB}

We denote by $\C^\times$ the multiplicative group of the field of complex
numbers.
Consider the group $G=\SL{\C}$ and let $H$ be one of the
following subgroups of $G$: 
\begin{gather*}
\pm I =\left\{\begin{pmatrix}
            \pm1 & 0\\ 0 & \pm1
            \end{pmatrix}\right\},\quad 
T=\left\{\begin{pmatrix}
            a & 0\\ 0 & a^{-1}
            \end{pmatrix}\,\biggm|\,\text{$a\in\C^\times$}\right\},\quad
U=\left\{\begin{pmatrix}
            1 & b\\ 0 & 1
            \end{pmatrix}\,\biggm|\, b\in\C\right\},\\[5pt]
P=\left\{\begin{pmatrix}
            \pm1 & b\\ 0 & \pm1
           \end{pmatrix}\,\biggm|\, b\in\C\right\},\qquad
B=\left\{\begin{pmatrix}
           a & b\\ 0 & a^{-1}
           \end{pmatrix}\,\biggm|\,\text{$a\in\C^\times$, $b\in\C$}\right\}.
\end{gather*}
By abuse of notation, we denote by $I$ the identity matrix and also the subgroup
of $G$ which consists only of the identity matrix. 
Denote by $\bar{G}=G/\pm I=\PSL{\C}$. Given a subgroup $H$ of $G$ denote by
$\bar{H}$ the image of $H$ in $\bar{G}$. Notice that $\bar{U}=\bar{P}$. We
denote by $\bar{g}$ the element of $\bar{G}$ with representative $g\in G$. We
shall use this notation through the rest of the article except in
Section~\ref{sec:G.H.rep}.
As usual we consider all groups with the discrete topology. 

We list some known facts about these groups. Their proofs are in Lang
\cite{Lang:Algebra}.
\begin{itemize}
 \item $\SL{\C}$ is generated by elementary matrices
\cite[Lem.~XIII.8.1]{Lang:Algebra}.

\item $B$ is a maximal proper subgroup
\cite[Proposition~XIII.8.2]{Lang:Algebra}.

\item \textbf{Bruhat decomposition:} \cite[XIII \S8, p.~538]{Lang:Algebra}. Let
$w=\left(\begin{smallmatrix}
                                             0 & 1\\ -1 &0
                                            \end{smallmatrix}\right)$. There is
a decomposition of $\SL{\C}$ into disjoint subsets
\begin{equation*}
 \SL{\C}=B\cup BwB.
\end{equation*}

\item The subgroups $U$ and $P$ are normal in $B$ and we have the exact
sequences
\begin{gather}
I\to U\xrightarrow{i} B\to  T\cong \C^\times\to I\label{eq:es.UBT}\\[5pt]
I\to P\xrightarrow{i} B\to  T\cong \C^\times\to I.\notag
 \end{gather}

\item $\PSL{\C}$ is a simple group \cite[Theorem~XIII.8.4]{Lang:Algebra}. Hence,
$\pm I$ is the only normal subgroup of $\SL{\C}$.
\end{itemize}

In this section we give models for the $G$--sets $\Xsubgr{H}$ with $H=U,P,B$ and the explicit $G$-maps between them.

\begin{remark}\label{rem:XH=XbH}
Recall that for $H=P,B$ we have bijections of sets $G/H\cong \bar{G}/\bar{H}$
which are 
equivariant with respect to the actions of $G$ on $G/H$ and of $\bar{G}$ on
$\bar{G}/\bar{H}$ via the natural projection $G\to \bar{G}$. Thus, we have that
$\Xsubgr{H}=\Xsubgr{\bar{H}}$ as sets, the subgroup will indicate whether we are
considering the action of $G$ or $\bar{G}$ on it.
\end{remark}

\subsection{The $G$--set $\Xsubgr{U}$}

Consider the action of $G$ on $\C^2\setminus\{(0,0)\}$ given by left matrix
multiplication. 

\begin{proposition}\label{prop:trans}
 The group $G$ acts transitively on $\C^2\setminus\{(0,0)\}$.
\end{proposition}

\begin{proof}
 Let $(x,y)$ and $(z,w)$ be elements of $\C^2\setminus\{(0,0)\}$. Then a matrix
$g=(\begin{smallmatrix}
  a & b\\c & d
 \end{smallmatrix})\in SL_2(\C)
$ which send $(x,y)$ to $(z,w)$ is given by:
\begin{description}
 \item[If $x\neq0$] we have two cases:
\begin{description}
 \item[$w\neq0$]
\begin{equation*}
 a=\frac{zw+xy}{xw},\quad b=-\frac{x}{w},\quad c=\frac{w}{x},\quad d=0.
\end{equation*}
That is
\begin{equation*}
 g=\begin{pmatrix}
    \frac{zw+xy}{xw} & -\frac{x}{w}\\[5pt]\frac{w}{x} & 0
   \end{pmatrix}
\end{equation*}

\item[$w=0$] which implies that $z\neq0$
\begin{equation*}
 a=\frac{z-by}{x},\quad c=-\frac{y}{z},\quad d=\frac{x}{z},\quad b=\text{any
complex number.}
\end{equation*}
That is (simplifying taking $b=0$)
\begin{equation*}
g=\begin{pmatrix}
 \frac{z}{x} & 0 \\[5pt] -\frac{y}{z} & \frac{x}{z}
\end{pmatrix}
\end{equation*}
\end{description}
\item[If $x=0$] which implies $y\neq0$, then we have
\begin{equation*}
 b=\frac{z}{y},\quad d=\frac{w}{y},\quad a=\begin{cases}
					   0 & \text{if $z\neq0$,}\\
					   \frac{y}{w} & \text{if $w\neq0$.}
					  \end{cases}
\quad c=\begin{cases}
         -\frac{y}{z} & \text{if $z\neq0$,}\\
	 0 & \text{if $w\neq0$.}
      \end{cases}
\end{equation*}
That is
\begin{align*}
g&= \begin{pmatrix}
  0 & \frac{z}{y}\\[5pt] -\frac{y}{z} & \frac{w}{y}
 \end{pmatrix} &&\text{if $z\neq0$,}\\[5pt]
g&= \begin{pmatrix}
     \frac{y}{w} & \frac{z}{y}\\[5pt] 0 & \frac{w}{y}
    \end{pmatrix}&&\text{if $w\neq0$.}
\end{align*}
\end{description}
This proves the transitivity of the action.
\end{proof}

Hence we have that
\begin{proposition}
The isotropy subgroup of $(1,0)$ is $U$. Therefore, there is a $G$--isomorphism
between $SL_2(\C)/U$ and  $\Xsubgr{U}$ given by
\begin{align*}
 SL_2(\C)/U &\to \Xsubgr{U}\\
gU &\mapsto g\begin{pmatrix}
              1\\0
             \end{pmatrix}.
\end{align*}
\end{proposition}

\begin{proof}
Let $g=(\begin{smallmatrix}
       a & b\\ c & d
      \end{smallmatrix})\in \SL{\C}$. Then
\begin{equation*}
 \begin{pmatrix}
  a & b\\
  c & d
 \end{pmatrix}\begin{pmatrix}
               1\\0
              \end{pmatrix}=\begin{pmatrix}
                             a\\c 
                            \end{pmatrix}=\begin{pmatrix}
                                           1\\0
                                          \end{pmatrix}
\end{equation*}
implies that $a=1$ and $c=0$, but since the determinant of $g$ is $1$ we have
that $d=1$ and therefore $g=(\begin{smallmatrix}
       1 & b\\ 0 & 1
      \end{smallmatrix})\in U$.
\end{proof}
Therefore we can set $\Xsubgr{U}=\C^2\setminus\{(0,0)\}$.

\subsection{The $G$--set $\Xsubgr{P}$}\label{ssec:XP}

Now consider the action of $\Z_2$ on $\C^2\setminus\{(0,0)\}$ such that $-1$
sends the pair $(x,y)$ to its antipodal point $(-x,-y)$. Consider the orbit
space $\C^2\setminus\{(0,0)\}/\Z_2$ and denote by $[x,y]=\{(x,y),(-x,-y)\}$ the
orbit of the pair $(x,y)$.

Let $g\in G$ and $[x,y]\in\C^2\setminus\{(0,0)\}/\Z_2$. We define an action of
$G$ on $\C^2\setminus\{(0,0)\}/\Z_2$ given by
\begin{equation*}
 g\cdot\begin{bmatrix}
        x\\y
       \end{bmatrix}=\biggl[g\begin{pmatrix}
                              x\\y
                     \end{pmatrix}\biggr].
\end{equation*}
This gives a well-defined action since
\begin{equation*}
 \begin{pmatrix}
  a & b\\
  c & d
 \end{pmatrix}\begin{pmatrix}
               x\\y
              \end{pmatrix}=\begin{pmatrix}
                             ax+by\\cx+dy
                            \end{pmatrix}\sim\begin{pmatrix}
                                              -ax-by\\-cx-dy
                                             \end{pmatrix}=\begin{pmatrix}
                                                            a & b\\
							    c & d
                                                          
\end{pmatrix}\begin{pmatrix}
                                                                         -x\\-y
                                                                       
\end{pmatrix}.
\end{equation*}

\begin{remark}\label{rem:P.bG}
Since $-I$ acts as the identity, this action descends to an action of $\bar{G}$.
\end{remark}

By the definition of the action of $G$ on $\C^2\setminus\{(0,0)\}/\Z_2$ and
Proposition~\ref{prop:trans} we have
\begin{corollary}
 The group $G$ acts transitively on $\C^2\setminus\{(0,0)\}/\Z_2$.
\end{corollary}

\begin{proposition}\label{prop:Ca}
The isotropy subgroup of $[1,0]$ is $P$. Therefore, there is a $G$--isomorphism
between $SL_2(\C)/P$ and $\C^2\setminus\{(0,0)\}/\Z_2$ given by 
\begin{align*}
 \SL{\C}/P &\to \C^2\setminus\{(0,0)\}/\Z_2\\
gP &\mapsto g\begin{bmatrix}
              1\\0
             \end{bmatrix}.
\end{align*}
\end{proposition}

\begin{proof}
Let $g=(\begin{smallmatrix}
       a & b\\ c & d
      \end{smallmatrix})\in \SL{\C}$ such that
\begin{equation*}
 \begin{pmatrix}
  a & b\\
  c & d
 \end{pmatrix}\begin{bmatrix}
               1\\0
              \end{bmatrix}=\begin{bmatrix}
                             a\\c 
                            \end{bmatrix}=\begin{bmatrix}
                                           1\\0
                                          \end{bmatrix}.
\end{equation*}
This implies that $a=\pm1$ and $c=0$, but since $\det g=1$ we have that $d=\pm1$
and therefore $g=(\begin{smallmatrix}
       \pm1 & b\\ 0 & \pm1
      \end{smallmatrix})\in P$.
\end{proof}
Therefore we can set $\Xsubgr{P}=\C^2\setminus\{(0,0)\}/\Z_2$. By
Remark~\ref{rem:XH=XbH} we also have that
$\Xsubgr{\bar{P}}=\C^2\setminus\{(0,0)\}/\Z_2$.

There is another model for  $\Xsubgr{P}$ which we learned from Ramadas
Ramakrishnan. Consider the set $\Sym$ of $2\times 2$ non-zero symmetric complex
matrices with determinant zero. The set $\Sym$ is given by matrices of the form
\begin{equation}\label{eq:Sym}
\Sym=\biggl\{\begin{pmatrix}
       x^2 & xy \\ xy & y^2
      \end{pmatrix}\,\biggm|\,(x,y)\in \Xsubgr{U}\biggr\}.
\end{equation}
Let $g\in \SL{\C}$ and $S\in\Sym$. We define an action of $G$ on $\Sym$ by
\begin{equation*}
 g\cdot S=gSg^T,
\end{equation*}
where $g^T$ is the transpose of $g$. The action is well-defined because
transpose conjugation preserves symmetry and the determinant function is a
homomorphism. Since $-I$ acts as the identity, this action descends to an action
of $\bar{G}$.

\begin{proposition}\label{prop:trans.Sym}
 The group $G$ acts transitively on $\Sym$.
\end{proposition}

\begin{proof}
 Let $(\begin{smallmatrix}
        x^2 & xy\\ xy & y^2
       \end{smallmatrix})$ and 
$(\begin{smallmatrix}
          z^2 & zw\\ zw & w^2
         \end{smallmatrix})$ 
be elements of $\Sym$. Then the matrix
$g'=(\begin{smallmatrix}
  a & b\\c & d
 \end{smallmatrix})\in SL_2(\C)
$ which send 
$(\begin{smallmatrix}
        x^2 & xy\\ xy & y^2
       \end{smallmatrix})$
to 
$(\begin{smallmatrix}
          z^2 & zw\\ zw & w^2
         \end{smallmatrix})$
is given by:
\begin{description}
 \item[If $x\neq0$] we have two cases:
\begin{description}
 \item[$w\neq0$]
\begin{equation*}
 a=-\frac{zw+xy}{xw},\quad b=\frac{x}{w},\quad c=-\frac{w}{x},\quad d=0.
d=\text{any complex number}
\end{equation*}
That is
\begin{equation*}
g'=\begin{pmatrix}
  -\frac{zw+xy}{xw} & \frac{x}{w}\\[5pt] -\frac{w}{x} & 0
 \end{pmatrix}
\end{equation*}

\item[If $w=0$] which implies that $z\neq0$
\begin{equation*}
a=-\frac{z}{x},\quad c=\frac{y}{z},\quad d=-\frac{x}{z},\quad b=0
\end{equation*}
That is
\begin{equation*}
g'=\begin{pmatrix}
 -\frac{z}{x} & 0\\[5pt] \frac{y}{z} &-\frac{x}{z}
\end{pmatrix}
\end{equation*}
\end{description}
\item[If $x=0$] which implies that $y\neq0$, then we have
\begin{equation*}
 b=-\frac{z}{y},\quad d=-\frac{w}{y},\quad a=\begin{cases}
					   0 & \text{if $z\neq0$,}\\
					   -\frac{y}{w} & \text{if $w\neq0$.}
					  \end{cases}
\quad c=\begin{cases}
         \frac{y}{z} & \text{if $z\neq0$,}\\
	 0 & \text{if $w\neq0$.}
      \end{cases}
\end{equation*}
That is
\begin{align*}
g'&= \begin{pmatrix}
  0 & -\frac{z}{y}\\[5pt] \frac{y}{z} & -\frac{w}{y}
 \end{pmatrix} &&\text{if $z\neq0$,}\\[5pt]
g'&= \begin{pmatrix}
     -\frac{y}{w} & -\frac{z}{y}\\[5pt] 0 & -\frac{w}{y}
    \end{pmatrix}&&\text{if $w\neq0$.}
\end{align*}
\end{description}
\end{proof}

\begin{remark}
 Notice that the matrices $g'$ found in the proof of
Proposition~\ref{prop:trans.Sym} are the negatives of the corresponding matrices
$g$ found in the proof of Proposition~\ref{prop:trans}, that is $g'=-g$. In the
proof of Proposition~\ref{prop:trans.Sym} one can also take the corresponding
matrices $g$.
\end{remark}

\begin{proposition}\label{prop:Sym}
 The isotropy subgroup of $(\begin{smallmatrix}
                             1 & 0\\ 0 &0
                            \end{smallmatrix})\in\Sym$ is $P$. Therefore, there
is a $G$--isomorphism between $\SL{\C}/P$ and $\Sym$ given by 
\begin{align*}
 \SL{\C}/P &\to \Sym\\
gP &\mapsto g\begin{pmatrix}
              1 & 0\\0 &0
             \end{pmatrix}g^T.
\end{align*}
\end{proposition}

\begin{proof}
 Let $g=(\begin{smallmatrix}
       a & b\\ c & d
      \end{smallmatrix})\in \SL{\C}$ such that
\begin{equation*}
 \begin{pmatrix}
  a & b\\ c & d
 \end{pmatrix}\begin{pmatrix}
               1 & 0\\ 0 & 0
              \end{pmatrix}\begin{pmatrix}
                            a & c\\ b & d
                           \end{pmatrix}=\begin{pmatrix}
                                          a^2 & ac \\ ac & c^2
                                         \end{pmatrix}=\begin{pmatrix}
                                                        1 & 0\\ 0 & 0
                                                       \end{pmatrix}.
\end{equation*}
This implies that $a=\pm1$, $c=0$ and $d=\pm1$ since $\det g=1$. Thus
$g=(\begin{smallmatrix}
\pm1 & b\\ 0 &\pm1
\end{smallmatrix})\in P$ as claimed.
\end{proof}
We denote $\Xsp=\Sym$ to distinguish it from $\Xsubgr{P}$. By
Remark~\ref{rem:XH=XbH} we have that $\Sym$ is also a model for
$\bar{G}/\bar{P}$. We use the notation $\Xspb=\Sym$ to distinguish it from
$\Xsubgr{\bar{P}}$ and emphasize the action of $\bar{G}$.

\begin{corollary}\label{cor:iso.P}
 The sets $\Xsubgr{P}$ and $\Xsp$ are isomorphic as $G$--sets. 
\end{corollary}

\begin{proof}
 This is immediate from Propositions~\ref{prop:Ca} and \ref{prop:Sym}. We can
give an explicit isomorphism by using the transitivity of the actions of $G$ on
each of these $G$--sets. Let $[x,y]\in\Xsubgr{P}$, then
\begin{equation*}
 \begin{bmatrix}
  x\\ y
 \end{bmatrix}=\begin{pmatrix}
                x & z \\ y & w
               \end{pmatrix}\begin{bmatrix}
                             1 \\ 0
                            \end{bmatrix},
\end{equation*}
where $z$ and $w$ are complex numbers such that $xw-yz=1$. On the other hand, we
have
\begin{equation*}
 \begin{pmatrix}
  x^2 & xy \\ xy & y^2
 \end{pmatrix}=\begin{pmatrix}
                x & z\\ y & w
               \end{pmatrix}\begin{pmatrix}
                             1 & 0\\ 0 & 0
                            \end{pmatrix}\begin{pmatrix}
                                          x & y \\ z & w
                                         \end{pmatrix}
\end{equation*}
We define the isomorphism by
\begin{equation}\label{eq:iso.P}
\begin{split}
\varrho\co\Xsubgr{P} &\to \Xsp\\
\begin{bmatrix}
 x\\y
\end{bmatrix}
&\leftrightarrow \begin{pmatrix}
                      x^2 & xy \\ xy & y^2
                     \end{pmatrix}.
\end{split}
\end{equation}
It is equivariant because
\begin{equation*}
 \begin{pmatrix}
  a & b \\ c & d
 \end{pmatrix}\begin{bmatrix}
               x\\y
              \end{bmatrix}=\begin{bmatrix}
                             ax+by\\cx+dy
                            \end{bmatrix}
\end{equation*}
is sent by the isomorphism to
\begin{equation*}
 \begin{pmatrix}
(ax+by)^2 & (ax+by)(cx+dy)\\
(ax+by)(cx+dy) & (cx+dy)^2
\end{pmatrix}=\begin{pmatrix}
		a & b \\ c & d
	      \end{pmatrix}\begin{pmatrix}
                             x^2 & xy \\ xy & y^2
                            \end{pmatrix}\begin{pmatrix}
                                          a & c\\ b & d
                                         \end{pmatrix}.
\end{equation*}
 \end{proof}

\subsection{The $G$--set $\Xsubgr{B}$}\label{ssec:B}

Let $\widehat{\C}=\C\cup\{\infty\}$ be the Riemann sphere. Let
$\mathrm{LF}(\widehat{\C})$ be the group of fractional linear transformations on
$\widehat{\C}$. Let
\begin{align*}
 \phi\co G=\SL{\C}&\to \mathrm{LF}(\widehat{\C})\\
g=\begin{pmatrix}
 a & b\\
c & d
\end{pmatrix} & \mapsto \phi(g)(z)=\frac{az+b}{cz+d},
\end{align*}
be the canonical homomorphism. Let $g\in G$ act on $\widehat{\C}$ by the
corresponding fractional linear transformation $\phi(g)$. It is well-known that
this action is transitive. Abusing of the notation, given $g\in G$ and
$z\in\widehat{\C}$ we denote the action by $g\cdot z=\phi(g)(z)$.

We can also identify $\mathbb{CP}^1$ with $\widehat{\C}$ via
$[z_1:z_2]\leftrightarrow\frac{z_1}{z_2}$, where $[z_1:z_2]\in\mathbb{CP}^1$ is
written in homogeneous coordinates. In this case an element
$g=(\begin{smallmatrix}
    a& b\\ c&d                                                                  
\end{smallmatrix})\in G$ acts in
an element $[z_1:z_2]$ in $\mathbb{CP}^1$ by matrix multiplication
\begin{equation*}
 g\cdot [z_1:z_2]=\begin{pmatrix}
 a & b\\
c & d
\end{pmatrix}\begin{pmatrix}
   z_1\\ z_2
  \end{pmatrix}=[az_1+bz_2:cz_1+dz_2].
\end{equation*}

\begin{proposition}\label{prop:S2}
 The isotropy subgroup of $\infty\in\widehat{\C}$ is $B$. Therefore, there is a
$G$--isomorphism between $\SL{\C}/B$ and $\widehat{\C}$ given by 
\begin{align*}
 \SL{\C}/B &\to \widehat{\C}\\
gB &\mapsto g\cdot \infty.
\end{align*}
\end{proposition}

\begin{proof}
 Let $g=(\begin{smallmatrix}
       a & b\\ 0 & a^{-1}
      \end{smallmatrix})\in B$. Then
\begin{equation*}
g\cdot[1:0]= \begin{pmatrix}
       a & b\\ 0 & a^{-1}
      \end{pmatrix}\begin{pmatrix}
   1\\ 0
  \end{pmatrix}=\begin{pmatrix}
   a\\ 0
  \end{pmatrix}=[a:0]=[1:0].
\end{equation*}
It is easy to see that all the elements in $G$ that fix $[1:0]$ are of this
form, \ie that they are elements in $B$.
\end{proof}
Therefore we set $\Xsubgr{B}=\widehat{\C}$. Again, by Remark~\ref{rem:XH=XbH} we
also have that $\Xsubgr{\bar{B}}=\widehat{\C}$.

\subsection{The explicit $G$--maps}

The inclusions
\begin{equation}\label{eq:incl}
 I\hookrightarrow U \hookrightarrow P \hookrightarrow B
\end{equation}
induce $G$--maps
\begin{gather*}
 G \to G/U \to G/P \to G/B\\
g \mapsto gU \mapsto gP \mapsto gB.
\end{gather*}
Using the models $\Xsubgr{H}$ for the $G$--sets $G/H$ with $H=U,P,B$ given in
the previous subsections we give the  explicit $G$--maps between them. 

Let $g=(\begin{smallmatrix}
a & b \\ c &d                            
\end{smallmatrix})\in G$. We have the $G$--maps
\begin{equation}\label{eq:h.hHK}
\xymatrix{
G \ar[r]^{h_I^U} & \Xsubgr{U} \ar[r]^{h^P_U} & \Xsubgr{P} \ar[r]^{h^B_P} &
\Xsubgr{B}\\
g={\begin{pmatrix}
   a & b \\ c & d
  \end{pmatrix}}
 \ar@{|->}[r] & g{\begin{pmatrix}
                         1\\0
                   \end{pmatrix} = \begin{pmatrix}
                                         a\\c
                                        \end{pmatrix}}\ar@{|->}[r] &g\cdot
{\begin{bmatrix}
1\\0
\end{bmatrix} = \begin{bmatrix}
a\\c
\end{bmatrix}}\ar@{|->}[r] & 
g\cdot [1:0]=g\cdot\infty=\frac{a}{c}.
}
\end{equation}
Notice that $h_U^P\co \Xsubgr{U}\to\Xsubgr{P}$ is just the quotient map given by
the action of $\Z_2$. On the other hand, we have that
\begin{equation}\label{eq:UB}
h_U^B=h_P^B\circ h_U^P,
\end{equation}
where $h_U^B$ is the Hopf map
\begin{gather}
 h_U^B\co\Xsubgr{U}\to\Xsubgr{B}\notag\\
h_U^B(a,c)=\frac{a}{c}.\label{eq:hopf}
\end{gather}
Using $\Xsp$ instead of $\Xsubgr{P}$ we have
\begin{equation}\label{eq:bar.hHK}
\xymatrix{
G \ar[r]^{h_I^U} & \Xsubgr{U} \ar[r]^{\bar{h}_U^P} & \Xsp \ar[r]^{\bar{h}_P^B} &
\Xsubgr{B}\\
{\begin{pmatrix}
   a & b \\ c & d
  \end{pmatrix}}
 \ar@{|->}[r] &(a,c) \ar@{|->}[r] & {\begin{pmatrix}
                       a^2 & ac\\
			ac & c^2
                      \end{pmatrix}} \ar@{|->}[r] &
\frac{a^2}{ac}=\frac{ac}{c^2}=\frac{a}{c}.\\
}
\end{equation}
We have that 
\begin{equation*}\label{eq:bar.UB}
 \bar{h}_P^B\circ \bar{h}_U^P=h_U^B.
\end{equation*}
It is also useful to write the $G$--map $\bar{h}_U^P\co\Xsp\to\Xsubgr{B}$ in
terms of the entries
of the matrix in $\Xsp$ without writing it in the form given in \eqref{eq:Sym}.
Let $(\begin{smallmatrix}
   r & t\\ t & s                                                                
\end{smallmatrix}
)\in\Xsp$, that is, $(\begin{smallmatrix}
   r & t\\ t & s                                                                
\end{smallmatrix})
\neq (\begin{smallmatrix}
        0 & 0\\0 & 0
       \end{smallmatrix})$ and $rs=t^2$. Then we have
\begin{align*}
\bar{h}_U^P\co \Xsp &\to \Xsubgr{B}\\[5pt]
\begin{pmatrix}
  r & t\\
  t & s
\end{pmatrix} &\mapsto  \frac{r}{t}=\frac{t}{s}.
\end{align*}

\begin{remark}\label{rem:h.bhHK}
For the case of $\bar{G}$ we have practically the same $\bar{G}$--homomorphisms
as in \eqref{eq:h.hHK} and \eqref{eq:bar.hHK} except that
$\Xsubgr{\bar{U}}=\Xsubgr{\bar{P}}$. 
\end{remark}

\begin{remark}
Consider $\infty\in\Xsubgr{B}$ and its inverse image under the Hopf map
\eqref{eq:hopf}
\begin{equation*}
 (h_U^B)^{-1}(\infty)=\sets{(x,0)}{x\in\C^\times}\subset\Xsubgr{U},
\end{equation*}
which corresponds to the first coordinate complex line minus the origin. Since
by Proposition~\ref{prop:S2} the isotropy subgroup of $\infty\in\Xsubgr{B}$
under the action of $G$ is $B$, we have that $(h_U^B)^{-1}(\infty)$ is a
$B$-invariant subset of $\Xsubgr{U}$. Since the short exact sequence
\eqref{eq:es.UBT} splits, any element of $B$ can be written in a unique way as
the product of an element in $U$ and an element in $T$
\begin{equation}\label{eq:B=UT}
 \begin{pmatrix}
a & b\\
0 & a^{-1}
 \end{pmatrix}=\begin{pmatrix}
                1 & ab\\
                0 & 1
               \end{pmatrix}\begin{pmatrix}
                             a & 0\\
                             0 & a^{-1}
                            \end{pmatrix}.
\end{equation}
It is easy to see that $U$ fixes pointwise the points of $(h_U^B)^{-1}(\infty)$,
while $T$ acts freely and transitively on $(h_U^B)^{-1}(\infty)$ where the
matrix 
$\big(\begin{smallmatrix}
 a & b\\
0 & a^{-1}                                                                      
\end{smallmatrix}\bigr)$ acts multiplying $(x,0)$ by $a\in\C^\times$.

Another way to interpret this is to write $\infty\in\Xsubgr{B}$ in homogeneous
coordinates $[\lambda:0]$, then the elements in $U$ fix $\infty$ and the
homogeneous coordinates $[\lambda:0]$, while an element in
$\big(\begin{smallmatrix}
 a & b\\
0 & a^{-1}                                                                      
\end{smallmatrix}\bigr)\in T$ fix $\infty$ but multiply the  homogeneous
coordinates by $a\in\C^\times$ obtaining the homogeneous coordinates
$[a\lambda:0]$.

More generally, for any point $z\in\Xsubgr{B}$, its isotropy subgroup $G_z$ is a
conjugate of $B$, which can be written as the direct product of the
corresponding conjugates of $U$ and $T$, which we denote by $U_z$ and $T_z$.
Writing $z\in\Xsubgr{B}$ in homogeneous coordinates $[\lambda z:\lambda]$ the
elements of $U_z$ fix $z$ and the homogeneous coordinates, while the elements of
$T_z$ fix $z$ but multiply the homogeneous coordinates by a constant.
\end{remark}

\begin{remark}\label{rem:im.inv}
Analogously, consider $\infty\in\Xsubgr{B}$ and its inverse image under the
$G$--map $h_P^B$
\begin{equation*}
 (h_P^B)^{-1}(\infty)=\sets{[x,0]}{x\in\C^\times}\subset\Xsubgr{P}.
\end{equation*}
By \eqref{eq:B=UT} any element of $P$ can be written in a unique way as the
product of an element in $U$ and an element in $T$
\begin{equation}\label{eq:P=UT}
 \begin{pmatrix}
\pm1 & b\\
0 & \pm1
 \end{pmatrix}=\begin{pmatrix}
                1 & \pm b\\
                0 & 1
               \end{pmatrix}\begin{pmatrix}
                             \pm1 & 0\\
                             0 & \pm1
                            \end{pmatrix}.
\end{equation}
Given a representative $(x,0)$ of $[x,0]\in (h_P^B)^{-1}(\infty)$, the matrix
$\bigl(\begin{smallmatrix}
-1 & 0\\ 0 & -1
\end{smallmatrix}\bigr)\in T$
changes the sign of the representative to $(-x,0)$ but fixes its class $[x,0]$
and the matrix $\bigl(\begin{smallmatrix}
         1 & \pm b\\ 0 & 1
      \end{smallmatrix}\bigr)\in U$ fixes any representative of $[x,0]$, thus it
fixes the class itself. Therefore. the elements in $P$ fix pointwise the points
in $(h_P^B)^{-1}(\infty)$ while $T$ acts transitively on $(h_P^B)^{-1}(\infty)$
with isotropy $\Z_2$, where the matrix 
$\bigl(\begin{smallmatrix}
 a & b\\
0 & a^{-1}                                                                      
\end{smallmatrix}\bigr)$ acts multiplying $[x,0]$ by $a\in\C^\times$ obtaining
$[ax,0]$.

If we use instead $\bar{h}_P^B$ the inverse image of $\infty\in\Xsubgr{B}$ is
given by
\begin{equation*}
{(\bar{h})_P^B}^{-1}(\infty)=\biggl\{\begin{pmatrix}
                                   x & 0\\ 0 & 0
                                   \end{pmatrix} \biggm|
x\in\C^\times\biggr\}\subset\Xsp.
\end{equation*}
The elements in $P$ fix pointwise the points in $(\bar{h}_P^B)^{-1}(\infty)$
while $T$ acts transitively on $(\bar{h}_P^B)^{-1}(\infty)$ with isotropy
$\Z_2$, where the matrix 
$\big(\begin{smallmatrix}
 a & b\\
0 & a^{-1}                                                                      
\end{smallmatrix}\bigr)$ acts multiplying $\bigl(\begin{smallmatrix}
                  x & 0\\ 0 & 0
                 \end{smallmatrix}\bigr)$ by $a^2$ with $a\in\C^\times$obtaining
$\bigl(\begin{smallmatrix}
                  a^2x & 0\\ 0 & 0
                 \end{smallmatrix}\bigr)$.

More generally, for any point $z\in\Xsubgr{B}$, its isotropy subgroup $G_z$ is a
conjugate of $B$, and let $P_z$ denote the corresponding conjugate of $P$. The
elements of $P_z$ fix pointwise the points in $(\bar{h}_P^B)^{-1}(z)$ while
$T_z$ acts transitively on $(\bar{h}_P^B)^{-1}(z)$ multiplying by a constant.
\end{remark}

\subsection{Canonical homomorphisms}

As in Remark~\ref{rem:fam.H}, we denote by $\fami(H)$ the family of
subgroups of $G$ generated by $H$. The inclusions \eqref{eq:incl} induce the
inclusions of
families of subgroups of $G$
\begin{equation*}
 \fami(I)\hookrightarrow \fami(U) \hookrightarrow \fami(P) \hookrightarrow
\fami(B)
\end{equation*}
and in turn, these inclusion give canonical $G$--maps between classifying spaces
\begin{equation}\label{eq:G.maps}
 EG\to \EGF[\fami(U)]{G} \to \EGF[\fami(P)]{G}\to \EGF[\fami(B)]{G}
\end{equation}
which are unique up to $G$--homotopy. Taking the quotient by the action of $G$
we
get canonical maps
\begin{equation*}
 BG\to \BGF[\fami(U)]{G} \to \BGF[\fami(P)]{G}\to \BGF[\fami(B)]{G}.
\end{equation*}
The homomorphisms induced in homology give the sequence
\begin{equation*}
H_i(BG)\to H_i(\BGF[\fami(U)]{G}) \to H_i(\BGF[\fami(P)]{G})\to
H_i(\BGF[\fami(B)]{G}),
\end{equation*}
which by Proposition~\ref{prop:CS.HRGH} is the same as the sequence of
homomorphisms
\begin{equation*}
H_i(G)\xrightarrow{(h_I^U)_*} H_i([G:U])\xrightarrow{(h_U^P)_*}
H_i([G:P])\xrightarrow{(h_P^B)_*} H_i([G:B]).
\end{equation*}

Recall that we denote by $\bar{G}=G/\pm I=\PSL{\C}$ and given a subgroup $H$ of
$G$ we denote by $\bar{H}$ the image of the $H$ in $\bar{G}$. Notice that
$\bar{U}=\bar{P}$.  Analogously, the inclusions
$\bar{I}\hookrightarrow \bar{P} \hookrightarrow \bar{B}$
induce a 
sequence of homomorphisms
\begin{equation*}
H_i(\bar{G})\to H_i([\bar{G}:\bar{P}])\to H_i([\bar{G}:\bar{B}]).
\end{equation*}

The relation between the coset sets of $G$ and $\bar{G}$ can be shown in the
following diagram
\begin{equation*}
\xymatrix{
G \ar[r]\ar[dr]& G/U \ar[r]& G/P=\bar{G}/\bar{P} \ar[r]& G/B=\bar{G}/\bar{B}\\
 & \bar{G}=G/\pm I\ar[r]\ar[ru] & \bar{G}/\bar{T}=G/T.\ar[ur] &
}
\end{equation*}
In turn, by Proposition~\ref{prop:3.IT} this induces the following commutative
diagram of relative homology groups
\begin{equation}\label{eq:hom.diag}
\xymatrix{
H_i(G) \ar[r]\ar[dr]& H_i([G:U]) \ar[r]& H_i([G:P]) \ar[r]& H_i([G:B])\\
 & H_i(\bar{G})\ar[r]\ar[ru] & H_i([G:T]).\ar[ur] &
}
\end{equation}

\subsection{Takasu relative homology of $\SL{\C}$}

The following proposition states that for the cases of $(G,U)$ and $(\bar{G},\bar{P})$ 
Takasu relative homology can be computed using Theorem~\ref{thm:hHK.Takasu},
compare with \cite[Remark.~3.6]{Dupont-Zickert:DFCCSC} and
\cite[\S7]{Zickert:VCSIR}). 

\begin{lemma}\label{lem:UgUg}
Let $g\notin B$. Then
\begin{align*}
U\cap gUg^{-1}&=I.\\
P\cap gPg^{-1}&=\pm I.
\end{align*}
Hence, $U$ is a $B$-malnormal subgroup of $\SL{\C}$ and $\bar{P}$ is a $\bar{B}$-malnormal subgroup of $\PSL{\C}$.
\end{lemma}

\begin{proof}
Since $g\notin B$ by Bruhat decomposition we have that $g$ can be written as
\begin{equation*}
 g=g_1wg_2,\quad g_1,g_2\in B.
\end{equation*}
Then we have that
\begin{equation*}
 gUg^{-1}=g_1wg_2Ug_2^{-1}w^{-1}g_1^{-1}=g_1wUw^{-1}g_1^{-1}.
\end{equation*}
Let us analyze the elements in $g_1wUw^{-1}g_1^{-1}$. Consider
$h=\left(\begin{smallmatrix}
1 & e\\ 0 & 1
\end{smallmatrix}\right)\in U$ and
$g_1=\left(\begin{smallmatrix}
            a & b\\ 0 & a^{-1}
           \end{smallmatrix}\right)\in B$, so $a\in\C^\times$. We have that
\begin{equation*}
g_1whw^{-1}g_1^{-1}=\begin{pmatrix}
                    1+a^{-1}be & -b^2e\\ a^{-2}e & 1- a^{-1}be
                   \end{pmatrix}.
\end{equation*}
The only way to have $g_1whw^{-1}g_1^{-1}\in U$ is to have $e=0$ and in that
case $g_1whw^{-1}g_1^{-1}=I$.
Analogously if $h=\left(\begin{smallmatrix}
                \pm1 & e\\ 0 & \pm1
               \end{smallmatrix}\right)\in P$ then
\begin{equation*}
 g_1whw^{-1}g_1^{-1}=\begin{pmatrix}
                    \pm1+a^{-1}be & -b^2e\\ a^{-2}e & \pm1- a^{-1}be
                   \end{pmatrix},
\end{equation*}
and the only way to have $g_1whw^{-1}g_1^{-1}\in P$ is to have $e=0$ and in that
case $g_1whw^{-1}g_1^{-1}=\pm I$.

\end{proof}

Consider the $h_U^B$--subcomplex $\CC{C}^{h_U^B\neq}(\Xsubgr{U})$ and the
$h_{\bar{P}}^{\bar{B}}$--subcomplex
$\CC{C}^{h_{\bar{P}}^{\bar{B}}\neq}(\Xsubgr{\bar{P}})$ defined in
Subsection~\ref{sssec:free}. 
\begin{proposition}\label{prop:tak.hUB.sub}
We have isomorphisms
\begin{align*}
 H_n(G,U;\Z)&\cong H_n(\CC{B}^{h_U^B\neq}(\Xsubgr{U})),\quad n=2,3,\dots.\\
H_n(\bar{G},\bar{P};\Z)&\cong
H_n(\CC{B}^{h_{\bar{P}}^{\bar{B}\neq}}(\Xsubgr{\bar{P}})),\quad n=2,3,\dots.
\end{align*}
\end{proposition}

\begin{proof}
By Lemma~\ref{lem:UgUg} $U$ is a $B$-malnormal subgroup of $\SL{\C}$ and $\bar{P}$ is a $\bar{B}$-malnormal subgroup of $\PSL{\C}$.
Hence, the result follows by Theorem~\ref{thm:hHK.Takasu}.
\end{proof}

\begin{remark}
 The first isomorphism is mentioned in \cite[Remark~3.6]{Dupont-Zickert:DFCCSC} but no proof was given, while in \cite[\S7]{Zickert:VCSIR}
it is mentioned that follows from \cite[Theorem~2.1]{Zickert:VCSIR} but there is no proof that $\CC{C}^{h_U^B\neq}(\C^2\setminus\{(0,0)\})$ is a $\SL{\C}$-free resolution
of the kernel of the augmentation map.
\end{remark}

\section{Invariants for finite volume hyperbolic $3$--manifolds}\label{sec:inv}

In this section we generalize the construction given in Cisneros-Molina--Jones
\cite{Cisneros-Jones:Bloch} to 
define invariants $\beta_H(M)\in H_3([\SL{\C}:H];\Z)$ of a complete oriented
hyperbolic $3$--manifold of finite volume $M$, where $H$ is one of the subgroups
$P$ or $B$ of $\SL{\C}$ defined in Section~\ref{sec:GUPB}.

\subsection{Hyperbolic $3$--manifolds and the fundamental class of
$\widehat{M}$}\label{ssec:fc}

Consider the upper half space model for the hyperbolic $3$--space $\hyp$ and
identify it with the set of quaternions $\sets{z+t\jq}{z\in\C,\quad t>0}$.
Let $\hypc=\hyp\cup\widehat{\C}$ be the standard compactification of $\hyp$.
The group of orientation
preserving isometries of $\hyp$ is isomorphic to
$\PSL{\C}$ and the action of $\bigl(\begin{smallmatrix}
a&b&\\c&d
\end{smallmatrix}\bigr)\in\PSL{\C}$  in $\hyp$ is given by the linear fractional
transformation
\begin{equation*}
\phi(w)=(aw+b)(cw+d)^{-1},\qquad w=z+t\jq,\qquad ad-bc=1,
\end{equation*}
which is the Poincar\'e extension to $\hyp$ of the complex linear fractional
transformation on $\widehat{\C}$ given by $\bigl(\begin{smallmatrix}
							a&b&\\c&d
						\end{smallmatrix}\bigr)$.
Recall that isometries of hyperbolic $3$--space $\hyp$ can be of three types:
\emph{elliptic} if fixes a point in $\hyp$; \emph{parabolic} if fixes no point
of $\hyp$ and fixes a unique point of $\widehat{\C}$ and \emph{hyperbolic} if
fixes no point of $\hyp$ and fixes two points of $\widehat{\C}$. 

A subgroup of $\SL{\C}$ or $\PSL{\C}$ is called \emph{parabolic} if all its
elements correspond to parabolic isometries of $\hyp$ fixing a common point in
$\widehat{\C}$. Since the action of
$\SL{\C}$ (or $\PSL{\C}$) in $\widehat{\C}$ is transitive and  the conjugates 
of
parabolic isometries  are parabolic \cite[(4.7.1)]{Ratcliffe:HypMan} we
can assume  that the fixed point is the  point at  infinity $\infty$
which we  denote by  its homogeneous coordinates  $\infty=[1:0]$ and
therefore parabolic subgroups are conjugate to a group of
matrices of the form  
$\left(\begin{smallmatrix}
\pm 1 & b\\
0 & \pm 1
\end{smallmatrix}\right)$, with $b\in\C$, or its image in $\PSL{\C}$. 
In other words, a parabolic subgroup of $\SL{\C}$ or $\PSL{\C}$ is conjugate to
a subgroup of $P$ or $\bar{P}$ respectively.

A complete oriented hyperbolic $3$--manifold $M$ is the quotient of the
hyperbolic $3$--space $\hyp$ by a discrete, torsion-free subgroup
$\dsubgr$ of orientation preserving isometries. 
Since $\dsubgr$ is torsion-free, it acts freely on $\hyp$
\cite[Theorem~8.2.1]{Ratcliffe:HypMan} and therefore it consist only of
parabolic and hyperbolic isometries.
Notice that since $\hyp$ is contractible it is the universal cover of $M$ and
therefore $\hg{M}=\dsubgr$ and  $M$ is a $\EMc{\dsubgr}{1}$, \ie $M=B\dsubgr$,
the classifying space of $\dsubgr$.
To such an 
hyperbolic $3$--manifold we can associate a \emph{geometric representation}
$\bar{\rho}\co\hg{M}=\dsubgr\to\PSL{\C}$ given by the inclusion, which is
canonical up to equivalence.
This representation can be lifted to a representation $\rho\co\dsubgr\to\SL{\C}$
\cite[Proposition~3.1.1]{Culler-Shalen:VGRS3M}. There is a one-to-one
correspondence between such lifts and spin structures on $M$.
We identify $\dsubgr$ with a subgroup of $\SL{\C}$ using the
representation $\rho\co\dsubgr\to\SL{\C}$ which corresponds to
the spin structure.

Let $M$ be an orientable complete hyperbolic $3$--manifold of finite volume. 
Such manifolds contain a compact $3$--manifold-with-boundary $\Mo$
such that $M-\Mo$ is the disjoint union of a finite number of
cusps. Each cusp of $M$ is diffeomorphic to $T^{2}\times(0,\infty)$,
where $T^{2}$ denotes the $2$--torus, 
see for instance \cite[page~647, Corollary~4 and
Theorem~10.2.1]{Ratcliffe:HypMan}.
The number of cusps can be zero, and this case corresponds when the 
manifold $M$ is a closed manifold.

Let $M$ be an oriented complete hyperbolic $3$--manifold of finite volume
with $d$ cusps, with $d>0$. Each boundary component $T^{2}_{i}$ of $\Mo$ defines
a subgroup $\dsubgr_{i}$
of $\hg{M}$ which is well defined up to conjugation. The
subgroups $\dsubgr_{i}$ are called the \emph{peripheral
subgroups} of $\dsubgr$. 
The image of $\dsubgr_{i}$ under the representation
$\rho\co\dsubgr\to\SL{\C}$ given by the inclusion is a free
abelian group of rank 2 of $\SL{\C}$. 
The subgroups $\dsubgr_{i}$ are parabolic subgroups of $\SL{\C}$.
Hence we have that $\dsubgr_{i}\in\fami(P)$.
Therefore the image of $\dsubgr_{i}$ under the representation
$\bar{\rho}\co\dsubgr\to\PSL{\C}$ is contained in $\fami(\bar{P})$.

Let $M=\dsubgr\backslash\hyp$ be a non-compact orientable complete
hyperbolic $3$--manifold of finite volume. Let
$\pi\co\hyp\to\dsubgr\backslash\hyp=M$ be the universal cover of $M$. Consider
the set $\mathcal{C}$ of fixed points of parabolic elements of $\dsubgr$ in
$\widehat{\C}$ and divide by the action of $\dsubgr$. The elements of the
resulting set $\widehat{\mathcal{C}}$ are called the \emph{cusp points} of $M$.

\begin{remark}\label{rem:hypebolic.no.fix.c}
No hyperbolic element in $\dsubgr$ has as fixed point any point in
$\mathcal{C}$, otherwise the group $\dsubgr$ would not be discrete
\cite[Theorem~5.5.4]{Ratcliffe:HypMan}.
\end{remark}

Let $\widehat{\hyp}=\hyp\cup\mathcal{C}$ and consider
$\widehat{M}=\dsubgr\backslash\widehat{\hyp}$. If $M$ is closed
$\mathcal{C}=\emptyset$ and $\widehat{M}=M$, if $M$ is non-compact we have
that $\widehat{M}$ is the \emph{end-compactification} of $M$ which is the
result of adding the cusp points of $M$. We get an extension of the covering map
$\pi$ to a map $\widehat{\pi}\co \widehat{\hyp}\to\widehat{M}$.

Consider as well the \emph{one-point-compactification} $\Mp$ of $M$ which
consists in identifying all the cusps points of $\widehat{M}$ to a single point.
Since $M$ is homotopically equivalent to the compact $3$--manifold-with-boundary
$\Mo$ we have that $\Mp=\widehat{M}/\widehat{\mathcal{C}}=\Mo/\partial\Mo$. 
By the exact sequence of the pair $(\widehat{M},\widehat{\mathcal{C}})$ we have
that $H_3(\widehat{M};\Z)\cong H_3(\widehat{M},\widehat{\mathcal{C}};\Z)$ and
therefore we have that
\begin{equation}\label{eq:hM.fc}
H_3(\widehat{M};\Z)\cong H_3(\widehat{M},\widehat{\mathcal{C}};\Z)\cong
H_3(\widehat{M}/\widehat{\mathcal{C}};\Z)
\cong H_3(\Mp;\Z)\cong H_3(\Mo,\partial\Mo;\Z)\cong\Z.
\end{equation}
We denote by $[\widehat{M}]$ the generator and call it the \emph{fundamental
class of $\widehat{M}$}.

\subsection{Invariants of hyperbolic $3$--manifolds of finite
volume}\label{ssec:inv}

Let $M$ be a compact oriented hyperbolic $3$--manifold. To the canonical
representation $\bar{\rho}\co\hg{M}\to\PSL{\C}$ corresponds a map $B\rho\co M\to
B\PSL{\C}$
where $B\PSL{\C}$ is the classifying space of $\PSL{\C}$. There is a
well-known invariant $\fcPSL{M}$ of $M$ in the group $H_3(\PSL{\C};\Z)$ given by
the image of the
fundamental class of $M$ under the homomorphism induced in homology by $B\rho$.

As we said before, we generalize the construction given in
Cisneros-Molina--Jones \cite{Cisneros-Jones:Bloch} to extend this invariant when
$M$ is a complete oriented hyperbolic $3$--manifold of finite volume (\ie $M$
is compact or with cusps) to invariants $\beta_H(M)$, but in this case
$\beta_H(M)$ takes values in
$H_3([\PSL{\C}:\bar{H}];\Z)$, where $\bar{H}$ is one of the subgroups $\bar{P}$ or $\bar{B}$ of
$\bar{G}=\PSL{\C}$ defined in Section~\ref{sec:GUPB}. 

Let $\dsubgr$ be a discrete torsion-free subgroup of $\PSL{\C}$.
The action of $\dsubgr$ on the hyperbolic $3$--space $\hyp$ is free and
since $\hyp$ is contractible, by Theorem~\ref{thm:hom.ch} it is a model for
$E\dsubgr$.

The action of $\dsubgr$ on $\widehat{\hyp}$ is no longer free. The 
points in $\mathcal{C}$ have as isotropy subgroups the peripheral subgroups
$\dsubgr_1,\dots,\dsubgr_d$ of $\dsubgr$ or their conjugates and any subgroup in
$\fami(\dsubgr_1,\dots,\dsubgr_d)$ fixes only one point in $\mathcal{C}$.
Therefore, by Theorem~\ref{thm:hom.ch} we have that
$\widehat{\hyp}$ is a model for $\EGF[\fami(\dsubgr_1,\dots,\dsubgr_d)]{\dsubgr}$.

We have the following facts:
\begin{itemize}
 \item Since $\{I\}\subset\fami(\Gamma_1,\dots\Gamma_d)$ there is a
$\dsubgr$--map $\hyp\to\widehat{\hyp}$ unique up to $\dsubgr$--homotopy. We can use
the inclusion.
 \item By Proposition~\ref{prop:res} $\mathrm{res}_\dsubgr^{\bar{G}} E\bar{G}$ is a model for
$E\dsubgr$.
Therefore, there is a $\dsubgr$--homotopy equivalence $\hyp\to
\mathrm{res}_\dsubgr^{\bar{G}} E\bar{G}$ which is unique up to $\dsubgr$--homotopy. 
 \item Since
$\fami(\Gamma_1,\dots\Gamma_d)=\fami(\bar{P})/\dsubgr\subset\fami(\bar{B})/\dsubgr$ we have
$\dsubgr$--maps
\begin{equation*}
 \widehat{\hyp}\to \mathrm{res}_\dsubgr^{\bar{G}}\EGF[\fami(\bar{P})]{\bar{G}}\to
\mathrm{res}_\dsubgr^{\bar{G}}\EGF[\fami(\bar{B})]{\bar{G}}
\end{equation*}
which are unique up to $\dsubgr$--homotopy. 
\end{itemize}

\begin{remark}\label{rem:hY=EFP}
Since $\fami(\Gamma_1,\dots\Gamma_d)=\fami(\bar{P})/\dsubgr$ by
Proposition~\ref{prop:res} we have that the $\dsubgr$--space
$\mathrm{res}^{\bar{G}}_\dsubgr\EGF[\fami(\bar{P})]{\bar{G}}$ is a model for
$\EGF[\fami(\dsubgr_1,\dots,\dsubgr_d)]{\dsubgr}$. Therefore, the $\dsubgr$--map
$\widehat{\hyp}\to \mathrm{res}_\dsubgr^{\bar{G}}\EGF[\fami(\bar{P})]{\bar{G}}$ is in fact a
$\dsubgr$--homotopy equivalence.
\end{remark}

Combining the previous $\dsubgr$--maps with the $\bar{G}$--maps given in
\eqref{eq:G.maps} we have the following commutative diagram (up to equivariant
homotopy)
\begin{equation*}
\xymatrix{
 & E\bar{G}\ar[rr] & & \EGF[\fami(\bar{P})]{\bar{G}}\ar[r]
 & \EGF[\fami(\bar{B})]{\bar{G}}   \\
\hyp\ar[ru]\ar[rr] & & \widehat{\hyp}\ar[rru]_{\psi_{\bar{B}}}\ar[ru]^{\psi_{\bar{P}}} & &   \\
}
\end{equation*}
Taking the quotients by $\PSL{\C}$ and $\dsubgr$ we get the following
commutative diagram 
\begin{equation}\label{eq:cube}
\xymatrix{
 & E\bar{G}\ar[rr]\ar'[d][dd] & & \EGF[\fami(\bar{P})]{\bar{G}}\ar[r]\ar[dd] |!{[dl];[r]}\hole
 & \EGF[\fami(\bar{B})]{\bar{G}}\ar[dd]   \\
\hyp\ar[ru]\ar[rr]\ar[dd] & &
\widehat{\hyp}\ar[rru]_{\psi_{\bar{B}}}\ar[ru]^{\psi_{\bar{P}}}\ar[dd] & &   \\
 & B\bar{G}\ar'[r][rr] & & \BGF[\fami(\bar{P})]{\bar{G}}\ar[r] & \BGF[\fami(\bar{B})]{\bar{G}}  \\
M\ar[ru]^{f}\ar[rr] & &
\widehat{M}\ar[rru]_{\widehat{\psi}_{\bar{B}}}\ar[ru]^{\widehat{\psi}_{\bar{P}}} & & 
}
\end{equation}
where $f=B\bar{\rho}\co B\dsubgr\to B\bar{G}$ is the map between classifying spaces
which on fundamental groups induces the representation
$\bar{\rho}\co\dsubgr\to \PSL{\C}$ of $M$, and $\widehat{\psi}_{\bar{P}}$ and
$\widehat{\psi}_{\bar{B}}$ are given by the
compositions  
\begin{gather*}%\label{eq:comp}
\widehat{\psi}_{\bar{P}}\co\widehat{M}\to \EGF[\fami(\bar{P})]{\bar{G}}/\dsubgr\to
\BGF[\fami(\bar{P})]{\bar{G}},\\
\widehat{\psi}_{\bar{B}}\co\widehat{M}\to \EGF[\fami(\bar{B})]{\bar{G}}/\dsubgr\to
\BGF[\fami(\bar{B})]{\bar{G}},
\end{gather*} 
and they are well-defined up to homotopy.

The maps $\widehat{\psi}_{\bar{P}}$ and $\widehat{\psi}_{\bar{B}}$ induce homomorphisms 
\begin{align*}\label{eq:MtoB}
(\widehat{\psi}_{\bar{P}})_*\co H_3(\widehat{M};\Z)&\to H_3(\BGF[\fami(\bar{P})]{\bar{G}};\Z),\\ 
(\widehat{\psi}_{\bar{B}})_*\co H_3(\widehat{M};\Z)&\to H_3(\BGF[\fami(\bar{B})]{\bar{G}};\Z).
\end{align*}
We denote by $\beta_{\bar{P}}(M)$ and $\beta_{\bar{B}}(M)$ the canonical classes
in $H_3(\BGF[\fami(\bar{P})]{\bar{G}};\Z)$ and $H_3(\BGF[\fami(\bar{B})]{\bar{G}};\Z)$ respectively,
given by 
the images of the fundamental class $[\widehat{M}]$ of $\widehat{M}$
\begin{align*}
\beta_{\bar{P}}(M)&=(\widehat{\psi}_{\bar{P}})_*([\widehat{M}]),\\
\beta_{\bar{B}}(M)&=(\widehat{\psi}_{\bar{B}})_*([\widehat{M}]).
\end{align*}
By the commutativity of the lower triangle in \eqref{eq:cube} we have that
$\beta_{\bar{P}}(M)$ is sent to $\beta_{\bar{B}}(M)$ by the canonical homomorphism from
$H_3(\BGF[\fami(\bar{P})]{\bar{G}};\Z)$ to $H_3(\BGF[\fami(\bar{B})]{\bar{G}};\Z)$.

Thus, by Proposition~\ref{prop:CS.HRGH} we have the following
\begin{theorem}\label{thm:beta.P.B}
Given a complete oriented hyperbolic $3$--manifold of finite volume $M$
we have well-defined invariants
\begin{align*}
\beta_{\bar{P}}(M)&\in H_3([\PSL{\C}:\bar{P}];\Z),\\
\beta_{\bar{B}}(M)&\in H_3([\PSL{\C}:\bar{B}];\Z).
\end{align*}
Moreover, we have that
\begin{equation*}
\beta_{\bar{B}}(M)=(h_{\bar{P}}^{\bar{B}})_*\bigl(\beta_{\bar{P}}(M)\bigr),
\end{equation*}
where $(h_{\bar{P}}^{\bar{B}})_* \co H_n([\bar{G}:\bar{P}];\Z)\to H_n([\bar{G}:\bar{B}];\Z)$ is the homomorphism
described in \eqref{eq:hHK}.
\end{theorem}

\begin{remark}
In general is easier to work with matrices in $\SL{\C}$ than with cosets in $\PSL{\C}$. Let $\rho\colon\dsubgr\to\SL{\C}$
be a lifting of the geometric representation $\bar{\rho}\colon\dsubgr\to\PSL{\C}$.
Notice that in diagram \eqref{eq:cube} we can replace $\bar{G}$ by $G$. Since by
Proposition~\ref{prop:3.IT}
$H_3([G:P];\Z)\cong H_3([\bar{G}:\bar{P}];\Z)$ and 
$H_3([G:B];\Z)\cong H_3([\bar{G}:\bar{B}];\Z)$, by \eqref{eq:hom.diag} we get
the same invariants $\beta_{\bar{P}}(M)$ and $\beta_{\bar{B}}(M)$. That is, if we do the construction starting with a lifting $\rho$
of the geometric representation $\bar{\rho}$, the invariants $\beta_P(M)$ and $\beta_B(M)$ of $M$ only depend on the geometric
representation $\bar{\rho}\co\dsubgr\to\PSL{\C}$ and not on the lifting
$\rho\co\dsubgr\to\SL{\C}$.
In other words, they are independent of the choice a spin structure of $M$.
In order to simplify notation we simply write
$\beta_P(M)$ and $\beta_B(M)$.
\end{remark}

\begin{remark}
The invariants $\beta_P(M)$ and $\beta_B(M)$ extend the invariant
$[M]_{\mathrm{PSL}}$ for $M$ closed in the
following sense: when $M$ is compact $\widehat{M}=M$, by the commutativity of
the lower diagram in \eqref{eq:cube} and by
Remark~\ref{rem:hik*} we have that
\begin{align*} 
(\widehat{\psi}_{\bar{P}})_*&=(h_{\bar{I}}^{\bar{P}})_*\circ f_*,\\
(\widehat{\psi}_{\bar{B}})_*&=(h_{\bar{I}}^{\bar{B}})_*\circ f_*,
\end{align*}
where $(h_{\bar{I}}^{\bar{P}})_*$ and $(h_{\bar{I}}^{\bar{B}})_*$ are the
homomorphisms described in \eqref{eq:hHK}. Thus
\begin{align*}
\beta_P(M)&=(h_{\bar{I}}^{\bar{P}})_*(\fcPSL{M}),\\
\beta_B(M)&=(h_{\bar{I}}^{\bar{B}})_*(\fcPSL{M}).
\end{align*}
\end{remark}

\section{Relation with the extended Bloch group}\label{sec:bloch}

In the present section we recall the definitions of the Bloch and extended Bloch
groups and the Bloch invariant.

\subsection{The Bloch group}

The \emph{pre-Bloch group} $\pB{\C}$ is the abelian group
generated by the formal symbols $[z]$, $z\in \C\setminus\{0,1\}$ subject to
the relation
\begin{equation}\label{eq:5tr}
[x]-[y]+\Bigl[\frac{y}{x}\Bigr]-\Bigl[\frac{1-x^{-1}}{1-y^{-1}}\Bigr]+\Bigl[
\frac{1-x}{1-y}\Bigr]=0,
\quad x\neq y.
\end{equation}
This relation is called the \emph{five term relation}.
By Dupont--Sah \cite[Lemma~5.11]{Dupont-Sah:SCII} we also have the following
relations in $\pB{\C}$
\begin{equation}\label{eq:CR.rel}
[x]=\Bigl[\frac{1}{1-x}\Bigr]=\Bigl[1-\frac{1}{x}\Bigr]=-\Bigl[\frac{1}{x}\Bigr]
=-\Bigl[\frac{x}{x-1}\Bigr]
=-[1-x].
\end{equation}
Using this relations it is possible to extend the definition of $[x]\in\pB{\C}$
allowing $x\in\widehat{\C}$
and removing the restriction $x\neq y$ in \eqref{eq:5tr}. This is equivalent
\cite[after Lemma~5.11]{Dupont-Sah:SCII} to define $\pB{\C}$ as the abelian
group generated by the symbols $[z]$, $z\in\widehat{\C}$ subject to the
relations 
\begin{gather*}
[0]=[1]=[\infty]=0,\\
[x]-[y]+\Bigl[\frac{y}{x}\Bigr]-\Bigl[\frac{1-x^{-1}}{1-y^{-1}}\Bigr]+\Bigl[
\frac{1-x}{1-y}\Bigr]=0.
\end{gather*}

\begin{remark}
If we consider the first definition of the pre-Bloch group where for the
generators $[z]$ of $\pB{\C}$ we only allow $z$ to be in $\C\setminus\{0,1\}$,
the pre-Bloch group can be interpreted as a homology group. Consider the subcomplex
$\CC{C}^{\neq}(\Xsubgr{B}))=\CC{C}^{\neq}(\widehat{\C}))$ of tuples of distinct elements of $\widehat{\C}$ (see Remark~\ref{rem:H.neq}).

The action of $\bar{G}$ on $\widehat{\C}$ by fractional linear transformations (see
\S\ref{ssec:B}) is not only transitive but triply transitive, that is, given
four distinct points $z_0$, $z_1$, $z_2$, $z_3$ in $\widehat{\C}$, there exists
an element $g\in\PSL{\C}$ such that
\begin{equation*}
g\cdot z_{0}=0,\quad g\cdot z_{1}=\infty,\quad g\cdot z_{2}=1,\quad
g\cdot z_{3}=z
\end{equation*}
where $z=[z_{0}:z_{1}:z_{2}:z_{3}]$ is the \emph{cross-ratio} of $z_0$, $z_1$,
$z_2$, $z_3$ given by
 \begin{equation}\label{eq:X.R}
 [z_{0}:z_{1}:z_{2}:z_{3}]=
 \frac{(z_{0}-z_{3})(z_{1}-z_{2})}{(z_{0}-z_{2})(z_{1}-z_{3})}. 
 \end{equation}
In other words, the orbit of a $4$--tuple $(z_0,z_1,z_2,z_3)$ of distinct points
in $\widehat{\C}$ under the diagonal action of $\bar{G}$ is determined by its
cross-ratio. Thus, we get a surjective homomorphism given by the cross-ratio
\begin{equation}
\begin{split}\label{eq:sim.psiB.neq}
 \sigma\co \CC[3]{B}^{\neq}(\Xsubgr{\bar{B}})=\CC[3]{B}^{\neq}(\widehat{\C})&\to \pB{\C}\\
  (z_0,z_1,z_2,z_3)_G&\mapsto [z_{0}:z_{1}:z_{2}:z_{3}]
\end{split}
\end{equation}
where $(z_0,z_1,z_2,z_3)_G$ denotes the $G$--orbit of the generator $(z_0,z_1,z_2,z_3)\in\CC[3]{C}^{\neq}(\Xsubgr{\bar{B}})$.
It is easy to see that the five term relation \eqref{eq:5tr} is equivalent to the relation
\begin{equation*}
 \sum_{i=0}^4(-1)^i[z_0:\dots:\widehat{z_i}:\dots:z_4]=0.
\end{equation*}
Also by the triply transitivity of the action of $\bar{G}$ on $\widehat{\C}$ we have that
$\CC[2]{B}(\Xsubgr{\bar{B}})\cong\Z$ and $\CC[3]{B}(\Xsubgr{\bar{B}})$ consists only of
cycles. Thus $\sigma$ induces an isomorphism (compare with \cite[Lemma~2.2]{Suslin:K3FBG})
\begin{equation}
H_3(\CC{B}^{\neq}(\Xsubgr{\bar{B}})\cong\pB{\C}.
\end{equation}

Also using this definition of the pre-Bloch group it is possible to prove that
it is isomorphic to the corresponding Takasu relative homology group:
\begin{proposition}
 \begin{equation*}
  H_3(G,B;\Z)\cong\pB{\C}.
 \end{equation*}
\end{proposition}

\begin{proof}
By \cite[Theorem~2.2]{Takasu:RHRCTG} we have that $H_3(G,B;\Z)\cong
H_2(G,I_{(G,B)}(\Z))$ and by \cite[(A27), (A28)]{Dupont-Sah:SCII} we also have
$H_2(G,I_{(G,B)}(\Z))\cong\pB{\C}$.
\end{proof}
\end{remark}

\begin{remark}
If we extend the definition of the cross-ratio to
$[z_{0}:z_{1}:z_{2}:z_{3}]=0$ whenever $z_{i}=z_{j}$ for some $i\neq j$, we can extend the homomorphism $\sigma$
to a homomorphism
\begin{equation}
\begin{split}\label{eq:sim.psiB}
 \sigma\co \CC[3]{B}(\Xsubgr{\bar{B}})=\CC[3]{B}(\widehat{\C})&\to \pB{\C}\\
  (z_0,z_1,z_2,z_3)_G&\mapsto [z_{0}:z_{1}:z_{2}:z_{3}]
\end{split}
\end{equation}
where $(z_0,z_1,z_2,z_3)_G$ denotes the $G$--orbit of the $3$--simplex
$(z_0,z_1,z_2,z_3)\in\CC[3]{C}(\Xsubgr{\bar{B}})$. Using the relations \eqref{eq:CR.rel} one can prove
that the boundaries in $\CC[3]{B}(\Xsubgr{\bar{B}})$ go to $0$ in $\pB{\C}$ under $\sigma$, so it extends to
a homomorphism
\begin{equation}\label{eq:pB.H3B}
\sigma\co H_3([\bar{G}:\bar{B}];\Z)= H_3(\CC{B}(\Xsubgr{\bar{B}}))\to\pB{\C}.
\end{equation}
By direct computation one can see that the extra cycles in $\CC[3]{C}(\Xsubgr{\bar{B}})$ are also boundaries and therefore
the homomorphism \eqref{eq:pB.H3B} is also an isomorphism.
\end{remark}

The \emph{Bloch group} $\Bg{\C}$ is the kernel of the map
\begin{equation}\label{eq:lambda}
\begin{split}
\nu\co\pB{\C}&\to \antp{\C^\times}\\
[z]&\mapsto \Asym{z}{(1-z)}.
\end{split}
\end{equation}

\subsection{The Bloch invariant}

An \emph{ideal simplex} is a geodesic $3$--simplex in $\hypc$ whose vertices
$z_0$, $z_1$, $z_2$, $z_3$ are all in $\partial\hyp=\hat{\C}$. We consider the
vertex ordering as part of the data defining an ideal
simplex. By the triply transitivity of the action of $G$ on $\hyp$ the
orientation-preserving congruence class of an ideal simplex with vertices $z_0$,
$z_1$, $z_2$, $z_3$ is given by the cross-ratio $z=[z_{0}:z_{1}:z_{2}:z_{3}]$.
An ideal simplex is \emph{flat} if and only if the cross-ratio is real, and if
it is not flat, the orientation given by the vertex ordering agrees with the
orientation inherited from $\hyp$
if and only if the cross-ratio has positive imaginary part.

From \eqref{eq:X.R} we have that an even (\ie orientation preserving) permutation
of the $z_i$ replaces $z$ by one
of three so-called \emph{cross-ratio parameters},
\begin{equation*}
 z,\qquad z'=\frac{1}{1-z},\qquad z''=1-\frac{1}{z},
\end{equation*}
while an odd (\ie orientation reversing) permutation replaces $z$ by
\begin{equation*}
 \frac{1}{z},\qquad \frac{z}{z-1},\qquad 1-z.
\end{equation*}
Thus, by the relations \eqref{eq:CR.rel} in $\pB{\C}$ we can consider the
pre-Bloch group as being
generated by (congruence classes) of oriented ideal simplices.

Let $M$ be a non-compact orientable complete hyperbolic $3$--manifold of finite
volume. An \emph{ideal
triangulation} for $M$ is a triangulation where all the tetrahedra are ideal
simplices.

Let $M$ be an hyperbolic $3$--manifold and let $\suce{\triangle}$ be the
ideal simplices of an ideal triangulation of $M$. Let $z_{i}\in\C$ be
the parameter of $\triangle_{i}$  for each $i$. These parameters define
an element $\beta(M)=\sum_{i=1}^{n}[z_{i}]$ in the pre-Bloch group.
The element $\beta(M)\in\pB{\C}$ is called
the \emph{Bloch invariant of $M$}, it was defined by Neumann and Yang in \cite{Neumann-Yang:BIHM}, where they
proved that it does not depend on the ideal triangulation.

\begin{remark}
Neumann and Yang defined the Bloch invariant using \emph{degree one
ideal triangulations}, in that way it is defined for all hyperbolic
$3$--manifolds of finite volume, even the compact ones, see Neumann--Yang
\cite[\S2]{Neumann-Yang:BIHM} for details.
\end{remark}

In \cite[Theorem~1.1]{Neumann-Yang:BIHM} it is proved that the Bloch invariant
lies in the Bloch group $\Bg{\C}$. An alternative proof of this fact is given in
Cisneros-Molina--Jones \cite[Cor.~8.7]{Cisneros-Jones:Bloch}.

\begin{remark}
By \eqref{eq:pB.H3B} we have that $H_3([G:B];\Z)\cong\pB{\C}$ and in
\cite[Theorem~6.1]{Cisneros-Jones:Bloch} it is proved that $\beta_B(M)$ is
precisely the Bloch invariant $\beta(M)$ of $M$, see
Subsection~\ref{ssec:bloch.inv}.
\end{remark}

\subsection{The extended Bloch group}

Given a complex number $z$ we use the convention that its argument $\Arg z$
always denotes its
main argument $-\pi<\Arg z\leq \pi$ and $\Log z$ always denotes a fixed branch
of logarithm, for instance,
the principal branch having
$\Arg z$ as imaginary part.

Let $\triangle$ be an ideal simplex with cross-ratio $z$. A \emph{flattening} of
$\triangle$ is a triple of complex numbers of the form
\begin{equation*}
 (w_0,w_1,w_2)=(\Log z+p\pi i,-\Log(1-z)+q\pi i,\Log(1-z)-\Log z-p\pi i-q\pi i)
\end{equation*}
with $p,q\in\Z$. The numbers $w_0$ , $w_1$ and $w_2$ are called \emph{log
parameters} of $\triangle$. Up to multiples of $\pi i$, the log parameters are
logarithms of the cross-ratio parameters. 

\begin{remark}\label{rem:zpq}
The log parameters uniquely determine $z$. Hence we can write a flattening as
$[z; p, q]$. Note that this notation depends on the choice of logarithm branch \cite[Lemma~3.2]{Neumann:EBGCCSC}.
\end{remark}

Following \cite{Neumann:EBGCCSC} we assign cross-ratio parameters and log
parameters to the
edges of a flattened ideal simplex as indicated in Figure~\ref{fig:parameter}.
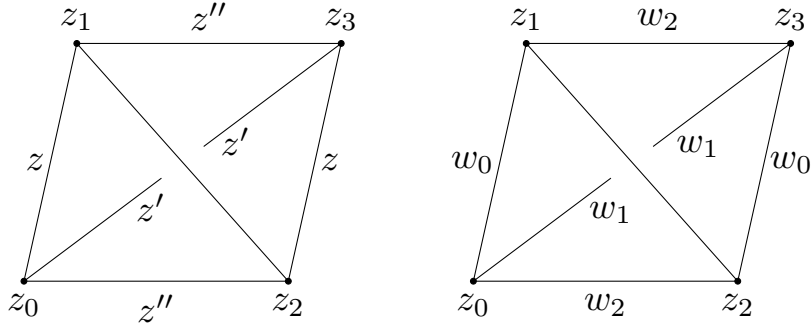
\begin{figure}[ht]
\begin{center}
\begin{tikzpicture}
\begin{scope}
\filldraw (0pt,0pt) circle (1.1pt)
          (100pt,0pt) circle (1.1pt)
          (120pt,90pt) circle (1.1pt)
          (20pt,90pt) circle (1.1pt);
\draw (0pt,0pt) -- node[scale=1.4,below]{$z''$} (100pt,0pt) --
node[scale=1.4,right=-2pt]{$z$} (120pt,90pt) -- node[scale=1.4,above]{$z''$}
(20pt,90pt) -- node[scale=1.4,left=-2pt]{$z$} (0pt,0pt) -- (120pt,90pt);
\fill[white] (60pt,45pt) circle (10pt);
\draw (100pt,0pt) -- (20pt,90pt);
\draw (0pt,0pt) node[scale=1.4,below] {$z_0$}
      (100pt,0pt) node[scale=1.4,below] {$z_2$}
      (120pt,90pt) node[scale=1.4,above] {$z_3$}
      (20pt,90pt) node[scale=1.4,above] {$z_1$}
      (48pt,28pt) node[scale=1.4] {$z'$}
      (80pt,53pt) node[scale=1.4] {$z'$};
\end{scope}
\begin{scope}[xshift=170pt]
\filldraw (0pt,0pt) circle (1.1pt)
          (100pt,0pt) circle (1.1pt)
          (120pt,90pt) circle (1.1pt)
          (20pt,90pt) circle (1.1pt);
\draw (0pt,0pt) -- node[scale=1.4,below]{$w_2$} (100pt,0pt) --
node[scale=1.4,right=-2pt]{$w_0$} (120pt,90pt) -- node[scale=1.4,above]{$w_2$}
(20pt,90pt) -- node[scale=1.4,left=-2pt]{$w_0$} (0pt,0pt) -- (120pt,90pt);
\fill[white] (60pt,45pt) circle (10pt);
\draw (100pt,0pt) -- (20pt,90pt);
\draw (0pt,0pt) node[scale=1.4,below] {$z_0$}
      (100pt,0pt) node[scale=1.4,below] {$z_2$}
      (120pt,90pt) node[scale=1.4,above] {$z_3$}
      (20pt,90pt) node[scale=1.4,above] {$z_1$}
      (52pt,26pt) node[scale=1.4] {$w_1$}
      (85pt,51pt) node[scale=1.4] {$w_1$};
\end{scope}
\end{tikzpicture}
\end{center}
\caption{Cross-ratio and log parameters of a flattened ideal
simplex}\label{fig:parameter}
\end{figure}

Let $z_0$, $z_1$, $z_2$, $z_3$ and $z_4$ be five distinct points in
$\widehat{\C}$ and let $\triangle_i$
denote the ideal simplices $(z_0,\dots,\widehat{z_i},\dots,z_4)$. Let
$(w_0^i,w_1^i,w_2^i)$ be flattenings of the simplices $\triangle_i$. Every edge
$[z_i z_j]$ belongs to exactly three of the $\triangle_i$ and therefore has
three associated log parameters. The flattenings are said to satisfy the
\emph{flattening condition} if for each edge the signed sum of the three
associated log parameters is zero. The sign is positive
if and only if $i$ is even.

From the definition we have that the flattening condition is equivalent to the
following ten equations:
\begin{equation}\label{eq:flat.cond}
\begin{aligned}[]
[z_0z_1]&: & w_0^2-w_0^3+w_0^4&=0 &  [z_0z_2]&: &-w_0^1-w_2^3+w_2^4&=0\\
[z_1z_2]&: & w_0^0-w_1^3+w_1^4&=0 &  [z_1z_3]&: & w_2^0+w_1^2+w_2^4&=0\\
[z_2z_3]&: & w_1^0-w_1^1+w_0^4&=0 &  [z_2z_4]&: & w_2^0-w_2^1-w_0^3&=0\\
[z_3z_4]&: & w_0^0-w_0^1+w_0^2&=0 &  [z_3z_0]&: &-w_2^1+w_2^2+w_1^4&=0\\
[z_4z_0]&: &-w_1^1+w_1^2-w_1^3&=0 &  [z_4z_1]&: & w_1^0+w_2^2-w_2^3&=0
\end{aligned}
\end{equation}

The \emph{extended pre-Bloch group} $\EpB{\C}$ is the free abelian group
generated by flattened ideal simplices subject to the relations:
\begin{description}
 \item[Lifted five term  relation] 
\begin{equation}\label{eq:l5tr}
 \sum_{i=0}^4(-1)^i(w_0^i,w_1^i,w_2^i)=0,
\end{equation}
if the flattenings satisfy the flattening condition.
 
\item[Transfer relation] 
\begin{equation*}
 [z;p,q]+[z;p',q']=[z;p,q']+[z;p',q].
\end{equation*}
\end{description}

The \emph{extended Bloch group} $\EBg{\C}$ is the kernel of the homomorphism
\begin{align*}
\widehat{\nu}\co\EpB{\C}&\to\antp{\C}\\
(w_0,w_1,w_2)&\mapsto w_0\wedge w_1.
\end{align*}

\begin{theorem}[{\cite[Theorem~2.6]{Neumann:EBGCCSC}}]\label{thm:EBG.HPSL}
There is an isomorphism
\begin{equation*}
 H_3(\PSL{\C};\Z)\cong \EBg{\C}.
\end{equation*}
\end{theorem}

\section{Computing $\beta_P(M)$ and $\beta_B(M)$ using an ideal triangulation of $M$}\label{sec:compute}

Let $M$ be a non-compact orientable complete hyperbolic $3$--manifold of
finite volume.
Let $\pi\co\hyp\to\dsubgr\backslash\hyp=M$ be the universal cover of
$M$. Then $M$ lifts to an exact, convex, fundamental, ideal polyhedron
$\mathsf{P}$ for $\dsubgr$ \cite[Theorem~11.2.1]{Ratcliffe:HypMan}. An ideal
triangulation of $M$ gives a decomposition of $\mathsf{P}$ into a finite number
of ideal tetrahedra  $(z_0^i,z_1^i,z_2^i,z_3^i)$, $i=1,\dots,n$. Since
$\mathcal{P}=\sets{g\mathsf{P}}{g\in\dsubgr}$ is an
exact tessellation of $\hyp$ \cite[Theorem~6.7.1]{Ratcliffe:HypMan}, this
decomposition of $\mathsf{P}$ gives an ideal triangulation of $\hyp$.

As in Subsection~\ref{ssec:fc} let $\mathcal{C}$ be the set of fixed points of
parabolic elements of $\dsubgr$ in $\partial\hyp=\widehat{\C}$ and consider
$\widehat{\hyp}=\hyp\cup\mathcal{C}$, which is the
result of adding the vertices of the ideal tetrahedra of the ideal
triangulation of $\hyp$.  Hence we can consider $\widehat{\hyp}$ as a
simplicial complex with $0$--simplices given by the elements of
$\mathcal{C}\subset\widehat{\C}$. 
The action of $\bar{G}$ on $\widehat{\hyp}$ induces an action of $\bar{G}$ on the tetrahedra of
the ideal triangulation of $\widehat{\hyp}$.
Taking the quotient of $\widehat{\hyp}$ by $\dsubgr$ we obtain $\widehat{M}$ and we
get an extension of the covering map $\pi$ to a map $\widehat{\pi}\co
\widehat{\hyp}\to\widehat{M}$.
The $\dsubgr$--orbits $(z_0,z_1,z_2,z_3)_\dsubgr$ of the tetrahedra
$(z_0,z_1,z_2,z_3)$ of the ideal triangulation of $\widehat{\hyp}$ correspond to
the tethahedra of the ideal triangulation of $\widehat{M}$. The $\dsubgr$--orbit
set $\widehat{\mathcal{C}}$ of $\mathcal{C}$ corresponds to the cusps points of
$M$, we suppose that $M$ has $d$ cups, so the cardinality of
$\widehat{\mathcal{C}}$ is $d$.

\subsection{Computation of $\beta_B(M)$}\label{ssec:bloch.inv}

Using the simplicial construction  of $\EGF{\bar{G}}$ given by
Proposition~\ref{prop:const3} we have that a model for $\EGF[\fami(\bar{B})]{\bar{G}}$ is
the geometric realization of the simplicial set whose $n$--simplices are the
ordered $(n+1)$--tuples $(z_0,\dots,z_n)$ of elements of
$\Xfami[\fami(\bar{B})]=\widehat{\C}$ and the $i$--th face (respectively, degeneracy)
of such a simplex is obtained by omitting (respectively, repeating) $z_i$. The
action of $\bar{g}\in \bar{G}$ on an $n$--simplex $(z_0,\dots,z_n)$ gives the simplex
$(\bar{g}z_0,\dots,\bar{g}z_n)$.

Considering $\widehat{\hyp}$ as the
geometric realization of its ideal triangulation and since its
vertices are elements in $\mathcal{C}\subset\widehat{\C}$ we have that the
$\dsubgr$--map $\psi_{\bar{B}}\co\widehat{\hyp}\to \EGF[\fami(\bar{B})]{\bar{G}}$ in diagram
\eqref{eq:cube} is given by the (geometric
realization of the) $\dsubgr$--equivariant simplicial map
\begin{align*}
\widehat{\hyp}&\xrightarrow{\psi_{\bar{B}}} \EGF[\fami(\bar{B})]{\bar{G}}\\
(z_0,z_1,z_2,z_3)&\mapsto (z_0,z_1,z_2,z_3).
\end{align*}
This induces the map $\widehat{\psi}_B\co\widehat{M}\to\BGF[\fami(\bar{B})]{\bar{G}}$ in
diagram \eqref{eq:cube}. Furthermore, this induces on simplicial $3$--chains the
homomorphism $(\widehat{\psi}_{\bar{B}})_*\co C_3(\widehat{M})\to
\CC[3]{B}(\BGF[\fami(\bar{B})]{\bar{G}})=\CC[3]{B}(\Xsubgr{\bar{B}})$ (see
\cite[Remark~4.25]{Arciniega-Cisneros:CRGCH}) which we can compose with homomorphism
\eqref{eq:sim.psiB} to get
\begin{gather}
C_3(\widehat{M})=C_3(\widehat{\hyp})_\dsubgr\xrightarrow{(\widehat{\psi}_{\bar{B}})_*}
\CC[3]{B}(\Xsubgr{\bar{B}})\xrightarrow{\sigma}\pB{\C}\notag\\
(z_0,z_1,z_2,z_3)_\dsubgr\mapsto
(z_0,z_1,z_2,z_3)_{\bar{G}}\to [z_{0}:z_{1}:z_{2}:z_{3}]\label{eq:comp.B}.
\end{gather}
where $(z_0,z_1,z_2,z_3)_\dsubgr$ (resp. $(z_0,z_1,z_2,z_3)_{G}$)
denotes the $\dsubgr$--orbit (respectively $\bar{G}$--orbit) of the 
$3$--simplex $(z_0,z_1,z_2,z_3)$ in $C_3(\widehat{\hyp})$ (respectively in
$\CC[3]{B}(\Xsubgr{\bar{B}})$) and $[z_{0}:z_{1}:z_{2}:z_{3}]$ is the cross-ratio
parameter of the ideal tetrahedron $(z_0,z_1,z_2,z_3)$.

Let $(z_0^i,z_1^i,z_2^i,z_3^i)_\dsubgr$, $i=1,\dots,n$, be
the ideal tetrahedra of the ideal triangulation of $\widehat{M}$ and let
$z_{i}=[z_0^i:z_1^i:z_2^i:z_3^i]\in\C$ be the cross-ratio parameter of the ideal
tetrahedron
$(z_0^i,z_1^i,z_2^i,z_3^i)$ for each $i$. Then the image of the fundamental
class  
$[\widehat{M}]$ under the homomorphism in homology given by \eqref{eq:comp.B} is
given by
\begin{gather*}
\homo[3](\widehat{M})\xrightarrow{(\widehat{\psi}_{\bar{B}})_*}
H_3([\bar{G}:\bar{B}];\Z)\xrightarrow[\cong]{\sigma}\pB{\C}\\
[\widehat{M}]=\Bigl[\sum_{i=1}^n(z_0^i,z_1^i,z_2^i,z_3^i)_\dsubgr\Bigr]\mapsto
\Bigl[\sum_{i=1}^n(z_0^i,z_1^i,z_2^i,z_3^i)_{\bar{G}}\Bigr]\mapsto\sum_{i=1}^n[z_i],
\end{gather*}
and we have that the invariant $\beta_B(M)$ under the isomorphism $\sigma$
corresponds to the Bloch invariant $\beta(M)$, see Cisneros-Molina--Jones
\cite[Theorem~6.1]{Cisneros-Jones:Bloch}.

\subsection{Computation of $\beta_P(M)$}\label{ssec:beta.P}

We want to give a simplicial description of the $\dsubgr$--map
$\psi_{\bar{P}}\co\widehat{\hyp}\to\EGF[\fami(\bar{P})]{\bar{G}}$ in diagram \eqref{eq:cube}. For this,
we also use the Simplicial Construction of Proposition~\ref{prop:const3} to give
a model for $\EGF[\fami(\bar{P})]{\bar{G}}$ as the geometric realization of the simplicial
set whose $n$--simplices are the ordered $(n+1)$--tuples %$(A_0,\dots,A_n)$ 
of elements of $\Xfami[\fami(\bar{P})]=\Xsubgr{\bar{P}}$ (or $\Xfami[\fami(\bar{P})]=\Xspb$) (see \cite[Remark~4.25]{Arciniega-Cisneros:CRGCH}
and Subsection~\ref{ssec:XP}). The $i$--th face
(respectively, degeneracy) of such a simplex is obtained by omitting
(respectively, repeating) the $i$--th element. The action of $\bar{g}\in \bar{G}$ on an
$n$--simplex is the diagonal action. 

Since the vertices of $\widehat{\hyp}$ are elements in
$\mathcal{C}\subset\widehat{\C}$, to give a simplicial description of the
$\dsubgr$--map $\psi_{\bar{P}}\co\widehat{\hyp}\to\EGF[\fami(\bar{P})]{\bar{G}}$ is enough to give a
$\dsubgr$--map
\begin{equation*}
 \Phi\co\mathcal{C}\to\Xsubgr{\bar{P}},\qquad\text{or}\qquad \Phi\co\mathcal{C}\to\Xspb
\end{equation*}
and define
\begin{equation}\label{eq:hYtoEGFP}
\begin{aligned}
\widehat{\hyp}&\xrightarrow{\psi_{\bar{P}}} \EGF[\fami(\bar{P})]{\bar{G}}\\
(z_0,z_1,z_2,z_3)&\mapsto (\Phi(z_0),\Phi(z_1),\Phi(z_2),\Phi(z_3)).
\end{aligned}
\end{equation}

For $i=1,\dots,d$, every cusp point $\hat{c}_i\in\widehat{\mathcal{C}}$ of $M$
corresponds to a $\dsubgr$--orbit of $\mathcal{C}$. Choose $c_i\in\mathcal{C}$
in the $\dsubgr$--orbit corresponding to $\hat{c}_i\in\widehat{\mathcal{C}}$.
Now choose an element $\V_i\in(\bar{h}_{\bar{P}}^{\bar{B}})^{-1}(c_i)\subset\Xsubgr{\bar{P}}$ (or
$A_i\in(\bar{h}_{\bar{P}}^{\bar{B}})^{-1}(c_i)\subset\Xspb$) and define
\begin{equation}\label{eq:Phi}
\begin{aligned}
\Phi\co\mathcal{C}&\to\Xsubgr{\bar{P}} & &\text{or}&\Phi\co\mathcal{C}&\to\Xspb\\
c_i &\mapsto \V_i & &&  c_i &\mapsto A_i
\end{aligned}
\end{equation}
and extend $\dsubgr$--equivariantly by
\begin{equation*}
\Phi(\bar{g}\cdot c_i)=\bar{g}\cdot\V_i,\qquad\text{or}\qquad \Phi(\bar{g}\cdot c_i)=\bar{g}A_i\bar{g}^T.
\end{equation*}

\begin{remark}\label{rem:ci.vi.Ai}
Suppose that for every cusp point $\hat{c}_i\in\widehat{\mathcal{C}}$ we have
chosen $c_i\in\mathcal{C}\subset\widehat{\C}$ in the $\dsubgr$--orbit
corresponding to $\hat{c}_i$. Using homogeneous coordinates we can write
\begin{equation*}
 c_i=[z_i:w_i].
\end{equation*}
So, one way to choose $\V_i$ (or $A_i$) is given by
\begin{equation*}
\V_i=[z_i,w_i],\qquad\text{or}\qquad A_i=\begin{pmatrix}
                                         z_i^2 & z_iw_i\\ z_iw_i & w_i^2
                                         \end{pmatrix}.
\end{equation*}
\end{remark}

\begin{remark}\label{rem:Phi.ind}
The $\dsubgr$--isotropy subgroups of the $c_i$ are conjugates of the peripheral
subgroups $\dsubgr_i$ and they consist of parabolic elements, that is, elements
in a conjugate of $P$, and by Remark~\ref{rem:im.inv} they fix pointwise the
elements of $(\bar{h}_{\bar{P}}^{\bar{B}})^{-1}(c_i)$. On the other hand, by
Remark~\ref{rem:hypebolic.no.fix.c} no hyperbolic element in $\dsubgr$ has as
fixed point any point in $\mathcal{C}$. Therefore $\Phi$ is a well-defined
$\dsubgr$--equivariant map and the map $\psi_{\bar{P}}$ in \eqref{eq:hYtoEGFP} is a
well-defined $\dsubgr$--equivariant map. By property \eqref{it:G-htpy} of $\EGF[\fami(\bar{P})]{\bar{G}}$ any two such $\dsubgr$--maps are
$\dsubgr$--homotopic, so $\psi_{\bar{P}}$ is independent of the choice of the $c_i$ 
and their corresponding $\V_i$ (or $A_i$) up to $\dsubgr$--homotopy.
This induces the map $\widehat{\psi}_{\bar{P}}\co\widehat{M}\to\BGF[\fami(\bar{P})]{\bar{G}}$ in
diagram \eqref{eq:cube} and choosing different $c_i$ and $A_i$ we obtain
homotopic maps. Thus, this induces a canonical homomorphism in homology
$(\widehat{\psi}_{\bar{P}})_*\co \homo[3](\widehat{M})\to H_3([\bar{G}:\bar{P}];\Z)$ independent of
choices.
\end{remark} 

Let $(z_0^i,z_1^i,z_2^i,z_3^i)_\dsubgr$, $i=1,\dots,n$, be
the ideal tetrahedra of the ideal triangulation of $\widehat{M}$.
The image of the fundamental class $[\widehat{M}]$ under $(\widehat{\psi}_P)_*$
is given by
\begin{align}
\homo[3](\widehat{M})&\xrightarrow{(\widehat{\psi}_P)_*} H_3([\bar{G}:\bar{P}];\Z)\notag\\
[\widehat{M}]=\Bigl[\sum_{i=1}^n(z_0^i,z_1^i,z_2^i,z_3^i)_\dsubgr\Bigr]&\mapsto
\beta_P(M)=\Bigl[\sum_{i=1}^n\bigl(\Phi(z_0^i),\Phi(z_1^i),\Phi(z_2^i),
\Phi(z_3^i)\bigr)_{\bar{G}}\Bigr]\label{eq:form.beta.P}
\end{align}
obtaining an explicit formula for the invariant $\beta_P(M)$. 

\begin{remark}\label{rem:Fphi}
Notice that the cycle
\begin{equation}\label{eq:F.cycle}
\sum_{i=1}^n\bigl(\Phi(z_0^i),\Phi(z_1^i),\Phi(z_2^i),\Phi(z_3^i)\bigr)_{\bar{G}}
\end{equation}
can also be seen as a cycle in $\CC[3]{B}^{h_{\bar{P}}^{\bar{B}\neq}}(\Xsubgr{\bar{P}})$
and by Proposition~\ref{prop:tak.hUB.sub} it represents a class
\begin{equation}
 F_\Phi(M)\in H_3(\PSL{\C},\bar{P};\Z)
\end{equation}
which depends on the choice of the $\dsubgr$-map $\Phi$ in \eqref{eq:Phi}. By Corollary~\ref{cor:sc.c} 
the inclusion $\CC[3]{C}^{h_H^K\neq}(G/H)\to\CC[3]{C}(G/H)$ induces the canonical homomorphism
$H_3(\bar{G},\bar{P};\Z)\to H_3([\bar{G}:\bar{P}];\Z)$; under this homomorphism $F_\Phi(M)$ is mapped to
$\beta_P(M)$, so its image is independent of the choice of the map $\Phi$.
\end{remark}

\section{The homomorphism $\widehat{\sigma}\co H_3(\PSL{\C},\bar{P};\Z)\to\EBg{\C}$}\label{sec:maps}

In \cite{Zickert:VCSIR} Zickert defined a homomorphism
\begin{equation*}
 \Psi\co H_3(\PSL{\C},\bar{P};\Z)\to \EBg{\C},
\end{equation*}
which by Neumann's isomorphism $H_3(\PSL{\C};\Z)\cong\EBg{\C}$ (Theorem~\ref{thm:EBG.HPSL}) corresponds to the homomorphism \eqref{eq:Psi} mentioned in the Introduction, used to define the class $[M]_{PSL}\in H_3(\PSL{\C};\Z)$ for an hyperbolic $3$-manifold with cusps. 
To define this homomorphism, Zickert uses a complex of truncated simplices to compute the Takasu relative homology groups $H_\bullet(\PSL{\C},\bar{P};\Z)$
(see \S\ref{ssec:cts} below).

In this section we use Proposition~\ref{prop:tak.hUB.sub} to compute $H_\bullet(\PSL{\C},\bar{P};\Z)$ and 
following ideas in Dupont and Zickert \cite[\S3]{Dupont-Zickert:DFCCSC} we define a homomorphism
\begin{equation*}
 H_3(\bar{G},\bar{P};\Z)\to\EBg{\C},
\end{equation*}
which corresponds to Zickert's homomorphism $\Psi$ (Proposition~\ref{prop:Psi=tildesigma} below). 
We learned that this map was first defined by Inoue and Kabaya in \cite[\S7.1]{Inoue-Kabaya:QHCV} in the context of quandle homology.

We use the models for $\Xsubgr{U}$, $\Xsubgr{P}$, $\Xsubgr{B}$ and the explicit
$G$-maps between them described in Section~\ref{sec:GUPB}. We simplify notation
by setting 
\begin{align*}
h_U&=h_U^B\co\Xsubgr{U}\to\Xsubgr{B},\\
h_P&=h_P^B\co\Xsubgr{P}\to\Xsubgr{B}.
\end{align*}
Consider the $h_H$--subcomplexes with $H=U,P,B$:  
$\CC{C}^{h_U\neq}(\Xsubgr{U})$, $\CC{C}^{h_P\neq}(\Xsubgr{P})$ and
$\CC{C}^{\neq}(\Xsubgr{B})$ defined in Subsection~\ref{sssec:maps}.

Since $h_U^P$, $h_U$ and $h_P$ are $G$--equivariant they induce maps 
\begin{align*}
(h_U^P)_{*}\co\CC{C}^{h_U\neq}(\Xsubgr{U})&\to\CC{C}^{h_P\neq}(\Xsubgr{P}),\\
(h_U)_{*}\co\CC{C}^{h_U\neq}(\Xsubgr{U})&\to\CC{C}^{\neq}(\Xsubgr{B}),\\
(h_P)_{*}\co\CC{C}^{h_P\neq}(\Xsubgr{P})&\to\CC{C}^{\neq}(\Xsubgr{B}),
\end{align*}

We start defining a homomorphism $\CC[3]{C}^{h_U\neq}(\Xsubgr{U})\to\EpB{\C}$
which descends to a homomorphism 
\begin{equation*}
 \CC[3]{C}^{h_P\neq}(\Xsubgr{P})\to\EpB{\C}.
\end{equation*}

We assign to each $4$--tuple $(v_0,v_1,v_2,v_3)\in
\CC[3]{C}^{h_U\neq}(\Xsubgr{U})$ a combinatorial flattening of the ideal simplex
$\bigl(h_U(v_0),h_U(v_1),h_U(v_2),h_U(v_3)\bigr)$ in such a way that the
combinatorial flattenings assigned to tuples
$(v_0,\dots,\widehat{v_i},\dots,v_4)$ satisfy the flattening condition.

Given $v_i=(v_i^1,v_i^2)\in\Xsubgr{U}$ we denote by
\begin{equation*}
 \det(v_i,v_j)=\det\begin{pmatrix}
		    v_i^1 & v_i^2\\
		    v_j^1 & v_j^2
                    \end{pmatrix}=v_i^1v_j^2-v_i^2v_j^1.
\end{equation*}
Let $(v_0,v_1,v_2,v_3)\in \CC[3]{C}^{h_U\neq}(\Xsubgr{U})$. As was noticed in
\cite[Section~3.1]{Dupont-Zickert:DFCCSC}, the cross-ratio parameters $z$,
$\frac{1}{1-z}$ and $\frac{z-1}{z}$ of the simplex
$\bigl(h_U(v_0),h_U(v_1),h_U(v_2),h_U(v_3)\bigr)$ can be expressed in terms of
determinants
\begin{align}
 z&=\bigl[h_U(v_0):h_U(v_1):h_U(v_2):h_U(v_3)\bigr]
=\frac{\Bigl(\frac{v_0^1}{v_0^2}-\frac{v_3^1}{v_3^2}\Bigr)
\Bigl(\frac{v_1^1}{v_1^2} -\frac {v_2^1}{v_2^2}\Bigr)}
{\Bigl(\frac{v_0^1}{v_0^2}-\frac{v_2^1}{v_2^2}\Bigr)
\Bigl(\frac{ v_1^1}{ v_1^2}-\frac{v_3^1}{v_3^2}\Bigr)}
=\frac{\det(v_0,v_3)\det(v_1,v_2)}{\det(v_0,v_2)\det(v_1,v_3)},
\label{eq:z}\\[5pt]
 \frac{1}{1-z}&=\bigl[h_U(v_1):h_U(v_2):h_U(v_0):h_U(v_3)\bigr]
=\frac{\Bigl(\frac{v_1^1}{v_1^2}-\frac{v_3^1}{v_3^2}\Bigr)
\Bigl(\frac{v_2^1}{v_2^2} -\frac { v_0^1}{v_0^2}\Bigr)}
{\Bigl(\frac{v_1^1}{v_1^2}-\frac{v_0^1}{v_0^2}\Bigr)
\Bigl(\frac{v_2^1} { v_2^2} -\frac{v_3^1}{v_3^2}\Bigr)}
=\frac{\det(v_1,v_3)\det(v_0,v_2)}{\det(v_0,v_1)\det(v_2,v_3)},
\label{eq:1/1-z}\\[5pt]
 \frac{z-1}{z}&=\bigl[h_U(v_0):h_U(v_2):h_U(v_3):h_U(v_1)\bigr]
=\frac{\Bigl(\frac{v_0^1}{v_0^2}-\frac{v_1^1}{v_1^2}\Bigr)
\Bigl(\frac{v_2^1}{v_2^2} -\frac{ v_3^1}{v_3^2}\Bigr)}
{\Bigl(\frac{v_0^1}{v_0^2}-\frac{v_3^1}{v_3^2}\Bigr)
\Bigl(\frac{v_2^1} { v_2^2} -\frac{v_1^1}{v_1^2}\Bigr)}
=\frac{\det(v_0,v_1)\det(v_2,v_3)}{\det(v_0,v_3)\det(v_2,v_1)}.
\notag\label{eq:z-1/z}
\end{align}
Hence we also have
\begin{equation}\label{eq:1-z/z}
 \frac{1-z}{z}=-\bigl[h_U(v_0):h_U(v_2):h_U(v_3):h_U(v_1)\bigr]
=\frac{\Bigl(\frac{v_0^1}{v_0^2}-\frac{v_1^1}{v_1^2}\Bigr)
\Bigl(\frac{v_2^1}{v_2^2} -\frac{ v_3^1}{v_3^2}\Bigr)}
{\Bigl(\frac{v_0^1}{v_0^2}-\frac{v_3^1}{v_3^2}\Bigr)
\Bigl(\frac{v_1^1} { v_1^2} -\frac{v_2^1}{v_2^2}\Bigr)}
=\frac{\det(v_0,v_1)\det(v_2,v_3)}{\det(v_0,v_3)\det(v_1,v_2)}.
\end{equation}
We have that 
\begin{equation*}
 h_U(v_i)\neq h_U(v_j)\Leftrightarrow
\frac{v_i^1}{v_i^2}-\frac{v_j^1}{v_j^2}\neq0\Leftrightarrow\frac{
v_i^1v_j^2-v_i^2v_j^1}{v_i^2v_j^2}\neq0\Leftrightarrow\det(v_i,v_j)\neq0.
\end{equation*}
Then all the previous determinants are non-zero.

We define
\begin{equation}\label{eq:det2}
 \dH{v_i}{v_j}=\frac{1}{2}\Log\det(v_i,v_j)^2.
\end{equation}
Using formulas \eqref{eq:z}, \eqref{eq:1/1-z}, and \eqref{eq:1-z/z}, we assign a
flattening to $(v_0,v_1,v_2,v_3)\in \CC[3]{C}^{h_U\neq}(\Xsubgr{U})$ by setting
\begin{equation}\label{eq:flattening}
\begin{aligned}
\widetilde{w}_0(v_0,v_1,v_2,v_3)&=\dH{v_0}{v_3}+\dH{v_1}{v_2}-\dH{v_0}{v_2}-\dH{
v_1}{v_3},\\
\widetilde{w}_1(v_0,v_1,v_2,v_3)&=\dH{v_0}{v_2}+\dH{v_1}{v_3}-\dH{v_0}{v_1}-\dH{
v_2,}{v_3},\\
\widetilde{w}_2(v_0,v_1,v_2,v_3)&=\dH{v_0}{v_1}+\dH{v_2}{v_3}-\dH{v_0}{v_3}-\dH{
v_1}{v_2}.
\end{aligned}
\end{equation}
Recall that for $w\in\C^\times$ we have that
\begin{equation}\label{eq:w.2}
 \frac{1}{2}\Log w^2=\begin{cases}
         \Log w + \pi i & \text{if $\Arg w\in(-\pi,-\frac{\pi}{2}]$,}\\
	 \Log w & \text{if $\Arg w\in(-\frac{\pi}{2},\frac{\pi}{2}]$,}\\
	 \Log w - \pi i & \text{if $\Arg w\in(\frac{\pi}{2},\pi]$.}
        \end{cases}
\end{equation}
By the addition theorem of the logarithm \cite[Ch.~2-\S3.4]{Ahlfors:CompAna} and
\eqref{eq:w.2} we have that:
\begin{equation*}
\widetilde{w}_0=\Log z + ik\pi,\quad \widetilde{w}_1=\Log
\Bigl(\frac{1}{1-z}\Bigr)+ il\pi,\quad \widetilde{w}_2=\Log
\Bigl(\frac{1-z}{z}\Bigr)+ im\pi,
\end{equation*}
for some integers $k$, $l$ and $m$. Hence, $\widetilde{w}_0$, $\widetilde{w}_1$
and $\widetilde{w}_2$ are respectively logarithms of the cross-ratio parameters
$z$, $\frac{1}{1-z}$ and $\frac{z-1}{z}$ up to multiples of $\pi i$ 
and clearly we have that $\widetilde{w}_0+\widetilde{w}_1+\widetilde{w}_2=0$.
Therefore $(\widetilde{w}_0,\widetilde{w}_1,\widetilde{w}_2)$ is a flattening in
$\EpB{\C}$ of the ideal simplex %with parameter
$\bigl(h_U(v_0),h_U(v_1),h_U(v_2),h_U(v_3)\bigr)$.

This defines a homomorphism
\begin{equation}\label{eq:t.sigma}
\begin{aligned}
\widetilde{\sigma}\co \CC[3]{C}^{h_U\neq}(\Xsubgr{U})&\to\EpB{\C}\\
 (v_0,v_1,v_2,v_3)&\mapsto(\widetilde{w}_0,\widetilde{w}_1,\widetilde{w}_2).
\end{aligned}
\end{equation}

\begin{lemma}\label{lem:det.inv}
 Let $v_i,v_j\in\Xsubgr{U}$. Then $\det(v_i,v_j)$ is invariant under the action
of $G$ on $\Xsubgr{U}$.
\end{lemma}

\begin{proof}
Let $\begin{pmatrix}
      a&b\\c&d
     \end{pmatrix}\in SL_2(\C)$ and set
\begin{equation*}
 \bar{v}_i=(\bar{v}_i^1,\bar{v}_i^2)=\begin{pmatrix}
      a&b\\c&d
     \end{pmatrix}\begin{pmatrix}
                   v_i^1\\ v_i^2
                  \end{pmatrix}=(av_i^1+bv_i^2, cv_i^1+dv_i^2)
\end{equation*}
We have that
\begin{equation}
\begin{aligned}\label{eq:inv.G}
 \det( \bar{v}_i, \bar{v}_j)&=\det\begin{pmatrix}
                 \bar{v}_i^1 &  \bar{v}_i^2\\  \bar{v}_j^1 &  \bar{v}_j^2
               \end{pmatrix}=  \bar{v}_i^1 \bar{v}_j^2- \bar{v}_i^2
\bar{v}_j^1\\
            &=(av_i^1+bv_i^2)(cv_j^1+dv_j^2)-(av_j^1+bv_j^2)(cv_i^1+dv_i^2)\\
            &=(ad-bc)(v_i^1v_j^2-v_i^2v_j^1)\qquad\text{since $ad-bc=1$}\\
            &=\det(v_i,v_j).
\end{aligned}
\end{equation}
\end{proof}
Hence the homomorphism $\widetilde{\sigma}$ descends to the quotient by the
action of $SL_2(\C)$
\begin{equation*}
 \widetilde{\sigma}\co \CC[3]{B}^{h_U\neq}(\Xsubgr{U})\to\EpB{\C}.
\end{equation*}

Now suppose that
$(\widetilde{w}_0^0,\widetilde{w}_1^0,\widetilde{w}_2^0)$,\dots,$(\widetilde{w}
_0^4,\widetilde{w}_1^4,\widetilde{w}_2^4)$ are flattenings defined as above for
the simplices 
$\bigl(h_U(v_0),\dots,\widehat{h_U(v_i)},\dots,h_U(v_4)\bigr)$. We must check
that these flattenings satisfy the flattening condition, that is, we have to
check that the ten equations \eqref{eq:flat.cond} are satisfied.
We check the first equation, the others are similar:
\begin{description}
\item[{$[z_0z_1]$}]
\begin{align*}
 \widetilde{w}_0^2&=\dH{v_0}{v_4}+\dH{v_1}{v_3}-
\dH{v_0}{v_3}-\dH{v_1}{v_4},\\
-\widetilde{w}_0^3&=-\dH{v_0}{v_4}-\dH{v_1}{v_2}+
\dH{v_0}{v_2}+\dH{v_1}{v_4},\\
 \widetilde{w}_0^4&=\dH{v_0}{v_3}+\dH{v_1}{v_2}-\dH{v_0}{v_2}-\dH{v_1}{v_3},
\end{align*}
\end{description}
Having verified all the ten equations, it now follows from \cite[Theorem
2.8]{Dupont-Zickert:DFCCSC}  or \cite[Lemma~3.4]{Neumann:EBGCCSC} that
$\widetilde{\sigma}$ sends
boundaries to zero and we obtain a homomorphism
\begin{equation*}
 \widetilde{\sigma}\co H_3(\CC{B}^{h_U\neq}(\Xsubgr{U}))\to\EpB{\C}.
\end{equation*}
By Proposition~\ref{prop:tak.hUB.sub} we get a homomorphism from Takasu relative homology
\begin{equation*}
 \widetilde{\sigma}\co H_3(\SL{\C},U;\Z)\to\EpB{\C}.
\end{equation*}

\subsection{The homomorphism $\widetilde{\sigma}$ descends to
$\CC[3]{C}^{h_P\neq}(\Xsubgr{P})$}\label{ssec:desc}

Given $v_i=(v_i^1,v_i^2)\in\Xsubgr{U}$ we denote by
$\V_i=h_U^P(v_i)=[v_i^1,v_i^2]$ its
class in $\Xsubgr{P}$.

\begin{remark}\label{rem:det2def}
Notice that if $v_i=(v_i^1,v_i^2)\in\Xsubgr{U}$, then
\begin{align}
 \det(v_i,v_j)&=\det\begin{pmatrix}
		    v_i^1 & v_i^2\\
		    v_j^1 & v_j^2
                    \end{pmatrix}=v_i^1v_j^2-v_i^2v_j^1=%\det\begin{pmatrix}
							%      -v_i^1 & -v_i^2\\
							%      -v_j^1 & -v_j^2
							%    \end{pmatrix}=
\det(-v_i,-v_j),\label{eq:det++}\\[5pt]
\intertext{but on the other hand}
 \det(-v_i,v_j)&=\det\begin{pmatrix}
		    -v_i^1 & -v_i^2\\
		     v_j^1 & v_j^2
                    \end{pmatrix}=-v_i^1v_j^2+v_i^2v_j^1=%-\det\begin{pmatrix}
							 %     v_i^1 & v_i^2\\
							 %     v_j^1 & v_j^2
							 %   \end{pmatrix}=
-\det(v_i,v_j).\label{eq:det-+}
\end{align}
Thus, the quantity $\det(\V_i,\V_j)$ is just well-defined up to sign. However,
its square $\det(\V_i,\V_j)^2$ is well-defined.
\end{remark}

By Lemma~\ref{lem:det.inv}, \eqref{eq:det++} and \eqref{eq:det-+} we have:
\begin{lemma}
 Let $\V_i,\V_j\in\Xsubgr{P}$. Then $\det(\V_i,\V_j)^2$ is invariant under the
action of $G$ on $\Xsubgr{P}$.
\end{lemma}
So, we define
\begin{equation}\label{eq:det2.vi}
 \dH{\V_i}{\V_j}=\frac{1}{2}\Log\det(\V_i,\V_j)^2.
\end{equation}

Let $(\V_0,\V_1,\V_2,\V_3)\in \CC[3]{C}^{\tilde{h}\neq}(\Xsubgr{P})$. The
homomorphism $\widetilde{\sigma}$ given in \eqref{eq:t.sigma} descends to a well-defined homomorphism
\begin{align*}
 \widehat{\sigma}\co \CC[3]{C}^{h_P\neq}(\Xsubgr{P})&\to\EpB{\C}\\
(\V_0,\V_1,\V_2,\V_3)&\mapsto(\widetilde{w}_0,\widetilde{w}_1,\widetilde{w}_2).
\end{align*}
by assigning to $(\V_0,\V_1,\V_2,\V_3)$ the flattening of the ideal simplex
\begin{equation*}
 \bigl(h_P(\V_0),h_P(\V_1),h_P(\V_2),h_P(\V_3)\bigr)=\bigl(h_U(v_0),h_U(v_1),
h_U(v_2),h_U(v_3)\bigr)
\end{equation*}
given by \eqref{eq:flattening}.
By Remark~\ref{rem:P.bG} the action of $G=\SL{\C}$ on $\CC[3]{C}^{h_P\neq}(\Xsubgr{P})$ descends to an action of $\bar{G}=\PSL{\C}$, we use the notation $\CC[3]{C}^{h_{\bar{P}}\neq}(\Xsubgr{\bar{P}})$ to emphasize the action of $\bar{G}$. 
Hence we obtain a homomorphism
\begin{equation*}
 \widehat{\sigma}\co H_3(\CC{B}^{h_{\bar{P}}\neq}(\Xsubgr{\bar{P}}))\to\EpB{\C}.
\end{equation*}
By Proposition~\ref{prop:tak.hUB.sub} we get a homomorphism from Takasu relative homology
\begin{equation*}
 \widehat{\sigma}\co H_3(\PSL{\C},\bar{P};\Z)\to\EpB{\C}.
\end{equation*}

\begin{remark}
 Notice that we do not get a homomorphism from $H_3(\SL{\C},P;\Z)$, since $P$ is not $B$-malnormal in $\SL{\C}$.
\end{remark}

\begin{proposition}\label{prop:im.EBg}
 The image of $\widehat{\sigma}\co
H_3(\CC{B}^{h_{\bar{P}}\neq}(\Xsubgr{\bar{P}}))\to\EpB{\C}$ is in $\EBg{\C}$.
\end{proposition}

\begin{proof}
 Define a map $\mu\co\CC[2]{B}^{h_{\bar{P}}\neq}(\Xsubgr{\bar{P}})\to\antp{\C}$ by
\begin{equation*}
 (\V_0,\V_1,\V_2)_G\mapsto
\dH{\V_0}{\V_1}\wedge\dH{\V_0}{\V_2}-\dH{\V_0}{\V_1}\wedge\dH{\V_1}{\V_2}+\dH{
\V_0}{\V_2}\wedge\dH{\V_1}{\V_2}.
\end{equation*}
Recall that the extended Bloch group $\EBg{\C}$ is the kernel of the
homomorphism
\begin{align*}
 \widehat{\nu}\co\EpB{\C}&\to\antp{\C}\\
(w_0,w_1,w_1)&\mapsto w_0\wedge w_1.
\end{align*}
A straightforward computation shows that the following diagram
\begin{equation*}
 \xymatrix{
\CC[3]{B}^{h_{\bar{P}\neq}}(\Xsubgr{\bar{P}})\ar[d]^{\partial}\ar[r]^{\widehat{\sigma}} &
\EpB{\C}\ar[d]^{\widehat{\nu}}\\
\CC[2]{B}^{h_{\bar{P}\neq}}(\Xsubgr{\bar{P}})\ar[r]^{\mu} & \antp{\C}
}
\end{equation*}
is commutative.
This means that cycles are mapped to $\EBg{\C}$ as desired.
\end{proof}

Therefore $\widehat{\sigma}\co H_3(\CC{B}^{h_{\bar{P}\neq}}(\Xsubgr{\bar{P}}))\to\EBg{\C}$. 
Then by Proposition~\ref{prop:tak.hUB.sub} we have a homomorphism
\begin{equation}\label{eq:tilde.sigma}
\widehat{\sigma}\co H_3(\PSL{\C},\bar{P};\Z)\to\EBg{\C}.
\end{equation}

\subsection{Using $\Xsp$}

Sometimes is useful to express the homomorphism $\widehat{\sigma}$ using
$\Xspb$ instead of $\Xsubgr{\bar{P}}$ since with $\Xspb$ we do not have to worry about
equivalence classes and representatives. Again, we simplify notation by setting
\begin{equation*}
 \bar{h}_{\bar{P}}=\bar{h}_{\bar{P}}^{\bar{B}}\co\Xspb\to\Xsubgr{\bar{B}}.
\end{equation*}

Let $A_i, A_j\in \Xsp$, then we have that
\begin{equation*}
A_i=\begin{pmatrix}
r_i & t_i\\t_i & s_i
\end{pmatrix},\qquad r_is_i=t_i^2.
\end{equation*}
Define
\begin{equation*}
 DS(A_i,A_j)=r_is_j-2t_it_j+r_js_i.
\end{equation*}

Recall from Corollary~\ref{cor:iso.P} that the $\bar{G}$--isomorphism between
$\Xsubgr{\bar{P}}$ and $\Xspb$ is given by
\begin{align*}
\varrho\co\Xsubgr{\bar{P}} &\to \Xspb,\\[5pt]
\begin{bmatrix}
 u\\v
\end{bmatrix}
&\leftrightarrow \begin{pmatrix}
                      u^2 & uv \\ uv & v^2
                     \end{pmatrix}.
\end{align*}

\begin{lemma}\label{lem:dH=DS}
Let $\V_i$ and $\V_j$ in $\Xsubgr{\bar{P}}$. Then
\begin{equation*}
 DS\bigl(\varrho(\V_i),\varrho(\V_j)\bigr)=\det(\V_i,\V_j)^2.
\end{equation*}
\end{lemma}

\begin{proof}
Let $\V_i=[u_i,v_i]$ and $\V_j=[u_j,v_j]$ in $\Xsubgr{P}$. We have that
\begin{equation*}
\varrho(\V_i)=\begin{pmatrix}
      u_i^2 & u_iv_i\\
      u_iv_i& v_i^2
     \end{pmatrix},\quad \varrho(\V_j)=\begin{pmatrix}
                         u_j^2 & u_jv_j\\
                         u_jv_j& v_j^2                  
                         \end{pmatrix}
\end{equation*}
Then
\begin{equation*}
 DS\bigl(\varrho(\V_i),\varrho(\V_j)\bigr)
=u_i^2v_j^2-2u_iv_iu_jv_j+u_j^2v_i^2=(u_iv_j-u_jv_i)^2=\det(\V_i,\V_j)^2.
\end{equation*}
\end{proof}

Combining Corollary~\ref{cor:iso.P} and Lemma~\ref{lem:dH=DS} we get the
following corollary which
can also be proved with a straightforward but tedious computation.

\begin{corollary}\label{prop:DS.G}
 Let $A_i,A_j\in\Xspb$ and $\bar{g}\in \bar{G}$. Then $DS(A_i,A_j)$ is $\bar{G}$--invariant, that
is,
\begin{equation*}
 DS(\bar{g} A_i\bar{g}^T,\bar{g} A_j\bar{g}^T)=DS(A_i,A_j).
\end{equation*}
\end{corollary}

We define
\begin{equation*}
 \dH{A_i}{A_j}=\frac{1}{2}\Log DS(A_i,A_j).
\end{equation*}
So by Lemma~\ref{lem:dH=DS}, given $\V_i,\V_j\in\Xsubgr{P}$ we have that 
$\dH{\varrho(\V_i)}{\varrho(\V_j)}=\dH{\V_i}{\V_j}$.

Let $(A_0,A_1,A_2,A_3)\in\CC[3]{C}^{\bar{h}_P\neq}(\Xsp)$. Then the homomorphism
\begin{align*}
 \widetilde{\sigma}\co\CC[3]{C}^{\bar{h}_P\neq}(\Xspb)&\to\EpB{\C}\\[5pt]
(A_0,A_1,A_2,A_3)&\mapsto(\bar{w}_0,\bar{w}_1,\bar{w}_2),
\end{align*}
is given by assigning to $(A_0,A_1,A_2,A_3)$ the flattening defined by
\begin{align*}
\widetilde{w}_0&=\dH{A_0}{A_3}+\dH{A_1}{A_2}-\dH{A_0}{A_2}-\dH{A_1}{A_3},\\
\widetilde{w}_1&=\dH{A_0}{A_2}+\dH{A_1}{A_3}-\dH{A_0}{A_1}-\dH{A_2}{A_3},\\
\widetilde{w}_2&=\dH{A_0}{A_1}+\dH{A_2}{A_3}-\dH{A_0}{A_3}-\dH{A_1}{A_2},
\end{align*}
which by Lemma~\ref{lem:dH=DS} is the same as the one given in
\eqref{eq:flattening}
and by Proposition~\ref{prop:im.EBg} we obtain the homomorphism
\eqref{eq:tilde.sigma}.

\section{Zickert's relative class revisited}\label{sec:zrc}

In \cite{Zickert:VCSIR} Zickert defines a complex $\bar{C}_*(\bar{G},\bar{P})$
of truncated simplices 
and proves that the complex
$\bar{B}_*(\bar{G},\bar{P})=\bar{C}_*(\bar{G},\bar{P})\otimes_{\Z[G]}\Z$ 
computes the Takasu relative homology groups $H_*(\bar{G},\bar{P};\Z)$
\cite[Corollary~3.8]{Zickert:VCSIR}. This complex is used to define a
homomorphism $\Psi\co H_3(\bar{G},\bar{P};\Z)\to\EBg{\C}$
\cite[Theorem~3.17]{Zickert:VCSIR}. Given an ideal triangulation of an
hyperbolic $3$--manifold $M$ and using a developing map of the geometric
representation to give to each ideal simplex a decoration by horospheres, it is
defined a \emph{relative class} $F(M)$ in the group $H_3(\bar{G},\bar{P};\Z)$
\cite[Corollary~5.6]{Zickert:VCSIR}. This relative class
depends on the choice of decoration, but it is proved that its image under the
homomorphism $\Psi$ is independent of the choice of decoration
\cite[Theorem~6.10]{Zickert:VCSIR}. In fact, in \cite{Zickert:VCSIR} it is
considered the more general situation of a tame $3$--manifold with a boundary
parabolic $\PSL{\C}$--representation, this will be considered in the following
section.

In this section we compare the results in \cite{Zickert:VCSIR} with our
construction of the invariant $\beta_P(M)$. We give an explicit isomorphism
between the complexes $\bar{C}_*(\bar{G},\bar{P})$ and
$\CC{C}^{h_{\bar{P}}^{\bar{B}}\neq}(\Xsubgr{\bar{P}})$; under this isomorphism, the homomorphisms
$\Psi$ and $\widehat{\sigma}$ in \eqref{eq:tilde.sigma} coincide. 

Moreover, we prove that the image of Zickert's class $\lambda_3(F(M))\in H_3([\bar{G}:\bar{P}];\Z)$ under the canonical homomorphism
\eqref{eq:Taka.Hoch} is independent on the choice of decoration and it is precisely the invariant $\beta_P(M)$ in $H_3([\bar{G}:\bar{P}];\Z)$.
 
\begin{remark}
Notice the difference of notation, in \cite{Zickert:VCSIR} $G=\PSL{\C}$ and $P$
corresponds to the subgroup of $\PSL[]{\C}$ given by the image of the group of
upper triangular matrices with $1$ in the diagonal, that is, in our notation to
the subgroup $\bar{U}=\bar{P}$.
\end{remark}

\subsection{The complex of truncated simplices}\label{ssec:cts}

Let $\triangle$ be an $n$--simplex with a vertex ordering given by associating
an integer $i\in\{0,\dots,n\}$ to each vertex. Let $\overline{\triangle}$ denote
the corresponding truncated simplex obtained by chopping off disjoint regular
neighborhoods of the vertices. Each vertex of $\overline{\triangle}$ is
naturally associated with an ordered pair $ij$ of distinct integers. Namely, the
$ij$th vertex of $\overline{\triangle}$ is the vertex near the $i$th vertex of
$\triangle$ and on the edge going to the $j$th vertex of $\triangle$.

Let $\overline{\triangle}$ be a truncated $n$--simplex. A
\emph{$\bar{G}$--vertex labeling} $\{\bar{g}^{ij}\}$ of $\overline{\triangle}$
assigns to the vertex $ij$ of $\overline{\triangle}$ an element
$\bar{g}^{ij}\in\bar{G}$ satisfying the following properties:
\begin{enumerate}[(i)]
 \item For fixed $i$, the labels $\bar{g}^{ij}$ are distinct elements in
$\bar{G}$ mapping to the same left $\bar{P}$--coset.\label{it:coset.labels}
 \item The elements $\bar{g}_{ij}=(\bar{g}^{ij})^{-1}\bar{g}^{ji}$ are
counterdiagonal, that is, of the form $\bigl(\begin{smallmatrix}
  0 & -c^{-1}\\ c & 0                                                           
\end{smallmatrix}\bigr)$.\label{it:counter}
\end{enumerate}
Let $\bar{C}_n(\bar{G},\bar{P})$, $n\geq1$, be the free abelian group generated
by $\bar{G}$--vertex labelings of truncated $n$--simplices. 

\begin{remark}\label{rem:same.ac}
Since $\Xsubgr{\bar{P}}$ is $\bar{G}$--isomorphic to the set of left
$\bar{P}$--cosets, using the homomorphism $h_{\bar{I}}^{\bar{P}}$ given in
\eqref{eq:h.hHK} (see Remark~\ref{rem:h.bhHK}) property \eqref{it:coset.labels}
means that for fixed $i$ we have
\begin{equation*}
h_{\bar{I}}^{\bar{P}}(\bar{g}^{ij})=[a_i,c_i]
\end{equation*}
for some fixed element $[a_i,c_i]\in\Xsubgr{\bar{P}}$ and for all $j\neq i$. By
the definition of $h_{\bar{I}}^{\bar{P}}$ we have that for fixed $i$ all the
$\bar{g}^{ij}$ have the form
\begin{equation*}
\bar{g}^{ij}=\overline{\begin{pmatrix}
          a_i & b_{ij}\\
          c_i & d_{ij}
        \end{pmatrix}},\quad\text{with $j\neq i$ and $b_{ij}\neq b_{ik}$ and
$d_{ij}\neq d_{ik}$ for $j\neq k$.}
\end{equation*}
\end{remark}

Left multiplication endows $\bar{C}_n(\bar{G},\bar{P})$ with a free $G$--module
structure and the usual boundary map on untruncated simplices induces a boundary
map on $\bar{C}_n(\bar{G},\bar{P})$, making it into a chain complex. Define
\begin{equation*}
\bar{B}_*(\bar{G},\bar{P})=\bar{C}_*(\bar{G},\bar{P})\otimes_{\Z[G]}\Z.
\end{equation*}
Let $\{\bar{g}^{ij}\}$ be a $\bar{G}$--vertex labeling of a truncated
$n$--simplex $\overline{\triangle}$. 
We define a \emph{$\bar{G}$--edge labeling} of $\overline{\triangle}$ assigning
to the oriented edge going from vertex $ij$ to vertex $kl$ the labeling
$(\bar{g}^{ij})^{-1}\bar{g}^{kl}$. It is easy to see that for any
$\bar{g}\in\bar{G}$, the $\bar{G}$--vertex labelings of $\overline{\triangle}$
given by $\{\bar{g}^{ij}\}$ and $\{\bar{g}\bar{g}^{ij}\}$ have the same
$\bar{G}$--edge labelings. Hence, a $\bar{G}$--edge labeling represents a
generator of $\bar{B}_*(\bar{G},\bar{P})$. The labeling of an edge going from
vertex $i$ to vertex $j$ in the untruncated simplex is denoted by
$\bar{g}_{ij}$, and the labeling of the edges near the $i$th vertex are denoted
by $\bar{\alpha}^i_{jk}$. These edges are called the \emph{long edges} and the
\emph{short edges} respectively. By properties \eqref{it:coset.labels} and
\eqref{it:counter} in the definition of $\bar{G}$--vertex labelings of a
truncated simplex, the $\bar{\alpha}^i_{jk}$'s are nontrivial elements in
$\bar{P}$ and the $\bar{g}_{ij}$'s are counterdiagonal. Moreover, from the
definition of $\bar{G}$--edge labelings we have that the product of edge
labeling along any two-face (including the triangles) is $\bar{I}$.

In \cite[Corollary~3.8]{Zickert:VCSIR} it is proved that the complex $\bar{B}_*(\bar{G},\bar{P})$ computes the groups $H_*(\bar{G},\bar{P})$.

In what follows we need a more explicit version of
\cite[Lemma~3.5]{Zickert:VCSIR}.

\begin{lemma}[{\cite[Lemma~3.5]{Zickert:VCSIR}}]\label{lem:Zlem3.5}
 Let $\bar{g}_i\bar{P}=\overline{\begin{pmatrix}
           a_i & b_i\\ c_i & d_i
          \end{pmatrix}}\bar{P}$ and $\bar{g}_j\bar{P}=\overline{\begin{pmatrix}
           a_j & b_j\\ c_j & d_j
          \end{pmatrix}}\bar{P}$ be $\bar{P}$--cosets satisfying the condition
$\bar{g}_i\bar{B}\neq\bar{g}_j\bar{B}$. There exists unique coset
representatives $\bar{g}_i\bar{x}_{ij}$ and
 $\bar{g}_j\bar{x}_{ji}$ satisfying the condition that
$(\bar{g}_i\bar{x}_{ij})^{-1}\bar{g}_j\bar{x}_{ji}$ be counterdiagonal given by
 \begin{equation}\label{eq:gx}
\bar{g}_i\bar{x}_{ij}=\overline{\begin{pmatrix}
             a_i & \frac{a_j}{a_ic_j-a_jc_i}\\
             c_i & \frac{c_j}{a_ic_j-a_jc_i} 
            \end{pmatrix}},\qquad
\bar{g}_j\bar{x}_{ji}=\overline{\begin{pmatrix}
             a_j & \frac{a_i}{a_jc_i-a_ic_j}\\
             c_j & \frac{c_i}{a_jc_i-a_ic_j} 
            \end{pmatrix}}.
 \end{equation}
\end{lemma}

\begin{proof}
We start by reproducing the proof of \cite[Lemma~3.5]{Zickert:VCSIR} since it
saves computations. Let
$\bar{g}_i^{-1}\bar{g}_j=\overline{\bigl(\begin{smallmatrix}
              a & b\\ c & d
             \end{smallmatrix}\bigr)}$, and let
$\bar{x}_{ij}=\bigl(\overline{\begin{smallmatrix}
                                                   1 & p_{ij}\\ 0 & 1
                                                  \end{smallmatrix}}\bigr)$ and
$\bar{x}_{ji}=\bigl(\overline{\begin{smallmatrix}
1 & p_{ji}\\ 0 & 1
\end{smallmatrix}}\bigr)$.
We have
\begin{equation*}
 \bar{x}_{ij}^{-1}\bar{g}_i^{-1}\bar{g}_j\bar{x}_{ji}=\overline{\begin{pmatrix}
                 a-cp_{ij} & ap_{ji}+b-p_{ij}(cp_{ji}+d)\\ c & cp_{ji}+d
                \end{pmatrix}}.
\end{equation*}
Since $\bar{g}_i\bar{B}\neq\bar{g}_j\bar{B}$, it follows that $c$ is nonzero.
This implies that there exists unique complex numbers $p_{ij}$ and $p_{ji}$ such
that the above matrix is conterdiagonal. They are given by
\begin{equation}\label{eq:xij.xji}
p_{ij}=\frac{a}{c},\qquad p_{ji}=-\frac{d}{c}.
\end{equation}
Now, using the explicit expressions for $\bar{g}_i$ and  $\bar{g}_j$ we have
\begin{equation}\label{eq:gigj}
 \bar{g}_i^{-1}\bar{g}_j=\overline{\begin{pmatrix}
              a_jd_i-b_ic_j & d_ib_j-b_id_j\\
              a_ic_j-a_jc_i & a_id_j-c_ib_j
             \end{pmatrix}}=\overline{\begin{pmatrix}
                             a & b\\ c & d
                            \end{pmatrix}}.
\end{equation}
Substituting in \eqref{eq:xij.xji} we get
\begin{equation*}
 p_{ij}=\frac{a_jd_i-b_ic_j}{a_ic_j-a_jc_i},\qquad
p_{ji}=\frac{a_id_j-b_ic_j}{a_ic_j-a_jc_i}.
\end{equation*}
Hence we have
\begin{align*}
\bar{g}_i\bar{x}_{ij}&=\overline{\begin{pmatrix}
           a_i & b_i\\ c_i & d_i
          \end{pmatrix}}\overline{\begin{pmatrix}
                        1 & \frac{a_jd_i-b_ic_j}{a_ic_j-a_jc_i}\\ 0 & 1
                       \end{pmatrix}}=\overline{\begin{pmatrix}
             a_i & \frac{a_j}{a_ic_j-a_jc_i}\\
             c_i & \frac{c_j}{a_ic_j-a_jc_i} 
            \end{pmatrix}},\\
\bar{g}_j\bar{x}_{ji}&=\overline{\begin{pmatrix}
           a_j & b_j\\ c_j & d_j
          \end{pmatrix}}\overline{\begin{pmatrix}
                        1 & \frac{a_id_j-b_jc_i}{a_jc_i-a_ic_j}\\ 0 & 1
                       \end{pmatrix}}=\overline{\begin{pmatrix}
             a_j & \frac{a_i}{a_jc_i-a_ic_j}\\
             c_j & \frac{c_i}{a_jc_i-a_ic_j} 
            \end{pmatrix}}.
\end{align*}
Notice that $\bar{g}_i\bar{x}_{ij}$ is a well defined element in $\bar{G}$, if
we change the signs of $a_i$ and $c_i$ the whole matrix changes sign, while if
we change the signs of $a_j$ and $c_j$ the matrix does not change. Analogously
for $\bar{g}_j\bar{x}_{ji}$.
\end{proof}

\begin{remark}\label{rem:hBP.neq}
Notice that the expressions for $\bar{g}_i\bar{x}_{ij}$ and
$\bar{g}_j\bar{x}_{ji}$ given in \eqref{eq:gx} only depend on the classes
$[a_i,c_i]$ and $[a_j,c_j]$ in $\Xsubgr{\bar{P}}$, so indeed they only depend on
the left $\bar{P}$--cosets $\bar{g}_i\bar{P}$ and $\bar{g}_j\bar{P}$. Also
notice that by \eqref{eq:gigj} the condition
$\bar{g}_i\bar{B}\neq\bar{g}_j\bar{B}$ is equivalent to $a_ic_j-a_jc_i\neq0$
which is equivalent to
$h_{\bar{P}}^{\bar{B}}(\bar{g}_i\bar{P})=\frac{a_i}{c_i}\neq
\frac{a_j}{c_j}=h_{\bar{P}}^{\bar{B}}(\bar{g}_j\bar{P})$.
\end{remark}

\begin{corollary}\label{cor:gens}
Let $\overline{\triangle}$ be a truncated $n$--simplex. A generator of
$\bar{C}_n(\bar{G},\bar{P})$, \ie a $\bar{G}$--vertex labeling $\{\bar{g}^{ij}\}$
of $\overline{\triangle}$ has the form
\begin{equation*}
 \bar{g}^{ij}=\overline{\begin{pmatrix}
             a_i & \frac{a_j}{a_ic_j-a_jc_i}\\
             c_i & \frac{c_j}{a_ic_j-a_jc_i} 
            \end{pmatrix}},\quad i,j\in\{1,\dots,n\},\quad j\neq i,\quad
a_ic_j-a_jc_i\neq0,
\end{equation*}
and the class $[a_i,c_i]\in\Xsubgr{\bar{P}}$ corresponds to the left
$\bar{P}$--coset associated to the $i$--th vertex of $\triangle$. Hence, a
generator of $\bar{B}_n(\bar{G},\bar{P})$, \ie a $\bar{G}$--edge labeling of
$\overline{\triangle}$ has the form
\begin{align}
\bar{\alpha}^i_{jk}&=\overline{\begin{pmatrix}
  1 & \frac{a_kc_j-a_jc_k}{(a_ic_j-a_jc_i)(a_ic_k-a_kc_i)}\\
  0 & 1
 \end{pmatrix}}, \quad i,j,k\in\{1,\dots,n\},\quad i\neq j,k,\quad j\neq
k,\label{eq:el.alpha}\\
\bar{g}_{ij}&=\overline{\begin{pmatrix}
                           0 & -\frac{1}{a_ic_j-a_jc_i}\\ a_ic_j-a_jc_i & 0
                          \end{pmatrix}},\quad  i,j\in\{1,\dots,n\},\quad i\neq
j.\label{eq:el.gij}
\end{align}
\end{corollary}

\begin{proof}
Follows immediately from Remark~\ref{rem:same.ac}, Lemma~\ref{lem:Zlem3.5} and
Remark~\ref{rem:hBP.neq}.
\end{proof}

\begin{corollary}\label{cor:iso.complex}
 There is a $\bar{G}$--isomorphism of chain complexes
\begin{align*}
\CC[n]{C}^{h_{\bar{P}}^{\bar{B}}\neq}(\Xsubgr{\bar{P}})&\leftrightarrow
\bar{C}_n(\bar{G},\bar{P})\\
\bigl([a_0,c_0],\dots,[a_n,c_n]\bigr)&\leftrightarrow \left\{\bar{g}^{ij}=\overline{\begin{pmatrix}
             a_i & \frac{a_j}{a_ic_j-a_jc_i}\\
             c_i & \frac{c_j}{a_ic_j-a_jc_i} 
            \end{pmatrix}}\right\},\ i,j\in\{1,\dots,n\},\quad i\neq j.
\end{align*}
Hence, there is an isomorphism of chain complexes
$\CC[n]{B}^{h_{\bar{P}}^{\bar{B}}\neq}(\Xsubgr{\bar{P}})\cong
\bar{B}_n(\bar{G},\bar{P})$ where the $\bar{G}$--orbit
$\bigl([a_0,c_0],\dots,[a_n,c_n]\bigr)_{\bar{G}}$ corresponds to the
$\bar{G}$--edge labeling given by \eqref{eq:el.alpha} and \eqref{eq:el.gij}.
\end{corollary}

\begin{proof}
This is a refined version of \cite[Corollary~3.6]{Zickert:VCSIR} and follows
from Corollary~\ref{cor:gens}. By direct computation it is easy to see that the
isomorphism is $\bar{G}$--equivariant. The only thing that remains to prove is
that the isomorphism commutes with the boundary maps of the complexes, which is
an easy exercise.
\end{proof}

\begin{remark}\label{rem:efficient}
From Corollary~\ref{cor:iso.complex}, to represent a generator of
$\CC[n]{C}^{h_{\bar{P}}^{\bar{B}}\neq}(\Xsubgr{\bar{P}})$ we just need $2(n+1)$
complex numbers, while to represent a generator of $\bar{C}_n(\bar{G},\bar{P})$
we need $4(n+1)n$ because there is a lot of redundant information in
$\bar{g}^{ij}$, the entries $b_{ij}$ and $d_{ij}$ in $g^{ij}$, see
Remarks~\ref{rem:same.ac} and \ref{rem:hBP.neq}. So it is more efficient to use
the complex
$\CC{C}^{h_{\bar{P}}^{\bar{B}}\neq}(\Xsubgr{\bar{P}})$ than the complex
$\bar{C}_*(\bar{G},\bar{P})$ to compute $H_*(\bar{G},\bar{P};\Z)$.
\end{remark}

\begin{remark}
If we denote by $\V_i=[v_i^1,v_i^2]$ an element in $\Xsubgr{\bar{P}}$ as in
Subsection~\ref{ssec:desc}, we have that the isomorphism of
Corollary~\ref{cor:iso.complex} is written as
\begin{equation*}
 (\V_0,\dots,\V_n) \leftrightarrow \left\{\bar{g}^{ij}=\overline{\begin{pmatrix}
             v_i^1 & \frac{v_j^1}{\det(\V_i,\V_j)}\\[5pt]
             v_i^2 & \frac{v_j^2}{\det(\V_i,\V_j)} 
            \end{pmatrix}}\right\},\ i,j\in\{1,\dots,n\},\quad i\neq j,
\end{equation*}
where $(\V_0,\dots,\V_n)$ is an $(n+1)$--tuple of elements of $\Xsubgr{\bar{P}}$
such that $\det(\V_i,\V_j)^2\neq0$ for $i\neq j$, see Remark~\ref{rem:hBP.neq}.
We also have that in this notation the $\bar{G}$--edge labeling given by
\eqref{eq:el.alpha} and \eqref{eq:el.gij} is written as
\begin{align}
\bar{\alpha}^i_{jk}&=\overline{\begin{pmatrix}
  1 & \frac{\det(\V_k,\V_j)}{\det(\V_i,\V_j)\det(\V_i,\V_k)}\\
  0 & 1
 \end{pmatrix}}, \quad i,j,k\in\{1,\dots,n\},\quad i\neq j,k,\quad j\neq
k,\notag\\
\bar{g}_{ij}&=\overline{\begin{pmatrix}
                           0 & -\frac{1}{\det(\V_i,\V_j)}\\ \det(\V_i,\V_j) & 0
                          \end{pmatrix}},\quad  i,j\in\{1,\dots,n\},\quad i\neq
j.\label{eq:el.gij.vi}
\end{align}
Notice that although $\det(\V_i,\V_j)$ is only well-defined up to sign, see
Remark~\ref{rem:det2def}, we get
well-defined elements in $\bar{G}$. The fact that the matrices
\eqref{eq:el.alpha} and \eqref{eq:el.gij} of the $\bar{G}$--edge labeling are
constant under the action of $\bar{G}$ is because $\det(\V_i,\V_j)$ is invariant
(up to sign) under the action of $\bar{G}$, see Lemma~\ref{lem:det.inv}.
\end{remark}

\subsection{Decorated ideal simplices and flattenings}

Also in \cite{Zickert:VCSIR} it is proved that there is a one-to-one
correspondence between generators of $\bar{B}_3(\bar{G},\bar{P})$ and congruence
classes of decorated ideal simplices.

Remember that the subgroup $\bar{P}$ fixes $\infty\in\hypc$ and acts by
translations on any horosphere at $\infty$. A horosphere at $\infty$ is endowed
with the counterclockwise orientation as viewed from $\infty$. Since $\bar{G}$
acts transitively on horospheres, we get an orientation on all horospheres.

A horosphere together with a choice of orientation-preserving isometry to $\C$
is called an \emph{Euclidean horosphere} \cite[Definition~3.9]{Zickert:VCSIR}.
Two horospheres based at the same point are considered equal if the isometries
differ by a translation. Denote by $H(\infty)$ the horosphere at $\infty$ at
height $1$ over the bounding complex plane $\C$, with the Euclidean structure
induced by projection. 
We let $\bar{G}$ act on Euclidean horospheres in the obvious way, this action is
transitive and the isotropy subgroup of $H(\infty)$ is $\bar{P}$. Hence the set
of Euclidean horospheres can be identified with the set $\bar{G}/\bar{P}$ of
left $\bar{P}$--cosets, which is $\bar{G}$--isomorphic to $\Xsubgr{\bar{P}}$,
where an explicit $\bar{G}$--isomorphism is given by
\begin{equation}\label{eq:horo.XP}
\begin{aligned}
 \{\text{Euclidean horospheres}\}&\leftrightarrow \Xsubgr{\bar{P}}\\
   H(\infty) &\leftrightarrow [1,0],
\end{aligned}
\end{equation}
and extending equivariantly using the action of $\bar{G}$.

A choice of Euclidean horosphere at each vertex of an ideal simplex  is called a
\emph{decoration} of the simplex. Having fixed a decoration, we say that the
ideal simplex is \emph{decorated}. Two decorated ideal simplices are called
\emph{congruent} if they differ by an element of $\bar{G}$.

Using the identification of Euclidean horospheres with left $\bar{P}$--cosets,
we can see a decorated ideal simplex as an ideal simplex with a choice of a left
$\bar{P}$--coset for each vertex of the ideal simplex.

\begin{proposition}[{\cite[Theorem~3.13]{Zickert:VCSIR}}]\label{prop:gen.dec}
Generators in $\CC[3]{C}^{h_{\bar{P}}^{\bar{B}}\neq}(\Xsubgr{\bar{P}})$ are in
one-to-one correspondence with decorated simplices. Thus, generators of
$\CC[3]{B}^{h_{\bar{P}}^{\bar{B}}\neq}(\Xsubgr{\bar{P}})$ are in one-to-one
correspondence with congruence classes of decorated simplices.
\end{proposition}

\begin{proof}
Consider the homomorphism $(h_{\bar{P}}^{\bar{B}})_*\co
\CC[3]{C}^{h_{\bar{P}}^{\bar{B}}\neq}(\Xsubgr{\bar{P}})\to
\CC[3]{C}^{\neq}(\Xsubgr{\bar{B}})$ and consider a generator
$(\V_0,\dots,\V_3)$ of
$\CC[3]{C}^{h_{\bar{P}}^{\bar{B}}\neq}(\Xsubgr{\bar{P}})$. Its image
$\bigl((h_{\bar{P}}^{\bar{B}})_*(\V_0),\dots,(h_{\bar{P}}^{\bar{B}}
)_*(\V_3)\bigr)$  is a $4$--tuple of distinct points in $\Xsubgr{\bar{B}}$, so
it determines a unique ideal simplex in $\overline{\hyp}$.
Moreover, $\V_i$ represents a left $\bar{P}$--coset which corresponds to the
vertex $(h_{\bar{P}}^{\bar{B}})_*(\V_i)$ of such ideal simplex. Hence
$(\V_0,\dots,\V_3)$ represents a decorated simplex.
\end{proof}

This together with the isomorphism given in Corollary~\ref{cor:iso.complex}
proves \cite[Theorem~3.13]{Zickert:VCSIR} that there is a one-to-one
correspondence between generators of $\bar{B}_3(\bar{G},\bar{P})$ and congruence
classes of decorated ideal simplices, see \cite[Remark~3.14]{Zickert:VCSIR}.

For a matrix $g=\bigl(\begin{smallmatrix}
                 a & b\\ c & d  
                  \end{smallmatrix}\bigr)$, let $c(g)$ denote the entry $c$. Let
$\alpha$ be a generator of
$\bar{B}(\bar{G},\bar{P})$. By \eqref{eq:el.gij.vi} we have
that $c(\bar{g}_{ij})=\pm\det(\V_i,\V_j)$, that is, it is only well-defined up
to sign. But we have that
\begin{equation}\label{eq:c2=det2}
 c(\bar{g}_{ij})^2=\det(\V_i,\V_j)^2
\end{equation}
is a well-defined non-zero complex number. Squaring formulas \eqref{eq:z},
\eqref{eq:1/1-z} and \eqref{eq:1-z/z} and using \eqref{eq:c2=det2} we get
\begin{equation*}
\frac{c(\bar{g}_{03})^2c(\bar{g}_{12})^2}{c(\bar{g}_{02})^2c(\bar{g}_{13})^2}
=z^2,\quad
\frac{c(\bar{g}_{13})^2c(\bar{g}_{02})^2}{c(\bar{g}_{01})^2c(\bar{g}_{23})^2}
=\Bigl(\frac{1}{1-z}\Bigr)^2,\quad
\frac{c(\bar{g}_{01})^2c(\bar{g}_{23})^2}{c(\bar{g}_{03})^2c(\bar{g}_{12})^2}
=\Bigl(\frac{1-z}{z}\Bigr)^2,
\end{equation*}
which are the formulas of \cite[Lemma~3.15]{Zickert:VCSIR}. Now, our choice of
logarithm branch defines a square root of $c(\bar{g}_{ij})^2$, see
\cite[Remark~3.4]{Zickert:VCSIR}, given by
\begin{equation}\label{eq:Psi=sigma}
\Log c(\bar{g}_{ij})= \frac{1}{2}\Log c(\bar{g}_{ij})^2= \frac{1}{2}\Log
\det(\V_i,\V_j)^2,
\end{equation}
which is the definition of $\dH{\V_i}{\V_j}$ given in \eqref{eq:det2.vi}. 
Using $\Log c(\bar{g}_{ij})$ in \cite[(3.6)]{Zickert:VCSIR} to every generator of 
$\bar{B}(\bar{G},\bar{P})$ associates a flattening giving a homomorphism \cite[Theorem~3.17]{Zickert:VCSIR}
\begin{equation}\label{eq:Z.Psi}
 \Psi\co H_3(\bar{G},\bar{P};\Z)\to\EBg{\C}.
\end{equation}

\begin{proposition}\label{prop:Psi=tildesigma}
The following diagram commutes
\begin{equation*}
\xymatrix{
H_3(\bar{G},\bar{P};\Z)\cong H_3\bigl(\CC{B}^{h_{\bar{P}}^{\bar{B}}\neq}(\Xsubgr{\bar{P}})\bigr)\ar[dr]_{\widehat{\sigma}}\ar[rr]^{\cong}& &
H_3\bigl(\bar{B}_*(\bar{G},\bar{P})\bigr)\cong H_3(\bar{G},\bar{P};\Z)\ar[ld]^{\Psi} \\
&\EBg{\C}&
}
\end{equation*}
where $\Psi$ is the homomorphism \eqref{eq:Z.Psi} given in \cite[Theorem~3.17]{Zickert:VCSIR},
$\widehat{\sigma}$ is the homomorphism given in \eqref{eq:tilde.sigma} and the
horizontal arrow is given by the isomorphism of Corollary~\ref{cor:iso.complex}.
\end{proposition}

\begin{proof}
The definition of $\Psi$ given by formula \cite[(3.6)]{Zickert:VCSIR}, by \eqref{eq:Psi=sigma} coincides
with the definition of $\widehat{\sigma}$ given by \eqref{eq:flattening} via
the isomorphism of Corollary~\ref{cor:iso.complex}.
\end{proof}

\begin{remark}\label{rem:split}
In \cite[Proposition~14.3]{Neumann:EBGCCSC} Neumann proves that the long exact
sequence \eqref{eq:long.es} gives rise to a split exact sequence
\begin{equation}\label{eq:split.es}
 0\to
H_3(\bar{G};\Z)\xrightarrow{i_*}H_3(\bar{G},\bar{P};\Z)\xrightarrow{\partial}
H_2(\bar{P};\Z)\to0.
\end{equation}
In \cite[Proposition~6.12]{Zickert:VCSIR} Zickert proves that $\Psi$ defines a
splitting of the sequence \eqref{eq:split.es}. This together with
Proposition~\ref{prop:Psi=tildesigma} proves that $\widetilde{\sigma}$ defines a splitting of the
sequence \eqref{eq:split.es}.
\end{remark}

\subsection{Zickert's relative class}\label{ssec:fun.class}

Now we compare the construction of Zicker's relative class in
$H_3(\bar{G},\bar{P};\Z)$ with our computation of the invariant $\beta_P(M)$
given in Subsection~\ref{ssec:beta.P}.

As in Section~\ref{sec:compute}, consider an hyperbolic $3$--manifold of finite
volume $M$ with an ideal triangulation and let $\bar{\rho}\co\hg{M}\to\PSL{\C}$ be the geometric
representation. Let $\widehat{\pi}\co\widehat{\hyp}\to\widehat{M}$ be the extension
of the universal cover of $M$ to its end-compactification.
A \emph{developing map} of $\bar{\rho}$ is a $\bar{\rho}$--equivariant map
\begin{equation*}
D\co\widehat{\hyp}\to\hypc
\end{equation*}
sending the points in $\mathcal{C}$ to $\partial\hypc$. Let
$\widehat{c}\in\widehat{\mathcal{C}}$ and for each lift $c\in\mathcal{C}$ of
$\widehat{c}$, let $H(D(c))$ be an Euclidean horosphere based at $D(c)$. The
collection $\{H(D(c))\}_{c\in\widehat{\pi}^{-1}(\widehat{c})}$ of Euclidean
horospheres is called a \emph{decoration} of $\widehat{c}$ if the following
equivariance condition is satisfied:
\begin{equation*}
H(D(\gamma\cdot c))=\bar{\rho}(\gamma)H(D(c)),\quad\text{for $\gamma\in\hg{M}$,
$c\in\widehat{\pi}^{-1}(\widehat{c})$.}
\end{equation*}
A developing map of $\bar{\rho}$ together with a choice of decoration of each
$\widehat{c}\in\widehat{\mathcal{C}}$ is called a \emph{decoration} of
$\bar{\rho}$.

By \cite[Corollary~5.16]{Zickert:VCSIR} a decoration of $\bar{\rho}$ defines a
class $F(M)$ in $H_3(\bar{G},\bar{P};\Z)$. This can be seen
as follows. The decoration of $\bar{\rho}$
endows each $3$--simplex of $M$ with the shape of a decorated simplex. By
\cite[Theorem~3.13]{Zickert:VCSIR} each congruence class of these decorated
simplices corresponds to a generator of $\bar{B}_3(\bar{G},\bar{P})$ which is a
truncated simplex with a $\bar{G}$--edge labeling. The decoration and the
$\bar{G}$--edge labelings respect the face pairings so this gives a well-defined
cycle $\alpha$ in $\bar{B}_3(\bar{G},\bar{P})$ which represents the class $F(M)$, see \cite[p.~518]{Zickert:VCSIR}
for details. The class $F(M)$ is not well-defined, it depends on the choice of decoration, see \cite[Remark~5.19]{Zickert:VCSIR}.

In Subsection~\ref{ssec:beta.P} the inclusion $\widehat{\hyp}\to\hypc$ is a
developing  map of the geometric representation
$\bar{\rho}\co\hg{M}\to\PSL{\C}$.
Using the bijection between the set of horospheres and $\Xspb$ given in
\eqref{eq:horo.XP} there is a one-to-one correspondence between \emph{decorations} of $\bar{\rho}$ and
the $\dsubgr$-equivariant maps $\Phi$ given in \eqref{eq:Phi}. So $\Phi$ endows each ideal $3$--simplex of $\widehat{\hyp}$
with the shape of a decorated simplex, via the homomorphism
\begin{align*}
C_3(\widehat{\hyp})&\to\CC[3]{C}^{h_{\bar{P}}^{\bar{B}}\neq}(\Xsubgr{\bar{P}})\\
(z_0,z_1,z_2,z_3)&\mapsto (\Phi(z_0),\Phi(z_1),\Phi(z_2),\Phi(z_3)).
\end{align*}
where $C_*(\widehat{\hyp})$ is the simplicial chain complex of $\widehat{\hyp}$ of the ideal triangulation of $\widehat{\hyp}$
given by the lifting of the ideal triangulation of $M$ (see \S\ref{sec:compute}). This homomorphism induces a homomorphism
\begin{align*}
C_3(\widehat{M})=C_3(\widehat{\hyp})_\dsubgr&\to\CC[3]{B}^{h_{\bar{P}}^{\bar{B}}\neq}(\Xsubgr{\bar{P}})\cong\bar{B}_3(\bar{G},\bar{P})\\
(z_0,z_1,z_2,z_3)_\dsubgr&\mapsto (\Phi(z_0),\Phi(z_1),\Phi(z_2),\Phi(z_3))_{\bar{G}},
\end{align*}
which in turn, induces a homomorphism in homology
\begin{equation}
 \widehat{\psi}_\Phi\colon H_3(\widehat{M})\to H_3(\PSL{\C},\bar{P};\Z).
\end{equation}
The image of the fundamental class $[\widehat{M}]$ of $\widehat{M}$ under this homomorphism gives Zickert's class
\begin{align*}
 %\widehat{\psi}_\Phi\colon H_3(\widehat{M})&\to H_3(\PSL{\C},\bar{P};\Z)\\
[\widehat{M}]=\Bigl[\sum_{i=1}^n(z_0^i,z_1^i,z_2^i,z_3^i)_\dsubgr\Bigr]&\mapsto
F(M)=\Bigl[\sum_{i=1}^n\bigl(\Phi(z_0^i),\Phi(z_1^i),\Phi(z_2^i),
\Phi(z_3^i)\bigr)_{\bar{G}}\Bigr].
\end{align*}
corresponding to the decoration $\Phi$; it is the class $F_\Phi(M)$ described in Remark~\ref{rem:Fphi}. Thus we proved the following proposition.

\begin{proposition}
 Zickerts class $F(M)$ is the class $F_\Phi(M)$ described in Remark~\ref{rem:Fphi}, and $\Phi$ corresponds to the decoration used to define $F(M).$
\end{proposition}

\begin{corollary}\label{cor:F=betaP}
The image of the class $F(M)$ under the canonical homomorphism $\lambda_3\colon H_3(\bar{G},\bar{P};\Z)\to H_3([\bar{G}:\bar{P}];\Z)$ given in \eqref{eq:Taka.Hoch},
is the invariant $\beta_P(M)$. Hence, it does not depend on the choice of decoration.
\end{corollary}

\begin{proof}
 The class $F_\Phi(M)\in H_3(\PSL{\C},\bar{P};\Z)$ is mapped by the canonical map \eqref{eq:Taka.Hoch} to the invariant $\beta_P(M)$ (Remark~\ref{rem:Fphi}) which is independent of the choice of decoration by Theorem~\ref{thm:beta.P.B} (see also Remark~\ref{rem:Phi.ind}).
\end{proof}

\begin{remark}\label{rem:F.preim.b}
If we vary the decoration $\Phi$ all the corresponding classes $F_\Phi(M)$ are in $\lambda_3^{-1}(\beta_P(M))$.
\end{remark}

\subsection{$PSL$--fundamental class for non-compact $M$}

As we mentioned above, the class $F_\Phi(M)$ is not well-defined, it depends on the choice of decoration $\Phi$, however, Zickert proves that the image  of $F_\Phi(M)$ under the homomorphism \eqref{eq:Z.Psi} is independent of such choice \cite[Theorem~6.10]{Zickert:VCSIR}. 
We shall denote such image by $\fcPSL{M}$ and call it the \emph{$PSL$--fundamental class} (in \cite[Theorem~14.2]{Neumann:EBGCCSC} it is denoted $\widehat{\beta}(M)$).

Recall that to construct Zickert's relative class $F_\Phi(M)$ it is necessary a triangulation of $M$. One can define the $PSL$--fundamental class $\fcPSL{M}$ without a triangulation of $M$ 
using the invariant $\beta_P(M)$. By Remark~\ref{rem:F.preim.b} the preimage $\lambda_3^{-1}(\beta_P(M))$ under the canonical map \eqref{eq:Taka.Hoch} consists of the classes $F_\Phi(M)$ 
varying the decoration $\Phi$, so choosing a preimage $F_\phi(M)\in\lambda_3^{-1}(\beta_P(M))$ we can define the $PSL$--fundamental class by
 \begin{equation}\label{eq:PSL.fc}
 \fcPSL{M}=\widehat{\sigma}(F_\Phi(M))\in \EBg{\C}\cong H_3(\PSL{\C};\Z).
\end{equation}
Now, using a triangulation of $M$ but the construction of the class $F_\Phi(M)$ using the $h_{\bar{P}}^{\bar{B}}$-subcomplex $\CC[n]{C}^{h_{\bar{P}}^{\bar{B}}\neq}(\Xsubgr{\bar{P}})$ to compute $H_3(\PSL{\C},\bar{P};\Z)$ 
as in Remark~\ref{rem:Fphi}, we can write down a explicit formula for $\fcPSL{M}$ applying to the cycle \eqref{eq:F.cycle} the homomorphism \eqref{eq:tilde.sigma} to get
\begin{multline}\label{eq:fun.class}
 \fcPSL{M}=\widehat{\sigma}(F_\Phi(M))\\
=\sum_{i=1}^n{\scriptstyle\Bigl(\widetilde{w}_0\bigl(\Phi(z_0^i),\Phi(z_1^i),
\Phi(z_2^i),\Phi(z_3^i)\bigr),
\widetilde{w}_1\bigl(\Phi(z_0^i),\Phi(z_1^i),\Phi(z_2^i),\Phi(z_3^i)\bigr),
\widetilde{w}_2\bigl(\Phi(z_0^i),\Phi(z_1^i),\Phi(z_2^i),\Phi(z_3^i)\bigr)\Bigr)
}.
\end{multline}
In view of Remark~\ref{rem:efficient}, such formula is very difficult to write using Zickert's definition of the class $F_\Phi(M)$ using the complex of truncated simplices $\bar{C}_*(\bar{G},\bar{P})$.

\begin{remark}\label{rem:T.inv}
Notice that the image of the $PSL$--fundamental class $\fcPSL{M}$ under the
homomorphism $(h_{\bar{I}}^{\bar{T}})_*\co H_3(\PSL{\C};\Z)\to
H_3([\PSL{\C}:\bar{T}];\Z)$ in diagram \eqref{eq:hom.diag} is also an invariant
of $M$ which we can denote by $\beta_T(M)$. This invariant is sent to the
classical Bloch invariant $\beta_B(M)$ by the homomorphism
\begin{equation*}
 H_3([\PSL{\C}:\bar{T}];\Z)\xrightarrow{h_{\bar{T}}^{\bar{B}}} H_3([\PSL{\C}:\bar{B}];\Z)\xrightarrow{\sigma}\pB{\C}.
\end{equation*}
It would be interesting to see which information carries this invariant. 
\end{remark}

\section{$(G,H)$--representations}\label{sec:G.H.rep}

Our construction also works in the general context of
$(G,H)$--representations of tame manifolds considered in \cite{Zickert:VCSIR}.
Here we give the basic definitions and facts, for more detail see
\cite[\S4]{Zickert:VCSIR}.

A \emph{tame} $n$-manifold is an $n$-manifold $M$ diffeomorphic to the interior of a
compact manifold $\overline{M}$. The boundary components $E_i$ of $\overline{M}$
are called the \emph{ends} of $M$. The number of ends can be zero to include
closed manifolds as tame manifolds with no ends.

Let $M$ be a tame $n$-manifold. We have that $\hg{M}\cong\hg{\overline{M}}$ and each
end $E_i$ of $M$ defines a subgroup $\hg{E_i}$ of $\hg{M}$ which is well defined
up to conjugation. These subgroups are called \emph{peripheral subgroups} of
$M$. 

Let $\widehat{M}$ be the compactification of $M$ obtained by identifying each
end of $M$ to a point. 
We call the points in $\widehat{M}$ corresponding to the ends as \emph{ideal
points} of $M$ . 
Let $\widehat{\widetilde{M}}$ be the compactification of the universal cover
$\widetilde{M}$ of $M$ obtained by adding ideal points corresponding to the
lifts of the ideal points of $M$ . The covering map extends to a map from
$\widehat{\widetilde{M}}$ to $\widehat{M}$. We choose a point in $M$ as a base
point of  $\widehat{M}$ and one of its lifts as base point of
$\widehat{\widetilde{M}}$. With the base points fixed, the action of $\hg{M}$ on
$\widetilde{M}$ by covering transformations extends to an action on
$\widehat{\widetilde{M}}$ which is not longer free. The stabilizer of a lift
$\tilde{e}$ of an ideal point $e$ corresponding to an end $E_i$ is isomorphic
to
a peripheral subgroup $\hg{E_i}$. Changing the lift $\tilde{e}$ corresponds to
changing the peripheral subgroup by conjugation.

Let $G$ be a discrete group, let $H$ be any subgroup and consider the family of
subgroups $\fami(H)$ generated by $H$. Let $M$ be a tame manifold, a
representation $\rho\co\hg{M}\to G$ is called a \emph{$(G,H)$--representation}
if the images of the peripheral subgroups under $\rho$ are in $\fami(H)$. 

In the particular case when $G=\PSL{\C}$ and $H=\bar{P}$ a
$(G,H)$--representation $\rho\co\hg{M}\to\PSL{\C}$ is called
\emph{boundary-parabolic}.

The geometric representation of a hyperbolic $3$--manifold is boundary
parabolic. For further examples see Zickert \cite[\S4]{Zickert:VCSIR}.

Let $M$ be a tame $n$--manifold with $d$ ends and let $\rho\co\hg{M}\to G$ be a
$(G,H)$--representation.
Let $\dsubgr$ be the image of $\hg{M}$ in $G$ under $\rho$, also denote by
$\dsubgr_i$ the image of the peripheral subgroup $\hg{E_i}$ under $\rho$ and
consider the family 
$\fami=\fami(\dsubgr_1,\dots,\dsubgr_d)$ of subgroups of $G$. On the other hand,
define $\dsubgr'_i=\rho^{-1}(\dsubgr_i)$ and consider the family
$\fami'=\fami'(\dsubgr'_1,\dots,\dsubgr'_d)$ of
subgroups of $\hg{M}$.

\begin{proposition}
Consider the classifying space $\EGF{\dsubgr}$ as a $\hg{M}$--space defining the
action by
\begin{equation*}
 \gamma\cdot x=\rho(\gamma)\cdot x,\qquad \gamma\in\hg{M},\quad
x\in\EGF{\dsubgr}.
\end{equation*}
Then, with this action $\EGF{\dsubgr}$ is a model for the classifying space
$\EGF[\fami']{\hg{M}}$.
\end{proposition}

\begin{proof}
Consider the $\dsubgr$--set $\Delta_\fami$ defined in
Subsection~\ref{ssec:Adamson}. It is enough to see that $\Delta_\fami$ seen as
a $\hg{M}$--set using $\rho$ is $\hg{M}$--isomorphic to  $\Delta_{\fami'}$. By
the definition of $\dsubgr'$ and $\dsubgr'_i$ we have that
$\dsubgr'/\ker\rho\cong\dsubgr$ and $\dsubgr'_i/\ker\rho\cong\dsubgr_i$.
Then
\begin{equation*}
 \dsubgr'/\dsubgr'_i\cong (\dsubgr'/\ker\rho)\bigm/ (\dsubgr'_i/\ker\rho)\cong
\dsubgr/\dsubgr_i.
\end{equation*}
Therefore
\begin{equation*}
\Delta_{\fami'}=\coprod_{i=1}^d \dsubgr'/\dsubgr'_i \cong \coprod_{i=1}^d
\dsubgr/\dsubgr_i=\Delta_\fami.
\end{equation*}
So now we can use $\Delta_{\fami'}$ in the simplicial construction of
$\EGF[\fami']{\hg{M}}$ and we obtain precisely $\EGF{\dsubgr}$.
\end{proof}

Since the action of $\hg{M}$ on $\widehat{\widetilde{M}}$ has as isotropy
subgroups the peripheral subgroups $\hg{E_i}$ and $\hg{E_i}\in\fami'$ there is a
$\hg{M}$--map unique up to $\hg{M}$--homotopy
\begin{equation}\label{eq:psi.famip}
 \psi_{\fami'}\co \widehat{\widetilde{M}}\to \EGF[\fami']{\hg{M}}\cong
\EGF{\dsubgr}.
\end{equation}
Now consider the classifying space $\EGF[\fami(H)]{G}$, by
Proposition~\ref{prop:res}, restricting the action of $G$ to $\dsubgr$ we have
that
$\mathrm{res}^G_\dsubgr\EGF[\fami(H)]{G}\cong\EGF[\fami(H)/\dsubgr]{\dsubgr}$. 
On the other hand, we have that $\fami\subset\fami(H)/\dsubgr$, so we have a
$\dsubgr$--map unique up to $\dsubgr$--homotopy
\begin{equation}\label{eq:psi.fami}
 \psi_{\fami}\co \EGF{\dsubgr}\to \EGF[\fami(H)]{G}.
\end{equation}
Composing \eqref{eq:psi.famip} with \eqref{eq:psi.fami} we obtain a
$\rho$--equivariant map unique up to $\rho$--homotopy
\begin{equation}\label{eq:psi.rho}
 \psi_\rho\co \widehat{\widetilde{M}}\to \EGF[\fami(H)]{G}.
\end{equation}
Taking the quotients by the actions of $\hg{M}$ and $G$ we get a map unique up
to homotopy given by the composition
\begin{equation*}
\widehat{\psi}_\rho\co\widehat{M}\to \EGF[\fami(H)]{G}/\dsubgr\to
\BGF[\fami(H)]{G}.
\end{equation*}
By a computation analogous to \eqref{eq:hM.fc} $\widehat{M}$ has a fundamental class $[\widehat{M}]$ in $H_n(\widehat{M};\Z)$.
Denote by $\beta_H(\rho)$ the image of the fundamental class $[\widehat{M}]$
of $\widehat{M}$ under the map induced in homology by $\widehat{\psi}_\rho$ 
\begin{equation}\label{eq:psi.rho.homo}
 \begin{aligned}
 H_n(\widehat{M};\Z)&\xrightarrow{(\widehat{\psi}_\rho)_*}
H_n(\BGF[\fami(H)]{G};\Z)\\
[\widehat{M}] &\mapsto \beta_H(\rho).
\end{aligned}
\end{equation}
Thus, by Proposition~\ref{prop:CS.HRGH} we have:

\begin{theorem}
Given an oriented tame $n$--manifold with a $(G,H)$--representation
$\rho\co\hg{M}\to G$ we have a well-defined invariant
\begin{equation*}
 \beta_H(\rho)\in H_n([G:H];\Z).
\end{equation*}
\end{theorem}

As before, one can compute the class $\beta_H(\rho)$ using a triangulation of
$M$.
A \emph{triangulation} of a tame manifold $M$ is an identification of
$\widehat{M}$ with a complex obtained by gluing together simplices with
simplicial attaching maps. A triangulation of $M$ always exists and it lifts
uniquely to a triangulation of $\widehat{\widetilde{M}}$.

\subsection{Zickert's clases for $(G,H)$-representations}

Let $M$ be a tame $n$--manifold with $d$ ends and let $\rho\co\hg{M}\to G$ be a
$(G,H)$--representation. In \cite[\S5.2]{Zickert:VCSIR}, given a triangulation
of $M$ it is constructed a $(G,H)$--cocycle, see  \cite[\S5.2]{Zickert:VCSIR}
for the definition, which defines a class $F(\rho)$ in
$H_n(G,H;\Z)$ (see \cite[Theorem~5.13]{Zickert:VCSIR}).
The construction of the $(G,H)$--cocycle depends on a \emph{decoration} of
$\rho$ by \emph{conjugation elements}. Such decorations are given as follows:
for each ideal point $e_i\in\widehat{M}$ choose a lifting
$\tilde{e_i}\in\widehat{\widetilde{M}}$ and assign to this lifting  an element
$g_i(\tilde{e}_i)\in G$, or rather an $H$--coset $g_i(\tilde{e}_i)H$, then
extend $\rho$--equivariantly by
\begin{equation*}
 g_i(\gamma\cdot\tilde{e}_i)=\rho(\gamma)g_i(\tilde{e}_i),\qquad
\gamma\in\hg{M}.
\end{equation*}
Let $\mathcal{I}$ denote the set of ideal point in $\widehat{\widetilde{M}}$.
Notice that a decoration by conjugation elements is equivalent to give a
$\rho$--equivariant map
\begin{equation*}
 \Phi_\rho\co\mathcal{I}\to G/H.
\end{equation*}
The map $\Phi_\rho$ defines explicitly the $\rho$--map \eqref{eq:psi.rho} and
using the triangulation of $M$ gives also explicitly the homomorphism
\eqref{eq:psi.rho.homo} as in Subsection~\ref{ssec:beta.P}.

\begin{remark}\label{rem:F.b.GH}
Given a group $G$ and a subgroup $H$, in general $H_n(G,H;\Z)$ does not coincides
with $H_n([G:H];\Z)$, see Subsection~\ref{ssec:Taka.Hoch}.
The construction of the $(G,H)$--cocycle depends on the choice
of decoration of $\rho$ by conjugation elements, so in principle, choosing
different decorations one can obtain different classes in
$H_n(G,H;\Z)$, in that case, all this classes are mapped to
$\beta_H(\rho)\in H_n([G:H];\Z)$ under the canonical homomorphism
\eqref{eq:Taka.Hoch} since $\beta_H(\rho)$ does not depend on the choice of
decoration because the $\rho$--map \eqref{eq:psi.rho} given by the decoration is
unique up to $\rho$--homotopy. So in general it is more appropriate
to use Adamson relative group homology than Takasu relative group homology
because we obtain invariants independent of choice.
\end{remark}

\subsection{Boundary-parabolic representations}\label{ssec:boun-par}

In the case of boundary-parabolic representations of tame $3$--manifolds we can
use a developing map with a decoration to compute $\beta_P(\rho)$. Let $M$ be a
tame $3$--manifold and let $\bar{\rho}$ be a boundary-parabolic representation.
A \emph{developing map} of $\rho$ is a $\rho$--equivariant map
$D_{\rho}\co \widehat{\widetilde{M}}\to \hypc$ 
sending the ideal points of $\widehat{\widetilde{M}}$ to
$\partial\hypc=\widehat{\C}$. Taking a sufficiently fine triangulation of $M$ it
is always possible to construct a developing map of $\rho$
\cite[Theorem~4.11]{Zickert:VCSIR}. Let $\mathcal{C}$ be the image under
$D_\rho$ of the set of ideal points $\mathcal{I}$ of $\widehat{\widetilde{M}}$.
A \emph{decoration} of $\rho$ is a $\rho$--equivariant map
$\Phi_\rho\co\mathcal{C}\to\Xsubgr{\bar{P}}$ which can be obtained assigning a
Euclidean horosphere to each element of $\mathcal{C}$ as in
Subsection~\ref{ssec:fun.class} or as in Remark~\ref{rem:ci.vi.Ai}. Again, the
decoration defines explicitly the $\rho$--map \eqref{eq:psi.rho} which gives
explicitly the homomorphism \eqref{eq:psi.rho.homo} with $G=\PSL{\C}$ and
$H=\bar{P}$.

As in the case of the geometric representation Zickert defines the
$PSL$--fundamen\-tal class $\fcPSL{\rho}\in H_3(\PSL{\C};\Z)$ of $\rho$ as the image of the class $F(\rho)$
under the homomorphism \eqref{eq:tilde.sigma}, since by \cite[Theorem~6.10]{Zickert:VCSIR} the image is independent of 
the choice of decoration. As in \eqref{eq:PSL.fc} one can define the class $\fcPSL{\rho}$ without using a triangulation of
$M$, by Remark~\ref{rem:F.b.GH} we can choose a preimage $F(\rho)\in\lambda_3^{-1}(\beta_P(\rho))$ and take
\begin{equation*}
\fcPSL{\rho}=\widehat{\sigma}(F(\rho))\in \EBg{\C}\cong H_3(\PSL{\C};\Z).
\end{equation*}
Using the construction of the class $F(\rho)$ using the $h_{\bar{P}}^{\bar{B}}$-subcomplex $\CC[n]{C}^{h_{\bar{P}}^{\bar{B}}\neq}(\Xsubgr{\bar{P}})$ to compute $H_3(\PSL{\C},\bar{P};\Z)$ 
as in Remark~\ref{rem:Fphi}, we can write down a explicit formula for $\fcPSL{\rho}$ applying to $F(\rho)$ homomorphism \eqref{eq:tilde.sigma}.

The image of $\beta_P(\rho)$ under $(h_{\bar{P}}^{\bar{B}})_*\co
H_3([\PSL{\C}:\bar{P};\Z)\to H_3([\PSL{\C}:\bar{B}];\Z)$ gives an
invariant $\beta_B(\rho)$ which can be computed using a developing map as in
Subsection~\ref{ssec:bloch.inv}.

Also as in Remarks~\ref{rem:T.inv} the image of the $PSL$--fundamental class
$\fcPSL{\rho}$ under the homomorphism $(h_{\bar{I}}^{\bar{T}})_*\co
H_3(\PSL{\C};\Z)\to H_3([\PSL{\C}:\bar{T}];\Z)$ in diagram \eqref{eq:hom.diag}
is also an invariant $\beta_T(\rho)$ of $\rho$.

\section{Complex volume}\label{sec:volume}

Recall that Rogers dilogarithm is given by
\begin{equation*}
 L(z)=-\int_0^z\frac{\Log(1-t)}{t}dt + \frac{1}{2}\Log(z)\Log(1-z).
\end{equation*}
In \cite[Proposition~2.5]{Neumann:EBGCCSC} Neumann defines the homomorphism
\begin{gather}
\widehat{L}\co\EpB{\C}\to\C/\pi^2\Z,\notag\\
[z;p,q]\mapsto L(z)+\frac{\pi i}{2}\bigl(q\Log
z+p\Log(1-z)\bigr)-\frac{\pi^2}{6}\label{eq:hatL}
\end{gather}
where $[z;p,q]$ denotes elements in the extended pre-Bloch group using our
choice of logarithm
branch, see Remark~\ref{rem:zpq}, but \eqref{eq:hatL} is actually independent of
this choice, see \cite[Remark~1.9]{Zickert:VCSIR}. In \cite[Theorem~2.6 or
Theorem~12.1]{Neumann:EBGCCSC} Neumann proves that
under the isomorphism $H_3(\PSL{\C};\Z)\cong\EBg{\C}$ the homomorphism
$\widehat{L}$ corresponds to the Cheeger--Chern--Simons class.

For the geometric
representation of a complete oriented hyperbolic $3$--manifold of finite volume
$M$ we have \cite[Corollary~14.6]{Neumann:EBGCCSC}
\begin{equation}\label{eq:comp.vol}
 \widehat{L}\bigl(\fcPSL{M}\bigr)=i\bigl(\Vol(M)+i\CS(M)\bigr)\in\C/i\pi^2\Z,
\end{equation}
where $\Vol(M)$ is the volume of $M$ and
$\CS(M)=2\pi^2\mathrm{cs}(M)\in\R/\pi^2\Z$ with $\mathrm{cs}(M)$ the
Chern--Simons invariant of $M$. Usually $\Vol(M)+i\CS(M)$ is called the
\emph{complex volume} of $M$, see Neumann \cite{Neumann:EBGCCSC} and Zickert
\cite{Zickert:VCSIR} for details. So, using formula \eqref{eq:fun.class} for $\fcPSL{M}$ in \eqref{eq:comp.vol} we get a formula to
compute the volume and Chern--Simons invariant of a complete oriented hyperbolic $3$-manifold of finite volume which by Remark~\ref{rem:efficient} 
it is more efficient than Zickert's formula in \cite{Zickert:VCSIR}.

For the case of a boundary-parabolic representation $\rho\co\hg{M}\to\PSL{\C}$
of a tame $3$--manifold $M$, Zickert \cite[\S6]{Zickert:VCSIR}
defined the \emph{complex volume} of the representation $\rho$ by
\begin{equation*}
 i\bigl(\Vol(\rho)+i\CS(\rho)\bigr)=\widehat{L}\bigl(\fcPSL{\rho}\bigr).
\end{equation*}
Applying $\widehat{L}$ to a formula for $\fcPSL{\rho}$ analogous to
\eqref{eq:fun.class}, but using the decoration
$\Phi_\rho\co\mathcal{C}\to\Xsubgr{\bar{P}}$ of $\rho$ of
Subsection~\ref{ssec:boun-par}, we obtain an explicit efficient formula for the complex volume of $\rho$.

\subsection*{Question}

The main disadvantage of invariant $\beta_P(M)\in H_3([\PSL{\C}:\bar{P}];\Z)$ is that, so far, we cannot get directly from it
the volume and the Chern-Simons invariant of $M$, so one can ask if there exists a homomorphism
\begin{equation*}
 H_3([\PSL{\C}:\bar{P}];\Z)\to H_3(\PSL{\C};\Z)\cong\EBg{\C},
\end{equation*}
such that the image of $\beta_P(M)$ is $\fcPSL{M}$.

\subsection*{Acknowledgments}

Both authors would like to thank Ramadas Ramakrishnan for many fruitful
discussions and suggestions about this work and for his hospitality during the
authors visits to the Abdus Salam International Centre for Theoretical Physics,
Trieste, Italy. Both authors also thank Francisco Gonzales Acuña for
his interest in this work and many illustrative talks.

The second author thanks Walter Neumann for many useful discussions and for his
kind hospitality during his visit to Columbia University as part of his
sabbatical year in 2012. He also thanks Christian Zickert for the discussions about his
work, his interest in this work and for pointing out to me a mistake in a previous version of this article. He also would like to thank 
Wolfgang L\"uck and Haydée Aguilar for the interesting conversations related to
this work.

The first author was partially supported by CONACYT-``Becas-Mixtas en el extranjero
para becarios CONACyT nacionales 2011-2012'', Mexico. The second author is a
Regular Associate of the Abdus Salam ICTP. He was partially supported by CONACYT
and UNAM-DGAPA-PASPA, Mexico during his sabbatical visits and he is currently supported by
project CONACYT 253506.

\end{document}